\documentclass[11pt, a4paper, reqno, twoside]{amsart} 

\usepackage{geometry}
\usepackage{color}
\usepackage{verbatim}
\usepackage[utf8x]{inputenc}
\usepackage{t1enc}
\usepackage[english]{babel}
\usepackage{soul}

\usepackage{amsmath,amssymb,amsthm}
\usepackage{mathrsfs,esint}
\usepackage{graphicx,wrapfig}
\usepackage[format=hang]{caption}
\usepackage{enumitem}			
\usepackage{soul}

\newcommand*{\R}{{\mathbb R}}

\newcommand*{\N}{{\mathbb N}}
\newcommand*{\Z}{{\mathbb Z}}
\newcommand*{\loc}{{\mathrm{loc}}}

\newcommand*{\eps}{\varepsilon}

\newcommand*{\Om}{\Omega}

\newcommand{\dd}{\text{d}}
\newcommand{\abs}[1]							
{\left| #1 \right|}
\newcommand{\smallabs}[1]						
{\lvert #1 \rvert}	
\newcommand{\norm}[1]							
{\left\lVert #1 \right\rVert}	
\newcommand{\smallnorm}[1]						
{\lVert #1 \rVert}						
\newcommand{\ip}[2]								
{\left< #1 , #2 \right>}
\DeclareMathOperator{\vol}{vol}					
\DeclareMathOperator{\id}{id}					
\DeclareMathOperator{\tr}{tr}					
\DeclareMathOperator*{\esslim}{ess\,lim}

\renewcommand{\S}{\mathbb{S}}					
\newcommand{\B}{\mathbb{B}}
\renewcommand{\H}{\mathbb{H}}					

\newcommand{\cF}{\mathcal{F}}
\newcommand{\cI}{\mathcal{I}}
\newcommand{\cL}{\mathcal{L}}
\newcommand{\cQ}{\mathcal{Q}}
\newcommand{\cR}{\mathcal{R}}
\newcommand{\cS}{\mathcal{S}}

\providecommand*{\vint}[1]{\mathchoice
          {\mathop{\vrule width 5pt height 3 pt depth -2.5pt
                  \kern -9pt \kern 1pt\intop}\nolimits_{\kern -5pt{#1}}}
          {\mathop{\vrule width 5pt height 3 pt depth -2.6pt
                  \kern -6pt \intop}\nolimits_{\kern -3pt{#1}}}
          {\mathop{\vrule width 5pt height 3 pt depth -2.6pt
                  \kern -6pt \intop}\nolimits_{\kern -3pt{#1}}}
          {\mathop{\vrule width 5pt height 3 pt depth -2.6pt
                  \kern -6pt \intop}\nolimits_{\kern -3pt{#1}}}}
\newcommand*{\jint}{\fint}
\DeclareMathOperator{\lip}{lip}

\DeclareMathOperator{\dist}{dist}
\DeclareMathOperator{\diam}{diam}
\DeclareMathOperator{\rad}{rad}

\DeclareMathOperator{\Mod}{Mod}
\DeclareMathOperator{\pcap}{cap}
\DeclareMathOperator{\len}{len}
\DeclareMathOperator*{\supp}{supp}	

\numberwithin{equation}{section}
\theoremstyle{plain}
\newtheorem{theorem}[equation]{Theorem}
\newtheorem{prop}[equation]{Proposition}
\newtheorem{cor}[equation]{Corollary}
\newtheorem{lemma}[equation]{Lemma}

\theoremstyle{definition}

\newtheorem{defn}[equation]{Definition}

\newtheorem{remark}[equation]{Remark}

\newtheorem{example}[equation]{Example}

\newtheorem*{stass}{Standing Assumptions}
\begin{document}

\title[On homogeneous Newton-Sobolev spaces]
{On homogeneous Newton-Sobolev spaces of functions in metric measure spaces of uniformly locally controlled geometry} 
\author{Ryan Gibara}
\address{Department of Mathematics, Physics and Geology, Cape Breton University, Sydney, NS~B1Y3V3, Canada}
\email{ryan\textunderscore gibara@cbu.ca}
\author{Ilmari Kangasniemi}
\address{Department of Mathematical Sciences, P.O.~Box 210025, University of Cincinnati, Cincinnati, OH~45221-0025, U.S.A.}
\email{kangaski@ucmail.uc.edu}
\author{Nageswari Shanmugalingam}
\address{Department of Mathematical Sciences, P.O.~Box 210025, University of Cincinnati, Cincinnati, OH~45221-0025, U.S.A.}
\email{shanmun@uc.edu}
\thanks{
 N.S.'s work is partially supported by the NSF (U.S.A.) grant DMS~\#2054960. I.K.'s work is partially 
 supported by the NSF (U.S.A.) grant DMS~\#2247469.}

\begin{abstract}
We study the large-scale behavior of Newton-Sobolev functions on
complete, connected, proper, separable metric measure spaces equipped with a Borel measure $\mu$ with 
$\mu(X) = \infty$ and $0 < \mu(B(x, r)) < \infty$ for all $x \in X$ and $r \in (0, \infty)$.
Our objective is to understand 
the relationship between the Dirichlet space $D^{1,p}(X)$, defined using upper gradients, 
and the Newton-Sobolev space $N^{1,p}(X)+\R$, for $1\le p<\infty$.
We show that 
when $X$ is of uniformly locally $p$-controlled geometry,
these two spaces do not coincide under a wide variety of geometric and
potential theoretic conditions. We also show that when the metric measure space is the standard hyperbolic space $\mathbb{H}^n$
with $n\ge 2$, these two spaces coincide precisely when $1\le p\le n-1$. We also provide additional characterizations of when
a function in $D^{1,p}(X)$ is in $N^{1,p}(X)+\R$ in the case that the two spaces do not coincide.
\end{abstract}

\maketitle

\noindent
    {\small \emph{Key words and phrases}: Upper gradients, Sobolev spaces, Dirichlet spaces, uniformly locally controlled geometry,
    	$p$-parabolic ends, $p$-hyperbolic ends, Gromov hyperbolicity, uniformization, standard hyperbolic spaces.
}

\medskip

\noindent
    {\small Mathematics Subject Classification (2020):
Primary: 46E36.
Secondary: 31E05, 30L15, 53C23.
}

\tableofcontents

\section{Introduction}

Throughout this paper, we will assume that $1\le p<\infty$, and that 
$(X,d,\mu)$ is a connected proper (i.e., closed and bounded subsets of $X$ are compact)
metric measure space with $\mu$ a Borel regular measure satisfying $\mu(X)=\infty$.

In the context of nonlinear potential theory on a complete metric space that is unbounded, 
boundary conditions manifest as sought-after behavior of solutions ``at infinity.''
There are many geometric notions of boundary ``at infinity'' of a complete unbounded metric space.
Of particular relevance in potential theory are the notions of ends at $\infty$, as in
Definition~\ref{def:ends}, and the metric boundary of (conformal) uniformizations of
the space, as in subsection~\ref{sub:hypunif}. 

In this paper, we consider metric measure spaces
 $(X,d,\mu)$ that are of \emph{uniformly locally $p$-controlled geometry}; that is, spaces which are
uniformly locally doubling and support
a uniformly local $p$-Poincar\'e inequality, as in Definition~\ref{def:controlled-geom} below.
A Riemannian manifold of non-negatively pinched curvature is necessarily of globally
$p$-controlled geometry, while negatively pinched Riemannian manifolds are of uniformly
locally $p$-controlled geometry but are not of globally $p$-controlled geometry. In this paper,
we consider non-smooth analogs of both cases. We pay particular attention to 
metric measure spaces of globally $p$-controlled
geometry as well as Gromov hyperbolic spaces of uniformly locally $p$-controlled geometry 
as these play special roles in potential theory and quasiconformal mappings in nonsmooth settings,
see for example~\cite{BHK}.

In this setting of unbounded spaces, a natural class of functions to consider as 
encoding the boundary data is the 
homogeneous analog of Newton-Sobolev spaces, the Dirichlet spaces $D^{1,p}(X)$ 
of measurable functions having $p$-integrable upper gradients. In particular, if it turns 
out that there are more functions in $D^{1,p}(X)$ than just functions in $N^{1,p}(X)+\R$, 
then we are dealing with a larger class of functions, in some cases resulting in a richer theory 
of Dirichlet problems for Dirichlet data ``at infinity''.
As such, the primary purpose of the present paper is to investigate 
under what conditions on the metric measure space $(X,d,\mu)$ we have that 
$D^{1,p}(X)$ and $N^{1,p}(X)+\R$ do not agree. The latter class $N^{1,p}(X)+\R$ is considered
to be a tamer class, as demonstrated by Corollary~\ref{cor:integrable-trace-zero}
and the discussion in Section~\ref{Sec:N1p-tame}.

The following are the principal results of this paper, and they summarize the results obtained throughout the paper.
In particular, these results show that the question at hand is not answered by considering the single property of $p$-parabolicity
versus $p$-hyperbolicity of ends of the space, nor solely the $p$-controlled geometry of the space.

\begin{theorem}\label{thm:main1}
Let $1\le p<\infty$ and let $(X,d,\mu)$ be a complete, connected, proper, separable metric measure space of 
uniformly locally $p$-controlled geometry
with $\mu(X)=\infty$ and $0 < \mu (B(x, r)) < \infty$ for all $x \in X, r \in (0, \infty)$. 
Then $D^{1,p}(X)\ne N^{1,p}(X)+\R$ if any one of the following conditions holds:
\begin{enumerate}[label=(\roman*)]
\item $X$ is of globally $p$-controlled geometry {\rm{(}}this is Corollary~\ref{cor:globalcontrol}{\rm{)}}.
\item $X$ is Gromov hyperbolic and roughly starlike, $\inf_{x\in X}\mu(B(x,1))>0$, 
$\partial_\eps X$ has at least two points, and $p\ge \beta/\eps$, where $\beta$ and
$\eps$ are positive real numbers such that the uniformized metric measure space $(X,d_\eps,\mu_\beta)$ is a uniform domain
of globally $p$-bounded geometry {\rm{(}}this is Theorem~\ref{prop:two-pts-large-p}{\rm{)}}.
\item $X$ is Gromov hyperbolic and roughly starlike, $\sup_{x\in X}\mu(B(x,1))<\infty$, 
and $\partial_\eps X$ has exactly one point, where $\eps$ is a positive real number such that the 
uniformized metric space $(X,d_\eps)$ is a uniform domain
{\rm{(}}this is Corollary~\ref{cor:one_bdry_point}{\rm{)}}.
\item $X$ is $p$-parabolic {\rm{(}}this is Theorem~\ref{thm:para}{\rm{)}}.
\item $X$ is a length space with at least two ends,
and  
$\inf_{x\in X}\mu(B(x,1))>0$ {\rm{(}}this is Theorem~\ref{thm:2ends}{\rm{)}}.
\item $X$ is a length space with at least one $p$-parabolic end,
and $\inf_{x\in X}\mu(B(x,1))>0$ {\rm{(}}this is Corollary~\ref{cor:parabolic_end}{\rm{)}}.
\end{enumerate}
One the other hand, the standard hyperbolic space $X=\mathbb{H}^n$, which is a Gromov hyperbolic space of locally $p$-controlled 
geometry with exactly one end, satisfies $D^{1,p}(X)=N^{1,p}(X)+\R$ if and only if $1\le p\le n-1$ {\rm{(}}this is 
Theorem~\ref{thm:Hn_classification}
{\rm{)}}. 
\end{theorem}

Hence, the main case left unresolved by the above theorem is when $X$ is Gromov hyperbolic and 
roughly starlike, $1 \le p \le \beta/\eps$, and $X$ has precisely one end which is $p$-hyperbolic. 
Moreover, the result on $\H^n$ shows that there exist examples with $D^{1,p}(X)=N^{1,p}(X)+\R$ in this 
category. 

Subsequent to writing the manuscript, a comment by Ilkka Holopainen led the authors to a literature 
reference for Theorem~\ref{thm:Hn_classification}, see~\cite[Theorem~5.8]{Strichartz}. The 
proof given in~\cite{Strichartz} and the proof given in the present paper have similar core ideas.
However, the two proofs differ significantly in their execution; in particular, we do not explicitly
use Poincar\'e inequalities on $\mathbb{S}^{n-1}$ while~\cite{Strichartz} does. For this reason, we chose to
retain our proof in this paper as our methods might be useful in other contexts.

In the setting where
$D^{1,p}(X)\ne N^{1,p}(X)+\R$, it is desirable to know how to identify functions in $D^{1,p}(X)$ that do
belong to $N^{1,p}(X)+\R$. This is the following second main result of this paper.
The first claim of Theorem~\ref{thm:main2} is established in Corollary~\ref{cor:integrable-trace-zero}, while 
the second claim follows from Theorem~\ref{thm:lp-f-at-infty}.

\begin{theorem}\label{thm:main2}
Let $1\le p<\infty$ and let $(X,d,\mu)$ be a complete, connected, proper, separable 
metric measure space of uniformly  locally $p$-controlled geometry with 
$\mu(X)=\infty$ and $0 < \mu (B(x, r)) < \infty$ for all $x \in X, r \in (0, \infty)$.
Suppose that $u\in D^{1,p}(X)$. 
\begin{enumerate}
\item[(i)] If $X$ is Gromov hyperbolic and roughly starlike, $\inf_{x\in 
X}\mu(B(x,1))>0$, 
$\int_X|u|^p\, d\mu$ is finite, and $\eps, \beta$ are positive real numbers such that the 
uniformized metric measure space $(X, d_\eps, \mu_\beta)$ is a uniform domain of globally
$p$-bounded geometry, then for each $\xi\in\partial_\eps X$ that is a
$\mu_\beta$-Lebesgue point of $u$, we have that 
\[
\lim_{r\to 0^+}\frac{1}{\mu_\beta(B_{d_\eps}(\xi,r))}\, \int_{B_{d_\eps}(\xi,r)}|u|\, \dd \mu_\beta=0.
\]
\item[(ii)] If $X$ supports a global $(p,p)$-Sobolev inequality, then
$u$ is in $N^{1,p}(X)$ if and only if for $\Mod_p$-a.e.~path $\gamma:[0,\infty)\to X$ that tends to $\infty$
we have that $\lim\limits_{t\to\infty}u(\gamma(t))=0$.
\end{enumerate}
\end{theorem}

The second claim in the above theorem is a natural extension of the study done in~\cite{EKN, KN, KKN, LN}, and
has its precursors in~\cite{Kud} for the Euclidean $[0,\infty)$, \cite{Usp, Timan} for radial lines in $\R^n$, and
\cite{Port, Feff} for vertical lines in $\R^n$. Interestingly,
the works~\cite{Usp, Timan} tell us that when $u\in D^{1,p}(\R^n)$, for $\mathcal{H}^{n-1}$-almost every
direction in $\mathbb{S}^{n-1}$ the radial limit along that direction exists and is independent of the direction precisely when
$1\le p<n$. In contrast, from Theorem~\ref{thm:main1}(i), we know that 
for each $1\le p<\infty$
there exists a function $f\in D^{1,p}(\R^n)$ for which, even if the limit along $p$-almost every curve tending to $\infty$
is independent of the curve, the function $f$ fails to be in $N^{1,p}(X)+\R$. Moreover, in the case that $D^{1,p}(X)\ne N^{1,p}(X)+\R$,
either there is a function in $D^{1,p}(X)$ for which there exist multiple limits along a positive 
$p$-modulus family of curves tending to $\infty$ (as is the case of the hyperbolic space $\mathbb{H}^n$ with $p>n-1$), or else $X$ does not support a global $(p,p)$-Sobolev inequality
(as is the case with the Euclidean $\R^n$).

Of the extant literature mentioned above, the manuscripts~\cite{KN, KKN} consider the family $\Gamma^{+\infty}$
of locally rectifiable curves $\gamma$ in $X$ such that $\gamma\setminus B$ is nonempty for each ball $B\subset X$, while the other
papers focused on families of radial curves~\cite{EKN, Usp, Timan}, and uni-directional families of curves~\cite{Port, Feff, EKN}.
A priori, $\Gamma^{+\infty}$ is a larger family than the family of all curves tending to infinity as considered in the present paper.
However, any curve $\gamma$ that belongs to $\Gamma^{+\infty}$ but does not tend to infinity in our sense has the property that
it intersects some ball $B(x_0,n)$ infinitely often, and subfamilies of such curves have $p$-modulus zero. Hence, in spirit, our
notion of tending to infinity agrees with the notion of infinite curves from~\cite{KN}.

In this paper, we avoid using the term \emph{homogeneous spaces} as there is more than one contender for that term.
In addition to $D^{1,p}(X)$, which consists of \emph{functions}  (see also~\cite{EKN, KN, KKN, LN}), 
certain equivalence classes of 
functions from $D^{1,p}(X)$ are considered
in~\cite{Peet, Sh-2024, ShGib2024}, and the $D^{1,p}(X)$-closure of $N^{1,p}(X)$ is considered,
for example, in~\cite{V, HZ, DKM, ST}. 
This last class serves as the test class of functions in determining global $p$-harmonicity of functions on $X$,
and so is intimately connected to the first class $D^{1,p}(X)$ which serves as the boundary data for the 
Dirichlet problem. We do point out here that  $D^{1,p}(X)$ being larger than $N^{1,p}(X)+\R$ \emph{does not}
automatically guarantee existence of globally $p$-harmonic functions in the class $D^{1,p}(X)$ with the prescribed
``boundary data'' from $D^{1,p}(X)$. Indeed, it was shown in~\cite{BBS-A} that in metric measure spaces 
that are of globally $p$-controlled
geometry and are annular quasiconvex (as for example $\R^n$ is when $n>1$) there are no nonconstant
$p$-harmonic functions in $X$ with finite $p$-energy, even though, as we show in this note, in such 
spaces we have $D^{1,p}(X)\ne N^{1,p}(X)+\R$. 
We refer interested readers to~\cite{Maz} for more on Sobolev spaces.

\vskip .3cm
\noindent {\bf Acknowledgments:} The authors thank Ilkka Holopainen for a
discussion that lead them to the reference~\cite{Strichartz}.

\section{Background} \label{Sec:Background}

In this section, we gather together notions and results needed throughout the paper. 

We use $\overline{\R}$ to denote the natural two-point compactification $[-\infty, \infty]$ of $\R$, 
while $[0, \infty]$ denotes the one-point compactification of the ray $[0, \infty)$. Given $A_1, 
A_2 \in \overline{\R}$, we use $A_1 \lesssim A_2$ to denote that there exists a constant $C \in (0, \infty)$ such 
that $A_1 \le C A_2$; in case the reader needs to be aware of which parameters $C$ is dependent or 
independent on, this will be clarified in writing. We also denote $A_1 \approx A_2$ if $A_1 \lesssim A_2$ 
and $A_2 \lesssim A_1$.  

Let $(X, d)$ be a metric space. If $\mu$ is a Borel-regular measure on $(X, d)$, then we call the triple $(X,d,\mu)$ a \emph{metric measure space}. We also say that $(X,d)$ is \emph{proper} if closed and bounded subsets of $X$ are compact; this property is also sometimes called the Heine-Borel property.

Given $x\in X$ and
$r>0$, we set $B(x,r):=\{y\in X\, :\, d(x,y)<r\}$ and call it the ball centered at $x$ with radius $r$. 
We also call the set $\overline{B}(x,r):=\{y\in X\, :\, d(x,y)\le r\}$ the closed ball centered at $x$ with radius $r$.
Given a ball $B\subset X$, we use
$\rad(B)$ to denote the prescribed radius of $B$. If $x \in X$ and $A\subset X$, we denote $\diam(A):=\sup\{d(x,y)\, :\, x,y\in A\}$ and $\dist(x,A):=\inf\{d(x,y)\, :\, y\in A\}$. 

The length of a curve $\gamma$ in $X$ is denoted by $\ell(\gamma)$, and the path-integral of a Borel function $g$ on
$\gamma$ is denoted by $\int_\gamma g\, ds$. We refer the reader to~\cite{AT, HKSTbook} for more on these notions. 

\begin{stass} 
Throughout this paper, unless otherwise specified, we assume that $1 \le p < \infty$, and that $(X,d,\mu)$ is a 
connected proper 
metric measure space such that $\mu$ is a Borel-regular measure with 
$\mu(X) = \infty$ and $0 < \mu(B(x, r)) < \infty$ for all $x \in X$ and $r \in (0, \infty)$.
\end{stass}

\subsection{Geometry and potential theory related to curves in $X$}

\begin{defn}\label{def:modp}
Given a family $\Gamma$ of curves in $X$, we say that a non-negative Borel function $\rho$ is admissible for $\Gamma$ if
\[
\int_\gamma\!\rho\, \dd s\ge 1
\]
whenever $\gamma\in\Gamma$. The \emph{$p$-modulus} of $\Gamma$ is the number
\[
\Mod_p(\Gamma)=\inf_{\rho}\int_X\!\rho^p\, \dd\mu,
\]
where the infimum is over all admissible functions $\rho$ for $\Gamma$.
\end{defn}

Note that the $p$-modulus of the empty set is zero, and in fact, the $p$-modulus is an outer measure on the collection of all
curves in $X$; we refer the interested reader to~\cite{HKSTbook} for details regarding the $p$-modulus. The following lemma
is a key tool in potential theory on metric measure spaces, and is 
due to Fuglede~\cite{Fuglede}; see also~\cite{KoMc}.

\begin{lemma}\label{lem:useful}
	Let $(X, d, \mu)$ satisfy our standing assumptions. Then the following two statements 
	hold true.
	\begin{enumerate}
		\item If $\Gamma$ is a family of curves in $X$, then $\Mod_p(\Gamma)=0$ if and only if there exists a non-negative
	Borel function $\rho\in L^p(X)$ such that $\int_\gamma\rho\, ds=\infty$ for each $\gamma\in\Gamma$.
	    \item If $\{g_k\}_k$ is a sequence of non-negative Borel functions on $X$ that is convergent to a function $g\in L^p(X)$,
	then there exists a subsequence $\{g_{k_m}\}_m$ and a family $\Gamma$ of curves in $X$ with $\Mod_p(\Gamma)=0$ such that
	whenever $\gamma$ is a rectifiable curve in $X$ with $\gamma\not\in\Gamma$, then
	\[
	\lim_{m\to\infty}\int_\gamma\! |g_{k_m}-g|\, \dd s=0.
	\]
	\end{enumerate}
\end{lemma}

Given a Borel function $\rho:X\to[0,\infty]$ with $\int_X\rho^p\, d\mu$ finite, and $x,y\in X$, we say that $x\sim_\rho y$ if there 
exists a rectifiable curve $\gamma$ with endpoints $x$ and $y$ such that $\int_\gamma\rho\, ds$ is finite.
The relation $\sim_\rho$ is an equivalence relation, decomposing $X$ into pairwise disjoint equivalence classes. It follows
from~\cite{JJRRS} that when $\mu$ is doubling and $X$ is complete, the equivalence classes corresponding to 
$\sim_\rho$ must be Borel sets. More generally, equivalence classes in complete separable metric measure spaces
equipped with a locally finite Borel measure must be measurable sets, see~\cite[Theorem~9.3.1]{HKSTbook}.
We will make use of this fact in this paper.

The following definition was first identified by Ohtsuka as a property possessed by Euclidean spaces in~\cite{O1, O2}.

\begin{defn}\label{def:mecp}
We say that $X$ has the \emph{$p$-main equivalence class property}, or $MEC_p$ property, if for each 
$\rho:X\to[0,\infty]$ with $\int_X\rho^p\, \dd\mu<\infty$ there exists an equivalence class $E(\rho)$ under $\sim_\rho$ such that
$\mu(X\setminus E(\rho))=0$.
\end{defn}

We will show in Lemma~\ref{lem:MECp} that if a space $X$ that satisfies our standing assumptions 
also is of locally $p$-controlled geometry as in Definition~\ref{def:controlled-geom},
then $X$ has the $MEC_p$ property.

\subsection{Newton-Sobolev and Dirichlet classes of functions in nonsmooth settings}

In the nonsmooth setting of metric measure spaces, a viable substitute for distributional derivatives is given by the
notion of minimal $p$-weak upper gradients.

\begin{defn}\label{def:ug}
	Given a function $u:X\to\overline{\R}$, we say that a Borel function $\rho:X\to[0,\infty]$ is an \emph{upper gradient} of
$u$ if for every non-constant compact rectifiable curve $\gamma$ in $X$, we have
\begin{equation}\label{eq:fun-ug-ineq}
|u(y)-u(x)|\le \int_\gamma\!\rho\, \dd s,
\end{equation}
where $x$ and $y$ are the two endpoints of $\gamma$. By convention, we interpret the above inequality 
as $\int_\gamma \rho\, ds=\infty$ whenever at least one of $|u(x)|, |u(y)|$ is infinite. We say that $\rho$ is a 
\emph{$p$-weak upper gradient} of $u$ if the family $\Gamma$ of curves $\gamma$ for which~\eqref{eq:fun-ug-ineq} fails
has zero $p$-modulus.
\end{defn}

\begin{remark}\label{rem:minimal-weak}
Thanks to Lemma~\ref{lem:useful}, we know that $u$ has an upper gradient $g$ with $\int_Xg^p\, d\mu$ finite if and only if
it has a $p$-weak upper gradient $\rho$ with $\int_X\rho^p\, d\mu$ finite. By~\cite{Haj}, we know that if $u$ has a $p$-integrable $p$-weak upper 
gradient, then there exists a \emph{minimal $p$-weak upper gradient} $g_u$ of $u$, where the minimality is understood in the
sense that $g_u\le g$ $\mu$-a.e.~in $X$ whenever $g$ is a $p$-weak upper gradient of $u$. Observe that $g_u$ is unique up to sets of
$\mu$-measure zero.
\end{remark}

\begin{defn}\label{def:controlled-geom}
We say that $(X,d,\mu)$ is of \emph{uniformly locally $p$-controlled geometry} if it
is uniformly locally doubling and supports a uniformly local $p$-Poincar\'e inequality, 
i.e.\ if there exist $0<R_0\le \infty$, $C_d\ge 1$, and 
$C,\lambda>0$ such that whenever $x\in X$ and 
$0<r<R_0$, we have 
\begin{equation}\label{eq:doubling}
	\mu(B(x,2r))\le C_d\, \mu(B(x,r))
\end{equation}
and for all function-upper gradient pairs $(u, g)$ on $X$ 
with $u \in L^1_\loc(X)$, 
we have 
\begin{equation}\label{eq:p-p_Sob-Poin}
\jint_{B(x,r)}|u-u_{B(x,r)}|\, d\mu\le C\, r\, \left(\jint_{{B(x,\lambda r)}}\!g^p\, \dd\mu\right)^{1/p}.
\end{equation}
Moreover, we say that $(X,d,\mu)$ is of 
\emph{globally $p$-controlled geometry} if we can take $R_0=\infty$.
\end{defn}

The property~\eqref{eq:doubling} in the above definition is called the \emph{uniform 
local doubling property} of the measure $\mu$.

\begin{defn}\label{def:global-pp}
We say that $(X,d,\mu)$ supports a global $(p,p)$-Sobolev inequality for 
compactly supported functions if there is a constant
$C>0$ such that 
\[
\int_X\!|u|^p\, \dd\mu\le C\, \int_X\! g_u^p\, \dd\mu
\]
whenever $u\in N^{1,p}(X)$ has compact support in $X$.
\end{defn}

\begin{remark}\label{rem:Keith}
Thanks to~\cite[Theorem~2]{Keith}, we know that $X$ supports a $p$-Poincar\'e inequality 
(locally or not) for all measurable function-upper gradient
pairs if and only if it supports a $p$-Poincar\'e inequality for all compactly supported Lipschitz functions $u$ in terms of the 
local Lipschitz constant function $\lip u$. In the Euclidean setting, and more generally in the Riemannian setting, the 
$p$-Poincar\'e inequality is known to hold for Lipschitz functions $u$ and their upper gradients $g:=|\nabla u|$, where 
$|\nabla u|$ also coincides with $\lip u$ almost everywhere. Thus, in these settings, every measurable function 
with a $p$-weak upper gradient in the class $L^p(X)$ is automatically in the class $L^1_{\text{loc}}(X)$.
Moreover, if $u\in L^1_{\text{loc}}(\Omega)$, where $\Omega$ is a Euclidean domain, we know that $u$ has 
a $p$-weak upper gradient in $L^p(X)$ if and only if
it has a weak derivative $\nabla u$, and in this case, we  also have that $|\nabla u|$ is a minimal $p$-weak upper gradient of 
$u$; see~\cite[Theorem~7.1.13]{Haj}.
\end{remark}

When $(X,d)$ is connected, as is the case throughout this paper, 
the scale $R_0$ in the above condition can be increased, provided we also increase the doubling constant $C_d$ and
the Poincar\'e inequality constant $C$ as well. Hence, we can assume that $R_0\ge 1$. 

Note that if $\mu$ is globally doubling, i.e.\ when we can select $R_0=\infty$ for the aforementioned doubling property of $\mu$, then there 
exist $Q>0$ 
and $C\ge 1$ such that whenever $x\in X$ and $0<r<R<\infty$, we have
\begin{equation}\label{eq:lower-mass-bdd}
	\frac1C\, \left(\frac{r}{R}\right)^Q\le \frac{\mu(B(x,r))}{\mu(B(x,R))}.
\end{equation}
On the other hand, when $\mu$ is globally doubing and $X$ is connected, there is also a positive number
$s>0$ such that
\begin{equation}\label{eq:uppr-mass-bdd}
	\frac1C\, \left(\frac{r}{R}\right)^s\ge \frac{\mu(B(x,r))}{\mu(B(x,R))}
\end{equation}
whenever $x\in X$ and $0<r<R<2\diam(X)$.
Note that if $X$ supports a global $p$-Poincar\'e inequality for some $1\le p<\infty$, i.e.\ when we can 
select $R_0=\infty$ for the $p$-Poincar\'e inequality mentioned above, then $X$ is necessarily connected.

Next, we give a brief introduction to the construction of Newton-Sobolev spaces, which are suitable substitutes for
Sobolev spaces in the nonsmooth metric setting. When the ambient metric measure space is a Euclidean domain, 
these function spaces coincide with the standard Sobolev spaces; see the discussion in~\cite{Haj, Sh0}.

\begin{defn}\label{def:Dirich-Newt}
	We say that a measurable function $u:X\to\R$ is in the \emph{Dirichlet class} $D^{1,p}(X)$ if there exists an upper gradient $g$ 
of $u$ such that $\int_Xg^p\, \dd\mu$ is finite. We say that $u$ is in the \emph{Newton-Sobolev class} $N^{1,p}(X)$ if $u\in D^{1,p}(X)$ and $\int_X|u|^p\, \dd\mu$ is finite.
\end{defn} 

Thanks to Lemma~\ref{lem:useful}, we know that $u\in D^{1,p}(X)$ if and only if it has a $p$-weak upper gradient $g$ such that
$\int_Xg^p\, \dd\mu<\infty$; see Remark~\ref{rem:minimal-weak}.

\begin{remark}\label{rem:pointwise-Newton}
If we modify an upper gradient on a set of $\mu$-measure zero, the resulting function may not always be an upper gradient of $u$,
but it will be a $p$-weak upper gradient of $u$. For this reason, we can unambiguously say that 
$g\in L^p(X)$ if $\int_Xg^p\, \dd\mu$ is finite.
Modifying $u$ on a set of $\mu$-measure zero is a more serious issue, as the resulting function might not have \emph{even a single} 
$p$-weak upper gradient that belongs to $L^p(X)$. In the rest of the paper, when we say that $u\in L^p(X)$, 
we actually mean it to be a short-hand for $u$ being a pointwise well-defined measurable function with $\int_X|u|^p\, \dd\mu$ finite.
\end{remark}

We will make use of a localized version of the definition of quasiconvex spaces.

\begin{defn}\label{def:unif_loc_QC}
	We say that a metric space $(X, d)$ is \emph{uniformly locally quasiconvex} if there exist 
	constants $R \in (0, \infty]$ and $C_q \in [1, \infty)$ such that for all $x, y \in X$ with $d(x, 
	y) \le R$, there exists a rectifiable curve $\gamma_{xy}$ from $x$ to $y$ with $\ell(\gamma_{xy}) 
	\le C d(x, y)$.
\end{defn}

We point out the following lemma, which follows from a suitably localized version of the proof 
of~\cite[Theorem~8.3.2]{HKSTbook}; the details are left to the interested reader.

\begin{lemma}\label{lem:bdd_geom_is_ULQC}
	Let $1 \le p < \infty$, and let $(X, d, \mu)$ be a space of uniformly locally $p$-controlled 
	geometry that satisfies our standing assumptions. Then $X$ is uniformly locally quasiconvex, with 
	$R$ and $C_q$ dependent only on the data $R_0, C_d, C, \lambda$ associated with the locally 
	$p$-controlled geometry of $X$.
\end{lemma}

Next, we recall the definition of relative $p$-capacity of condensers.

\begin{defn}\label{def:p-capacity}
Let $1 \le p < \infty$. Given two sets $E,F\subset X$, the \emph{relative $p$-capacity} of the 
condenser $(E,F)$ is the number
\[
\pcap_p(E,F):=\inf_u\int_X \!g_u^p\, \dd\mu,
\]
where the infimum is over all $u\in N^{1,p}(X)$ that satisfy $u\ge 1$ on $E$ and $u\le 0$ on $F$.
\end{defn}

A note of caution is in order here: in~\cite{HKSTbook}, 
the relative $p$-capacity was defined using $u\in D^{1,p}(X)$, and so our definition is more restrictive than theirs. Instead,
our definition of relative $p$-capacity, designed to deal with ends at infinity, follows the definition used in~\cite{HoloKos}.
As a consequence of the more restrictive definition, we do not in general have $\pcap_p(E,F)=\pcap_p(F,E)$ when $\mu(X)=\infty$. Indeed, if 
$F=X\setminus O$ for some bounded open set $O$ with $E\Subset O$, as is the case 
in some of the applications in this paper, 
then there are no functions $u$ in $N^{1,p}(X)$ satisfying $u\ge 1$ on $F$, and hence we will have
$\pcap_p(F,E)=\infty$ while $\pcap_p(E,F)<\infty$.

\begin{lemma}\label{lem:mod-cap}
Let $1 \le p < \infty$, and let $(X, d, \mu)$ satisfy our standing assumptions. Suppose that 
$x_0\in X$, $0<r<R$. Set 
$\Gamma(B(x_0,r), X\setminus B(x_0,R))$
to be the collection of all curves in $X$ with one endpoint in $B(x_0,r)$ and the other endpoint in $X\setminus B(x_0,R)$. Then,
\[
\pcap_p(B(x_0,r), X\setminus B(x_0,R))=\Mod_p(\Gamma(B(x_0,r), X\setminus B(x_0,R))).
\]
\end{lemma}

\begin{proof}
	First, let $u\in N^{1,p}(X)$ with $u\ge 1$ on $B(x_0,r)$ and $u\le 0$ on $X\setminus B(x_0,R)$, and let $g$ be 
an upper gradient
of $u$. Then for each rectifiable compact curve $\gamma\in \Gamma(B(x_0,r), X\setminus B(x_0,R))$, setting 
$x$ and $y$ to be the two endpoints
of $\gamma$, we have
\[
1\le |u(x)-u(y)|\le \int_\gamma\! g\, \dd s,
\] 
that is, $g$ is admissible for $\Gamma(B(x_0,r), X\setminus B(x_0,R))$. Thus, we have that
\[
\pcap_p(B(x_0,r), X\setminus B(x_0,R))\ge \Mod_p(\Gamma(B(x_0,r), X\setminus B(x_0,R))).
\]
On the other hand, if $\rho$ is a admissible for $\Gamma(B(x_0,r), X\setminus B(x_0,R))$ with $\int_X\rho^p\, d\mu$ finite,
then the function $u:X\to\R$ given by 
\[
u(x):=\min\left\{1,\, \inf_\gamma\int_\gamma\!\rho\, \dd s\right\}
\]
with infimum over all rectifiable curves in $X$ with one endpoint in $X\setminus B(x_0,R)$ and the other at $x$ is a measurable function.
Since $u$ is also compactly supported and bounded, therefore $u$ belongs to $L^p(X)$. Moreover, $\rho$ is an upper gradient
of $u$ by Lemma~\ref{lem:upper-is-upper}. It follows that $u\in N^{1,p}(X)$. Furthermore, we have that $u=0$ on $X\setminus B(x_0,R)$
and $u=1$ on $B(x_0,r)$, the latter holding because $\rho$ is admissible for $\Gamma(B(x_0,r), X\setminus B(x_0,R))$. It follows that
\[
\pcap_p(B(x_0,r), X\setminus B(x_0,R))\le \Mod_p(\Gamma(B(x_0,r), X\setminus B(x_0,R))).
\]
\end{proof}

For further details regarding the notions discussed above, we refer the reader to~\cite{HKSTbook}.

\subsection{Large-scale geometry and potential theory of $X$}

We consider condensers of the form $(B(x_0,r), X\setminus B(x_0,R))$ and $(B(x_0,r),X\setminus\Om)$, 
where $x_0\in X$, $R > r > 0$, and $\Om$ is a bounded domain with
$\overline{B}(x_0,r)\subset \Om$. 
Note that when $0<r<R_1<R_2$, we have 
$\pcap_p(B(x_0,r), X\setminus B(x_0,R_1))\ge \pcap_p(B(x_0,r), X\setminus B(x_0,R_2))$, and if
$B(x_0,R_1)\subset\Om\subset B(x_0,R_2)$, then we have 
\[
\pcap_p(B(x_0,r), X\setminus B(x_0,R_1))\ge \pcap_p(B(x_0,r), X\setminus \Om) \ge \pcap_p(B(x_0,r), X\setminus B(x_0,R_2)).
\]

\begin{defn}\label{def:para-hyp}
We say that $X$ is \emph{$p$-hyperbolic} if $\lim_{R\to\infty}\, \pcap_p(B(x_0,r), X\setminus B(x_0,R))>0$ for some
$x_0\in X$ and $r>0$, and 
we say that $X$ is \emph{$p$-parabolic} if it is not $p$-hyperbolic.
\end{defn}

The early works regarding the above notions for the case $p=2$ are by Grigor'yan~\cite{Grig}; we refer the interested
reader to~\cite{Grig} and the references therein for more on the relationships between $2$-hyperbolicity and 
the transience of Brownian motion. In the metric setting, see~\cite{Holo, HoloKos, HoloSh}.

Under the assumption that $(X,d,\mu)$ is of locally $p$-controlled geometry, the space 
$X$ is $p$-hyperbolic if and only if for each $x_0\in X$ and $r>0$,  we have that
$\lim_{R\to\infty}\, \pcap_p(B(x_0,r), X\setminus B(x_0,R))>0$; see~\cite[Lemma~3.15]{HoloSh}.
The proof of that result in~\cite{HoloSh} required only the $MEC_p$-property of $X$,
but thanks to Lemma~\ref{lem:MECp} below, under our standing assumptions we know that $X$ has this property whenever $X$ is of locally $p$-controlled geometry.
We also wish to correct an erroneous reference in the proof given in~\cite{HoloSh}; the reference~[21] of that paper,
in that proof, should instead be the reference listed as~[22] there. The specific result used there is actually~\cite[Theorem~6.4]{Sh2}. 

The following is from~\cite[Proposition~2.3]{HoloKos}, see also~\cite{KZ, Holo} for the Riemannian setting. 
Note that the converse of the result need not hold. 

\begin{lemma}\label{lem:parab-volume}
Let $(X, d, \mu)$ satisfy our standing assumptions. Suppose that $1<p<\infty$, and fix $x_0\in 
X$. If
\[
\int_1^\infty\! \left(\frac{r}{\mu(B(x_0,r))}\right)^{1/(p-1)}\, \dd r=\infty, 
\]
then $X$ is $p$-parabolic.
\end{lemma}

However, the above lemma is not suitable for our needs; we prove the following modification of this lemma.

\begin{lemma}\label{lem:parab-volume-sum}
Let $(X, d, \mu)$ satisfy our standing assumptions, and suppose that $1<p<\infty$. If there 
exists an increasing sequence of bounded open sets $\Om_j$, $j\in\N$, with $\overline{\Om}_j\subset \Om_{j+1}$ and $X=\bigcup_{j\in\N}\Om_j$, 
for which the distances $\Delta_j=\dist(\Om_j,X\setminus\Om_{j+1})$ satisfy
\begin{equation}\label{eq:p_para_cond}
\sum_{j=1}^\infty \left(\frac{\Delta_j^p}{\mu(\Om_{j+1}\setminus\Om_j)}\right)^{1/(p-1)}=\infty,
\end{equation}
then $X$ is $p$-parabolic.
\end{lemma}

\begin{proof}
For each positive integer $j$ consider the function $u_j:X\to\R$ given by
\[
u_j(x)=\min\{1,\dist(x,X\setminus\Om_{j+1})/\Delta_j\}.
\]
Then $\Delta_j^{-1}\, \chi_{\Om_{j+1}\setminus\Om_j}$ is an upper gradient of $u_j$, and as $\Om_{j+1}$ has
compact closure, we have that $u_j\in N^{1,p}(X)$ with
\[
\pcap_p(\Om_j,X\setminus\Om_{j+1})\le \int_X\!g_{u_j}^p\, \dd\mu\le \frac{\mu(\Om_{j+1}\setminus\Om_j)}{\Delta_j^p}.
\]
It follows that
\[ 
\left(\pcap_p(\Om_j,X\setminus\Om_{j+1})\right)^{-1/(p-1)}
\ge \left(\frac{\mu(\Om_{j+1}\setminus\Om_j)}{\Delta_j^p}\right)^{-1/(p-1)}.
\]
Thus, if $\overline{B}_1$ is any closed ball in $\Om_1$, then~\cite[Lemma~2.1]{HoloKos} yields that 
$\lim_{j\to\infty}\pcap_p(\overline{B}_1,X\setminus\Om_j)=0$, 
and hence $(X,d,\mu)$ is $p$-parabolic.
\end{proof}

Next, we discuss the notion of ends of a metric measure space, using the notions from~\cite{AdSh, Est}. 
See~\cite{EO, Anc, AndSc} for earlier constructions in the setting of simply connected Riemannian manifolds 
of negative sectional curvature.

\begin{defn}\label{def:ends}
A sequence $\{F_j\}_j$ of subsets of $X$ is said to be an \emph{end at infinity} if 
$\overline{F_{j+1}} \subset F_j$ and $\dist(X\setminus F_j, F_{j+1})>0$ for all $j\in\N$, the set 
$\bigcap_{j\in\N}F_j$ is empty, and for each $j\in \N$, there exists a compact set $K_j\subset X$ such that $F_j$ is an unbounded component of $X\setminus K_j$. By a slight abuse of notation, we say that two ends at infinity, $\{E_j\}_j$ and $\{F_j\}_j$, 
are the same if
for every positive integer $j$ there exist positive integers $i,k$ such that $E_i\subset F_j$ and $F_k\subset E_j$.
\end{defn}

Since the only ends we consider in this paper are ends at infinity as in the above definition, we may refer to them merely as ends for brevity.

If $\{F_j\}$ is an end with associated compact sets $K_j$, $j \in \N$ as in the above definition, 
then $K_j\subset X\setminus F_j=:\widehat{K}_j$. Note that $\widehat{K}_j$ itself need not be compact. For each $j\in\N$, let
$\Gamma_j$ denote the collection of all locally rectifiable curves $\gamma:[0,\infty)\to X$ such that $\gamma(0)\in K_1$ and
$\gamma(\tau,\infty)\subset F_j$ for some $\tau>0$, and let $\Gamma:=\bigcap_j\Gamma_j$. Note that 
$\Gamma_{j+1}\subset\Gamma_j$ and, for each $j > 1$, the quantity $\Mod_p(\Gamma_j)$ is finite.

\begin{defn}\label{def:ends-hyp}
We say that the end $\{F_j\}_j$ is a \emph{$p$-hyperbolic end} if $\lim_{j\to\infty}\Mod_p(\Gamma_j)>0$. We say that $\{F_j\}_j$ is a \emph{$p$-parabolic end} if it is not a $p$-hyperbolic end. 
\end{defn}

We note that $\lim\limits_{j\to\infty}\Mod_p(\Gamma_j)>0$ if and only if $\Mod_p(\Gamma)>0$, see~\cite[Theorem~4.2]{Sh1}.

\subsection{Gromov hyperbolicity, uniformization}\label{sub:hypunif}

We next discuss a large-scale metric notion of hyperbolicity, proposed first by Gromov in~\cite{Gro}. We state the 
definition for geodesic spaces; while the notion of Gromov hyperbolicity is typically defined for a more general 
class of spaces, it has a simpler geometric characterization in this setting as opposed to the more natural
Gromov product definition given in the setting of geometric groups.

\begin{defn}\label{def:Gromov}
Given $\delta > 0$, a complete geodesic metric space $(X,d)$ is said to be \emph{Gromov $\delta$-hyperbolic} if, whenever 
$x,y,z\in X$ are three distinct points and $[x,y], [y,z]$, and $[z,x]$ are geodesic curves with endpoints $x,y$, $y,z$, and
$z,x$, respectively, then 
\[
	[z,x]\subset \bigcup_{w\in [x,y]\cup[y,z]}B(w,\delta).
\]
Moreover, a space $(X,d)$ that is Gromov $\delta$-hyperbolic for some $\delta > 0$ is called \emph{Gromov hyperbolic}.
\end{defn}

Furthermore, the following notion of rough starlikeness is often used in tandem with Gromov hyperbolicity, see for example~\cite{BHK}.

\begin{defn}\label{def:roughly-starlike}
Let $(X, d)$ be an unbounded metric space, let $z_0 \in X$, and let $M > 0$. We say that a curve 
$\gamma \colon [0, \infty) \to X$ of infinite length is a \emph{geodesic ray starting from $z_0$} if 
$\gamma(0) = z_0$ and $\gamma\vert_{[0, t]}$ is a geodesic for every $t > 0$. The space $X$ is 
called \emph{$M$-roughly $z_0$-starlike} if, for every $x \in X$, there exists a geodesic ray 
$\gamma$ starting from $z_0$ such that $\dist (x, \gamma) \le M$. Moreover, a space $X$ is 
\emph{roughly starlike} if it is $M$-roughly $z_0$-starlike for some $M > 0$ and $z_0 \in X$. 
\end{defn}

Note that if a space $(X,d)$ is $M$-roughly $z_0$-starlike and $z_1 \in X$, then $X$ is $(M+d(z_0, z_1))$-roughly 
$z_1$-starlike. Hence, unless one is concerned with the exact value of $M$, the base point $z_0$ of a roughly 
starlike space can be chosen arbitrarily.

The notion of Gromov hyperbolicity as described above has the additional advantage in comparison to the original definition 
in that the description in Definition~\ref{def:Gromov} is closely linked to the notion of uniform domains, see~\cite{BHK}. We now give a brief overview of this connection, as we use this link as a tool in this paper.

\begin{defn}\label{def:uniform}
A metric space $(Z,d_Z)$ is a \emph{uniform domain} if it is locally compact, non-complete, and 
there exists a constant $A\ge 1$ satisfying the following condition: for each pair of distinct points 
$x,y\in Z$, there exists a curve $\gamma$ in $Z$ with endpoints $x, y$, called an $A$-\emph{uniform curve}, 
such that $\ell(\gamma)\le A\, d_Z(x,y)$ 
and, for each point $z$ in the trajectory of $\gamma$ and for all subcurves $\gamma_{xz}$ and $\gamma_{zy}$ of $\gamma$ with endpoints $x,z$ and
$z,y$, respectively, we have
\[
	\dist(z,\partial Z)\ge \frac{1}{A}\, \min\{\ell(\gamma_{xz}), \ell(\gamma_{zy})\}.
\]
Here, $\partial Z:=\overline{Z}\setminus Z$, where $\overline{Z}$ is the metric completion of $Z$.
\end{defn}

Given a Gromov hyperbolic space $(X,d)$ and $\eps>0$, we define a new metric $d_\eps$,
called a \emph{uniformized metric}, on $(X,d)$ as follows:
we first fix a point $z_0\in X$, and denote
\begin{equation}\label{eq:rho_def}
	\rho_\eps(z) := e^{-\eps d(z, z_0)}
\end{equation}
for all $z \in X$. Then, for each pair of points $x,y\in X$, we set
\begin{equation}\label{eq:d-eps}
d_\eps(x,y):=\inf_\gamma \, \int_\gamma\! \rho_\eps(\gamma(s))\, \dd s
\end{equation}
where the infimum is over all arc-length parametrized curves $\gamma$ with endpoints $x$ and $y$.

The following lemma is from~\cite[Proposition~4.5]{BHK}. As it is a purely metric result, we elect 
not to assume our standing assumptions in its statement.

\begin{lemma}\label{lem:BHK1}
Given a proper Gromov $\delta$-hyperbolic space $(X,d)$, there exists a positive real number
$\eps(\delta)$ such that whenever $0<\eps\le \eps_0(\delta)$, we have that $(X,d_\eps)$ is an $A$-uniform domain, where $A = A(\delta, \eps)$.
\end{lemma}

\begin{defn}\label{def:partial-eps}
When $(X,d)$ is a Gromov hyperbolic space and $\eps>0$, we denote $\partial_\eps X := \overline{X}^\eps\setminus X$, 
where $\overline{X}^\eps$ is the metric completion of $(X,d_\eps)$.
\end{defn}

The object $\partial_\eps X$ should not be confused with the definition of ends at infinity from the prior subsection; 
while both of these concepts detect behavior at $\infty$, it is possible that a space $X$ has only a single end 
while $\partial_\eps X$ is a continuum of at least two 
points. Indeed, this occurs for the standard hyperbolic space $\mathbb{H}^n$.

Next, suppose that $(X,d,\mu)$ is a proper Gromov hyperbolic metric measure space. For each $\beta>0$, 
we let $\mu_\beta$ denote a \emph{uniformized measure} on $X$ given by
$\dd\mu_\beta=\rho_\beta\,\dd\mu$, that is, for $\mu$-measurable sets $E\subset X$ we set
\begin{equation}\label{eq:mu-beta}
	\mu_\beta(E)=\int_E \rho_\beta\, \dd\mu,
\end{equation}
where $\rho_\beta$ is again as in \eqref{eq:rho_def}. Recall that when $(X,d,\mu)$ is of uniformly 
locally $p$-controlled geometry, $C_d$ denotes the uniform local doubling constant
related to $\mu$;  see Definition~\ref{def:controlled-geom} above. For every such space, we set 
\begin{equation*}
\beta_0:=\frac{17\, \log(C_d)}{3\, R_0}.
\end{equation*}
Note that if we make $R_0$ larger, then correspondingly the doubling constant $C_d$ also could become
larger.

The following result is from~\cite[Theorem~6.2]{BBS1}.

\begin{lemma}\label{lem:BBS1-unif}
Let $1 \le p < \infty$, and let $(X, d, \mu)$ satisfy our standing assumptions. Suppose that 
$(X,d,\mu)$ is a proper Gromov hyperbolic metric measure space of uniformly locally $p$-controlled 
geometry, and let $0<\eps\le \eps_0(\delta)$ be as in Lemma~\ref{lem:BHK1} above. Then  for each 
$\beta>\beta_0$,  we have that 
$(X,d_\eps,\mu_\beta)$ is of globally $p$-controlled geometry.
\end{lemma}

We also recall the following three comparability results. The first one of the results is shown in \cite[Lemma 4.16]{BHK}, 
and the other two in \cite[Theorem 2.10 and Proposition 4.7]{BBS1}. Here and in what follows, 
if $x \in \overline{X}^\eps$, $r > 0$, and $A \subset \overline{X}^\eps$, we use $B_\eps(x, r)$, 
$\dist_\eps (x, A)$, and $\diam_\eps(A)$ to denote $B(x, r)$, $\dist(x, A)$, and $\diam(A)$ with respect to the metric $d_\eps$. 

\begin{lemma}\label{lem:unif_comparisons}
Let $(X, d, \mu)$ be a $z_0$-roughly $M$-starlike Gromov $\delta$-hyperbolic space that satisfies 
	our standing assumptions. Let $(X, d_\eps, \mu_\beta)$ be a uniformization of $X$ with base point $z_0$, 
	where $0 < \eps \le \eps_0(\delta)$ and $\beta > \beta_0$. Then there exist constants $K_1, K_2 \ge 1$, 
	both dependent only on $M$ and $\eps$, such that for every $x \in X$, we have
	\begin{equation}\label{eq:rho_d_comparison}
		K_1^{-1} \rho_\eps(x) \leq \dist_\eps(x, \partial_\eps X) \leq K_1 \rho_\eps(x),
	\end{equation}
	and for all $0 < r \le d_\eps(x, \partial_\eps X)/2$, we have
	\begin{equation}\label{eq:d_deps_ball_comparison}
		B\left(z, \frac{r}{K_2 \rho_\eps(z)}\right)
		\subset B_\eps (z, r)
		\subset B\left(z, \frac{K_2 r}{\rho_\eps(z)}\right).
	\end{equation}
	Moreover, if $(X, d, \mu)$ is locally uniformly doubling with parameters $R_0$ and $C_d$, and $a_0 > 0$, then there exists a constant $K_3 = K_3(a_0, \beta, C_d, \delta, \eps, M, R_0) > 0$ such that 
we have
	\begin{equation}\label{eq:boundary_to_inside_comparison}
		\frac{r^{\beta/\eps}}{K_3}  \mu \left(B\left(z, \frac{a_0 r}{\rho_\eps(z)} \right) \right) 
		\le \mu_\beta(B(\xi, r))
		\le K_3 r^{\beta/\eps} \mu \left(B\left(z, \frac{a_0 r}{\rho_\eps(z)} \right) \right)
	\end{equation}
	whenever 
$\xi \in \partial_\eps X$, $0<r<2\diam_\eps (X_\eps)$,  
and
$z \in X$ is such that $B_\eps(z, a_0 r) \subset B_\eps(\xi, r)$ with $\dist_\eps(z, \partial_\eps X) \ge 2a_0 r$.
\end{lemma}

\section{The case of globally doubling measures} 

In this section, we consider metric measure spaces $(X, d, \mu)$ on which the measure $\mu$ is 
globally doubling. In particular, the following result shows that in this case, one has $D^{1,p}(X) \ne 
N^{1,p}(X) + \R$ under highly general assumptions.
	
\begin{theorem}\label{prop:doubling_example}
Let $(X, d, \mu)$ be a globally doubling metric measure space that satisfies our standing assumptions. 
Suppose that there exists a point $x_0 \in X$ such that the function $r \mapsto \mu(B(x_0, r))$ is continuous on $(0, \infty)$. 
Then for every $p \geq 1$, there exists a locally Lipschitz function $f \in L^p_\loc(X)$ such that $\lip f \in L^p(X, \mu)$ 
but $f - c \notin L^p(X, \mu)$ for each $c \in \R$.
\end{theorem}

\begin{proof}
Recall that we employ the notation $C_d$ for the doubling constant of the measure $\mu$. For each non-negative integer $i$, we pick a radius $r_i$ by first setting $r_0 = 1$ 
and then selecting $r_i$ for $i \geq 1$ inductively such that $\mu(B(x_0, r_i)) = (C_d+1)\mu(B(x_0, r_{i-1}))$. 
We also write $B_i:=B(x_0, r_i)$ and $\mu_i:=\mu(B_i)$ for brevity. Note that this selection of radii is indeed 
always possible: our assumption that $\mu(X) = \infty$ implies that $\lim_{r \to \infty} \mu(B(x_0, r)) = \infty$, 
and thus our continuity assumption implies that $r \mapsto \mu(B(x_0, r))$ maps $[r_i, \infty)$ surjectively onto $[\mu_i, \infty)$ for each non-negative integer $i$.
		
By our selection procedure, we have
\begin{equation}\label{eq:measure_growth_rate}
	\mu_i = (C_d+1)^i \mu_0
\end{equation}
for every $i \geq 0$. We also note that since the doubling condition implies that $\mu(B(x_0, 2r_i)) \leq C_d\,\mu_i < \mu_{i+1}$, 
we must have $r_{i+1} > 2\,r_i$ for every $i\geq 0$. Applying this inductively and using $r_0 = 1$, we have $r_i \geq 2^i$. Moreover, 
since $r_{i+1} > 2\,r_i$ can be rearranged to $r_{i+1} - r_i > r_i$, we conclude that
\begin{equation}\label{eq:radii_difference_growth_rate}
	r_{i+1} - r_i >  2^i
\end{equation}
for every $i \geq 0$. We then define a function $\phi \colon [0, \infty) \to [0, \infty)$ as follows: we first set $\phi(0) = \phi(r_0) = 0$ and
\[
	\phi(r_i) = \sum_{j=0}^{i-1} (\mu_{j+1} - \mu_j)^{-\frac{1}{p}}
\]
for every $i \geq 1$, and then we extend $\phi$ to the rest of $[0, \infty)$ by linear interpolation. Clearly $\phi$ is 
non-decreasing, and by~\eqref{eq:measure_growth_rate}, we can compute that
\[
	\phi(r_i) = \sum_{j=0}^{i-1} (C_d\,\mu_0)^{-\frac{1}{p}} (C_d+1)^{-\frac{j}{p}}.
\]
In particular, it follows by the convergence properties of geometric series that $\lim_{r \to \infty} \phi(r) < \infty$.
		
We then define $f \colon X \to \R$ by $f(x) = \phi(d(x, x_0))$. The map $f$ is a composition of two locally Lipschitz 
functions, so it is locally Lipschitz. Since $f$ is continuous, it is also clearly in $L^p_\loc(X)$. Moreover, if 
$r_{i} < d(x, x_0) < r_{i+1}$, then by \eqref{eq:radii_difference_growth_rate} and the linear interpolation 
definition of $\phi$, we have
\begin{align*}
(\lip f)(x) \leq (\lip \phi)(d(x, x_0)) \cdot (\lip d(\cdot, x_0))(x)
	      &\leq \frac{\phi(r_{i+1}) - \phi(r_i)}{r_{i+1} - r_i} \cdot 1\\
	       &\leq 2^{-i} (\mu_{i+1} - \mu_i)^{-\frac{1}{p}}.
\end{align*}
We also note that $\mu(\partial B(x_0, r)) = 0$ for every $r$, since otherwise $r \mapsto \mu(B(x_0, r))$ would 
have a discontinuity. Thus, we may estimate
\[
\int_X\! (\lip f)^p\, \dd \mu \leq \sum_{i=0}^\infty \int_{B_{i+1} \setminus B_i}\! 2^{-pi} (\mu_{i+1} - \mu_i)^{-1}\,\dd\mu = \sum_{i=0}^\infty 2^{-pi} < \infty.
\]
		
Now, we claim that $f-c$ does not have a finite $L^p$-norm for any $c \in \R$. Since $\lim_{r \to \infty} \phi(r) < \infty$, this is 
clearly the case if $c \neq \lim_{r \to \infty} \phi(r)$. Indeed, otherwise we could find $R > 0$ such that 
$\abs{f - c} \geq \tfrac12\abs{c-\lim_{r \to \infty} \phi(r)}$ when $d(x,x_0) \geq R$, and this causes an infinite $L^p$-norm 
since we necessarily have $\mu(X \setminus B(x_0, R)) = \infty$. Hence, we need only consider the case 
$c = \lim_{r \to \infty} \phi(r)$. In this case, since $\phi$ is non-decreasing, we may use \eqref{eq:measure_growth_rate} and estimate
\begin{multline*}
\int_{B_{i+1} \setminus B_i}\! \abs{f - c}^p\, \dd \mu 
	\geq \int_{B_{i+1} \setminus B_i}\! \abs{\phi(r_{i+1}) - c}^p\, \dd \mu\\ 
= (\mu_{i+1} - \mu_i) \left(\sum_{j=i+1}^\infty (\mu_{j+1} - \mu_j)^{-\frac{1}{p}}\right)^p \geq \frac{\mu_{i+1} - \mu_i}{\mu_{i+2} - \mu_{i+1}} 
= \frac{1}{C_d+1}.
\end{multline*}
Thus,
\[
\int_X \! \abs{f - c}^p\, \dd \mu \geq \sum_{i=0}^\infty \frac{1}{C_d+1} = \infty,
\]
completing the proof.
\end{proof}

\begin{cor}\label{cor:globalcontrol}
Suppose that $(X,d,\mu)$ satisfies our standing assumptions and is of globally $p$-controlled geometry for some $1\le p<\infty$. 
Then there exists a locally Lipschitz function
$f\in D^{1,p}(X)\setminus (N^{1,p}(X)+\R)$.
\end{cor}

\begin{proof}
The hypothesis in Theorem~\ref{prop:doubling_example} that there exists a point $x_0\in X$ for which 
the function $r\mapsto\mu(B(x_0,r))$ is 
continuous on $(0, \infty)$ is equivalent to the assumption that for
each $r>0$, the sphere $S(x_0,r):=\{y\in X\, :\, d(x,y)=r\}\supset\partial B(x_0,r)$ satisfies 
$\mu(S(x_0,r))=0$. It is known that when $(X,d,\mu)$ 
is of globally $q$-controlled geometry for some $1\le q<\infty$, then
$X$ is necessarily quasiconvex~\cite[Theorem~8.3.2]{HKSTbook},
and hence, thanks to the properness of $X$ and the Arzel\`a-Ascoli theorem, 
there is a biLipschitz equivalent metric, $d_L$, with respect to which
$(X,d_L)$ is a geodesic space. Since the homogeneous Sobolev space $N^{1,p}(X)$ and the
Dirichlet space $D^{1,p}(X)$ are biLipschitz invariant, we can thus
assume without loss of generality that the metric $d$ is a geodesic metric on $X$. 
In this case, we also have that for each $x_0\in X$ and $0<r<\infty$, the measure of $S(x_0,r)$ is null, see~\cite{Buck}.
Now the conclusion follows from an application of Theorem~\ref{prop:doubling_example} above.
\end{proof}

\section{Sobolev-Poincar\'e inequalities and unbounded Ahlfors regular spaces} 
	
In this section alone, we are concerned with \emph{Ahlfors $s$-regular measures} $\mu$; that is,  
measures $\mu$ for which there exists a constant $C\ge 1$
such that $C^{-1}\, r^s\le \mu(B(x,r)) \le C\, r^s$ whenever $x\in X$ and $0<r<2\diam(X)$.
We call the constant $C$ the \emph{Ahlfors constant} of $\mu$.
In unbounded Ahlfors regular cases, the conclusion from the prior section can be 
strengthened in a way which allows for 
different exponents for the function and the derivative.
	
\begin{prop}\label{prop:Ahlfors_David_example}
Let $(X, d, \mu)$ be an Ahlfors $s$-regular metric measure space that satisfies our standing assumptions. 
Then for all 
$p, q \in [1, \infty)$, if $q^{-1} + s^{-1} \neq p^{-1}$, then there exists a function $f \in L^p_\loc(X)$ such that 
$\lip f \in L^p(X, \mu)$ but $f - c \notin L^q(X, \mu)$ for any $c \in \R$.
\end{prop}

To prove this proposition, we first recall and prove the following standard lemma on unbounded Ahlfors regular spaces.
	
\begin{lemma}\label{lem:Ahlfors_David_distance_integral}
Let $(X, d, \mu)$ be an Ahlfors $s$-regular metric measure space that satisfies our standing assumptions, let $x_0 \in X$, and let 
$r \in (0, \infty)$. Then the map $f(x) = d(x_0, x)^\gamma$ is integrable over $B(x_0, R)$ if and only if $\gamma > -s$. 
Similarly, $f$ is integrable over $X \setminus B(x_0, R)$ if and only if $\gamma < -s$.
\end{lemma}
	
\begin{proof}
For the integral over $B(x_0, R)$, we write
\[
f(x) = R^\gamma - \gamma \int_{d(x_0, x)}^R\! r^{\gamma - 1}\, \dd r.
\]
The part $R^\gamma$ clearly has a finite integral over the finite-measured $B(x_0, R)$. For the other part, we use the 
Tonelli theorem to switch the order of integrals, moving the condition $r > d(x_0, x)$ to the other integral: this yields
\begin{align*}
\int_{B(x_0, R)} \int_{d(x_0, x)}^R\! r^{\gamma - 1}\, \dd r\, \dd \mu(x)
&= \int_{0}^R \! r^{\gamma - 1} \left(\int_{B(x_0, r)} \dd \mu(x)\right) \dd r\\
&\approx \int_{0}^R\! r^{\gamma + s - 1}\, \dd r,
\end{align*}
with comparisons dependent only on the Ahlfors constant $C$. This integral is finite if and only if $\gamma > -s$.
		
For the other case, it is clear that $f$ is not integrable over $X \setminus B(x_0, R)$ if $\gamma \geq 0$, as otherwise 
$f \geq R^\gamma > 0$ in this infinite-measured set. Suppose then that $\gamma < 0$. In this case, we can write
\[
	f(x) = -\gamma \int_{d(x_0, x)}^\infty \! r^{\gamma - 1} \,\dd r.
\]
Analogously to the other case, we get via the Tonelli theorem that
\begin{align*}
	\int_{X \setminus B(x_0, R)} \int_{d(x_0, x)}^\infty \! r^{\gamma - 1} \, \dd r \, \dd \mu(x)
	&= \int_{R}^\infty \! r^{\gamma - 1} \mu(B(x_0, r) \setminus B(x_0, R))\,\dd r.
\end{align*}
Moreover, there exists $S > R$ such that $\mu(B(x_0, r) \setminus B(x_0, R)) \approx r^s$ when $r \geq 
S$, with the comparison constant again depending only on $C$. Thus, we similarly have 
convergence 
exactly when $\gamma < -s$.
\end{proof}

\begin{proof}[Proof of Proposition \ref{prop:Ahlfors_David_example}]
Let $C$ be the Ahlfors constant of $X$. Fix any $x_0 \in X$. Consider first the case $p^{-1} < q^{-1} + s^{-1}$. We define 
$f \colon X \to \R$ by 
\[
	f(x) = \min\left\{ 1, d(x_0, x)^{-\frac{s}{q}} \right\}.
\]
As $f$ is the composition of the $1$-Lipschitz map $x \mapsto d(x_0, x)$ and a locally Lipschitz map, $f$ is locally Lipschitz. 
Moreover, we have $(\lip f)(x) = 0$ when $d(x_0, x) < 1$, and
\begin{equation}\label{eq:local_lip_estimate}
	(\lip f)(x) \leq \frac{s}{q} \, d(x, x_0)^{-\frac{s+q}{q}}
\end{equation}
otherwise.  
Thus, by Lemma~\ref{lem:Ahlfors_David_distance_integral}, $(\lip f)^p$ is integrable since 
$p(s+q)/q = ps(q^{-1} + s^{-1}) > s$. On the other hand, since $\lim_{d(x, x_0) \to \infty} f(x) = 0$, the only possible 
$c \in \R$ for which $f - c$ may be $L^q$-integrable over $X$ is $c = 0$. And in this case, 
Lemma~\ref{lem:Ahlfors_David_distance_integral} shows that $f^q$ is not integrable.
		
In the other case $p^{-1} > q^{-1} + s^{-1}$, we instead use
\[
	f(x) = \max\left\{ 1, d(x_0, x)^{-\frac{s}{q}} \right\}.
\]
This function is locally Lipschitz in $X \setminus \{x_0\}$, and thus extends to $X$ as a measurable function with a single 
point $x_0$ where $(\lip f)(x_0) = \infty$. To conclude that it is in $L^p_\loc(X)$, we observe that 
$p^{-1} > q^{-1} + s^{-1}$ implies $p^{-1} > q^{-1}$, from which it follows that $ps/q < s$, which in turn shows that 
$f \in L^p(B(x_0, 1))$ by Lemma~\ref{lem:Ahlfors_David_distance_integral}.
		
For the local Lipschitz constant of $f$, we again have the estimate \eqref{eq:local_lip_estimate}, this time when 
$d(x_0, x) \leq 1$, with $\lip f$ vanishing otherwise. Thus, we again get by Lemma~\ref{lem:Ahlfors_David_distance_integral} 
that $(\lip f)^p$ 
is integrable, this time since $p(s+q)/q < s$. For $L^q$-integrability of $f-c$, we again observe that $f^q$ has 
infinite integral over $B(x_0, 1)$ by Lemma \ref{lem:Ahlfors_David_distance_integral}; since $B(x_0, 1)$ has 
finite $\mu$-measure and since $\abs{a+b}^q \le 2^q (\abs{a}^q + \abs{b}^q)$, no addition of a constant to $f$ 
will change this integral to a finite one.
\end{proof}

Proposition~\ref{prop:Ahlfors_David_example} also has a partial converse result 
when $(X, d, \mu)$ has globally $p$-controlled geometry.

\begin{prop}
Suppose that $(X,d,\mu)$ is an Ahlfors $s$-regular metric measure space of globally $p$-controlled 
geometry that satisfies our standing assumptions, 
where $1\le p<s$.
Let $q>1$ satisfy $\tfrac{1}{q}+\tfrac{1}{s}= \tfrac{1}{p}$.
If $f\in D^{1,p}(X)$, then $f-c\in L^q(X)$ for some real number $c$ that may depend on $f$.
\end{prop}

\begin{proof}
Suppose that $f$ has an upper gradient $g\in L^p(X)$. We fix $x_0 \in X$ and set $B_j:=B(x_0, 
2^j)$ for all non-negative integers $j$. 
By the Ahlfors $s$-regularity of $X$, it follows that 
	\begin{equation}\label{eq:ball_comparison}
		\frac{1}{C}\, 2^{js} \le \mu(B_j) \le C 2^{js},
	\end{equation}
where the constant $C$ is the Ahlfors constant of $X$.
Thus, by the Sobolev embedding theorem 
\cite[Theorem~5.1]{HaKo}, we have for all non-negative integers $i$ that 
\begin{multline}\label{eq:SP-est_for_pow_2}
\left(\jint_{B_j}\!|f-f_{B_j}|^q\, \dd\mu\right)^{1/q} \lesssim 2^j\, \left(\jint_{5\lambda B_j}\!g^p\, 
\dd\mu\right)^{1/p}\\
   \le \frac{2^j}{\mu(B_j)^{1/p}}\Vert g\Vert_{L^p(X)}
   \leq C^{1/p} \, 2^{j(1-s/p)} 
   \Vert g\Vert_{L^p(X)}
   \lesssim 2^{j(1-s/p)},
\end{multline}
where the comparison constants in this estimate, as well as all the other uses of $\lesssim$ in 
this proof, are independent of $j$ but may depend on any other parameters.

We next claim that $(f_{B_j})_j$ is a Cauchy sequence. Indeed, observe that for all positive integers 
$j$ and $k$, since $p<s$, we have by \eqref{eq:ball_comparison} and \eqref{eq:SP-est_for_pow_2} the 
estimate 
\begin{multline}\label{eq:cauchy_est}
	\abs{f_{B_{j+k}} - f_{B_j}} 
	\leq \sum_{i=j}^\infty \abs{f_{B_{i+1}} - f_{B_i}} 
	\leq \sum_{i=j}^\infty \jint_{B_{i}}\! \abs{f - f_{B_{i+1}}}\,\dd\mu\\
	 \le \frac{C^2}{2^s} \sum_{i=j}^\infty \jint_{B_{i+1}}\! \abs{f - f_{B_{i+1}}}\,\dd\mu
	 \le \frac{C^2}{2^s}  \sum_{i=j}^\infty \left( \jint_{B_{i+1}}\! \abs{f - 
	 f_{B_{i+1}}}^q\,\dd\mu\right)^{1/q}\\
 	\lesssim 
	\sum_{i=j}^\infty 2^{i(1-s/p)}
	=  
	\frac{ 2^{j(1-s/p)}}{1 - 2^{1-s/p}}.
\end{multline}
Hence, we have that $\lim_{j \to \infty} f_{B_j}$ exists. We denote $c := \lim_{j \to \infty} f_{B_j}$. 
Observe that 
by letting $k \to \infty$ in~\eqref{eq:cauchy_est}, we have for every $j$ the estimate
\begin{equation}\label{eq:chain_of_balls_lim_convergence_rate}
	\abs{f_{B_j} - c} \lesssim 2^{j(1-s/p)}.
\end{equation}
Now, by~\eqref{eq:ball_comparison}, \eqref{eq:SP-est_for_pow_2}, 
and~\eqref{eq:chain_of_balls_lim_convergence_rate}, 
we have for all $j$ that 
\begin{multline*}
\left( \int_{B_j}\!|f-c|^q\, \dd\mu\right)^{1/q} \le \left( \int_{B_j}\!|f-f_{B_j}|^q\, 
\dd\mu\right)^{1/q}+|f_{B_j}-c|\,\mu(B_j)^{1/q}\\
\lesssim 2^{j(1-s/p)} + 2^{j(1 - s/p + s/q)},
\end{multline*}
Since $1/q + 1/s = 1/p$, we have $1 - s/p + s/q = 0$. Thus, letting $j\to\infty$ in the above estimate, 
we obtain
\[
\left(\int_X \!|f-c|^q\, \dd\mu\right)^{1/q} <\infty.
\]
\end{proof}

\section{The case of Gromov hyperbolic spaces with uniformly locally $p$-controlled geometry}\label{Sec:4}

Our objective in this section is to study the relationship between $D^{1,p}(X,d,\mu)$ and 
$N^{1,p}(X,d,\mu)$ when $(X,d,\mu)$ is an  roughly starlike Gromov hyperbolic space with uniformly locally 
$p$-controlled geometry. Our primary tool is to consider a uniformization 
$(\overline{X}^\eps, d_\eps, \mu_\beta)$ of the $(X, d, \mu)$, where $0<\eps\le \eps_0(\delta)$ 
and $\beta>\beta_0$ as in Lemma~\ref{lem:BBS1-unif}.

\subsection{Gromov hyperbolic spaces $X$ with at least two points in $\partial_\eps X$.}\label{sub:GHyper_twopts-bdry}

We start by developing some lemmas as preliminary tools. 

\begin{lemma}\label{lem:1}
Let $(X, d, \mu)$ be a metric measure space that satisfies our standing assumptions, 
let $p \in [1, \infty)$, let $u:X\to\overline{\R}$ be a measurable function, 
and let $(X,d_\eps,\mu_\beta)$ be a uniformization of $(X, d, \mu)$ that is constructed 
with respect to the reference point $z_0\in X$.
If $g$ is an upper gradient of $u$ with respect to the metric $d$, then the function $g_\eps$ defined by
\[
	g_\eps(x) =\frac{1}{\rho_\eps(x)}\, g(x),
\]
where $\rho_\eps$ is as in \eqref{eq:rho_def}, is an upper gradient of $u$ with respect to the metric $d_\eps$.
Moreover, if $h_\eps$ is an upper gradient of $u$ with respect to $d_\eps$, then 
with $\rho_\eps$ as in~\eqref{eq:rho_def},
the function $h$ defined by
\[
	h(x)=\rho_\eps(x)\, h_\eps(x)
\]
is an upper gradient of $u$ with respect to $d$. Furthermore, 
if $p\le \beta/\eps$ and $g\in L^p(X, \mu)$, then $g_\eps\in L^p(X,\mu_\beta)$, 
and if $h_\eps\in L^p(X,\mu_\beta)$ and $p\ge \beta/\eps$, then $h\in L^p(X,\mu)$.
\end{lemma}

\begin{proof}
Note that the metric $d_\eps$ is of the 
form~\eqref{eq:d-eps}. Hence, when $\gamma$ is a nonconstant compact rectifiable curve in $X$, by denoting
its endpoints by $x$ and $y$, we have
\[
|u(x)-u(y)|\le \int_\gamma\! g\, \dd s=\int_\gamma\! g_\eps e^{-\eps d(z_0,\gamma(\cdot))}\, \dd s=\int_\gamma\! g_\eps\, \dd s_\eps,
\]
where $\dd s$ denotes the path integral with respect to the metric $d$ while $\dd s_\eps$ denotes the path integral with 
respect to $d_\eps$. It follows that $g_\eps$ is an upper gradient of $u$ with respect to $d_\eps$.

Moreover,
\[
\int_X\! g^p\, \dd\mu=\int_X\!g_\eps^p\, e^{-p\eps\, d(z_0,\cdot)}\, \dd\mu=\int_X\! g_\eps^p\, e^{(\beta-\eps p)d(z_0,\cdot)}\, \dd\mu_\beta,
\]
and so 
\[
\int_X g_\eps^p\, e^{(\beta-\eps p)d(z_0,\cdot)}\, \dd\mu_\beta<\infty.
\]
Hence, when $\beta\ge \eps\, p$ we know that $\int_X g_\eps^p\, \dd\mu_\beta$ is finite. A similar argument 
for $h$ completes the proof of the lemma.
\end{proof}

\begin{remark}\label{rem:p-large-betaeps}
Note that, if we wish for Lemma \ref{lem:BBS1-unif} to apply to a uniformized space $(X, d_\eps, \mu_\beta)$, 
we are not able to make $\beta>0$ as small as we like, nor are we able to make $\eps>0$ as large as we like.
However, we have the luxury of making $\beta$ as large as we like provided $\beta>\beta_0$,
and $\eps$ as small as we like provided $0<\eps\le \eps_0(\delta)$. Therefore, for
$p>\beta_0/\eps_0(\delta)$, we can choose $\beta$ and $\eps$ within these ranges so that $p=\beta/\eps$,
in which case, $h\in L^p(X,\mu)$ if and only if $h_\eps\in L^p(X,\mu_\beta)$.
\end{remark}

\begin{lemma}\label{lemma:key}
Let $(X,d,\mu)$ be an $M$-roughly $z_0$-starlike Gromov $\delta$-hyperbolic space that satisfies our standing assumptions.
Moreover, suppose that $\mu$ is uniformly locally doubling, and that
$a:=\inf_{x\in X}\mu(B(x,1))>0$.
Let $(X, d_\eps, \mu_\beta)$ be a uniformization of $X$ with base point $z_0$, where $0 < \eps \le \eps_0(\delta)$ and $\beta>\beta_0$,
let $1\le p<\infty$, let $c \in [0, \infty)$, let $f\in L^p(X,\mu)$ be bounded, and suppose that 
there exists a point $\xi\in\partial_\eps X$ such that 
\[
\liminf_{r\to 0^+}\frac{\mu_\beta(B_\eps(\xi,r)\cap\{f\ge c\})}{\mu_\beta(B_\eps(\xi,r))}=1.
\]
Then $c=0$.
\end{lemma}

\begin{proof}
We may assume that the scale of uniform doubling of $\mu$ is $R_0 = 1$. 
We denote the uniform doubling constant of $\mu$ by $C_d$.
Moreover, by Lemma~\ref{lem:BBS1-unif}, $\mu_\beta$ is globally doubling; 
we denote the respective doubling constant by $C_{\eps,\beta}$. 
In addition to this, by Lemma \ref{lem:BHK1}, we know that $(X,d_\eps)$ is a uniform domain with uniformity constant 
$A = A(\delta,\eps) > 1$.
	
Suppose to the contrary that $c > 0$. 
Then by assumption, for each $1/2>\eta>0$
there exists a $r_0>0$ so that whenever $0<r<r_0$, we
have 
\begin{equation}\label{eq:levelset1}
\frac{\mu_\beta(B_\eps(\xi,r)\cap\{f\ge {c}\})}{\mu_\beta(B_\eps(\xi,r))}\ge 1-\eta.
\end{equation}

Let $0 < r < \min \{r_0, d_\eps(z_0, \xi)\}$, and let $\gamma$ be an $A$-uniform simple curve
connecting $\xi$ to the base point $z_0$, and let $z_r$ be a point in the trajectory of $\gamma$ such that the subcurve $\gamma_z$
with end points $\xi$ and $z_r$ satisfies $r/4\le \ell_\eps (\gamma_z)\le r/2$, 
where $\ell_\eps$ denotes length with respect to $d_\eps$. 
Then $r/(4A) \le d_\eps(z_r,\partial_\eps X) \le d_\eps(z_r, \xi) \le r/2$, and so
by the fact that $A>1$, we have 
\[
 B_\eps\left(z_r,\frac{r}{4A}\right) \subset B_\eps(\xi,r).
\]
We then let $K_1, K_2 \ge 1$ be the constants of Lemma \ref{lem:unif_comparisons}, noting that both depend only on $M$ and $\eps$.
By setting $C_1 = C_1(A, M, \eps) := 2 K_1 K_2 A$ and $C_2 = C_2(A, M, \eps) := K_1 K_2 C_1$, we hence conclude that
\[
	B_\eps \left(z_r, \frac{r}{4AC_2}\right) 
	\subset B\left(z_r, \frac{1}{C_1}\right) 
	\subset B_\eps \left(z_r, \frac{r}{4A}\right)
	\subset B_\eps(\xi,r).
\]

In addition to this, we observe that in $B(z_r, 1/C_1)$, we have an oscillation estimate for the Radon-Nikodym -derivative $d\mu_\beta / d\mu = e^{-\beta d(\cdot, z_0)}$ of the form
\begin{equation}\label{eq:osc_est}
	\sup_{x\in B(z_r, 1/C_1)} \frac{d\mu_\beta}{d\mu}(x) \le e^{2\beta /C_1} \left( \inf_{x\in B(z_r, 1/C_1)} \frac{d\mu_\beta}{d\mu}(x) \right).
\end{equation}
Moreover, by the global doubling of $\mu_\beta$, 
we may select a constant $c_0 = c_0(A, C_2, C_{\eps,\beta}) \in (0, 1)$ such that
\begin{align*}
	\mu_\beta\left(B_\eps\left(z_r,\frac{r}{4AC_2}\right)\right)
	\ge c_0\, \mu_\beta(B_\eps(\xi, r)).
\end{align*}
We then select $\eta < c_0/2$ in \eqref{eq:levelset1}, in which case we have
\begin{align*}
	&\mu_\beta(B_\eps(z_r,r/(4AC_2))\cap\{f\ge 1\})\\
	&\qquad= \mu_\beta(B_\eps(z_r,r/(4AC_2)))
	- \mu_\beta(B_\eps(z_r,r/(4AC_2))\setminus\{f\ge 1\})\\
	&\qquad\ge c_0 \mu_\beta(B_\eps(\xi, r)) 
	- \mu_\beta(B_\eps(\xi, r)\setminus\{f\ge 1\})\\
	&\qquad \ge (c_0 / 2) \mu_\beta(B_\eps(\xi, r)).
\end{align*}
Thus, by using \eqref{eq:osc_est}, we obtain that
\begin{multline*}
	\frac{\mu(B(z_r,1/C_1)\cap\{f\ge 1\})}{\mu(B(z_r,1/C_1))}
	\ge \frac{1}{e^{2\beta /C_1}} 
		\frac{\mu_\beta(B(z_r,1/C_1)\cap\{f\ge 1\})}{\mu_\beta(B(z_r,1/C_1))}\\
	\ge \frac{1}{e^{2\beta /C_1}}	
		\frac{\mu_\beta(B_\eps(z_r,r/(4AC_2))\cap\{f\ge 1\})}{\mu_\beta(B_\eps(\xi, r))}
	\ge \frac{c_0}{2\,e^{2\beta /C_1}}.
\end{multline*}

Now for each positive integer $k$, we select a radius $r_k \in (0, \min (r_0, d_\eps(z_0, \xi)))$ 
such that the collection of balls $\{B(z_{r_k},1/C_1)\}_{k\in\N}$ is pairwise disjoint. 
This choice is clearly possible since $d(z_0, z_{r}) \to \infty$ as $r \to 0$. By the preceeding 
part of the proof, we have for every $k$ the estimate
\[
\mu(B(z_{r_k},1/C_1)\cap\{f\ge {c}\}) \ge \frac{c_0}{2\,e^{2\beta /C_1}} \, \mu(B(z_{r_k},1/C_1)).
\]
By the uniformly local doubling property of $\mu$ and by the hypothesis on $\mu$ and $a$, we have that
\[
\mu(B(z_{r_k},1/C_1)\cap\{f\ge {c}\}) \gtrsim \frac{c_0 a}{e^{2\beta /C_1}},
\]
with the comparison constant dependent only on $C_1$ and $C_d$. Hence, $\mu(\{f\ge 
c\})=\infty$, which is not possible when $f\in L^p(X)$ and $c > 0$. It follows that we must have $c=0$.
\end{proof}

\begin{cor}\label{cor:integrable-trace-zero}
Under the hypotheses on $(X,d,\mu)$ in the statement of Lemma~\ref{lemma:key}, if $f\in L^p(X,\mu)$  
for some $1\le p<\infty$, $\xi\in\partial_\eps X$, and $c\in\R$ 
such that 
\[
\lim_{r\to 0^+}\jint_{B_\eps(\xi,r)}\!|f-c|\, \dd\mu_\beta=0,
\]
then $c=0$.
\end{cor}

\begin{proof}
Again, suppose towards contradiction that $c\ne 0$.  By replacing $f$ with $-f$ if necessary,
we assume without loss of generality that $c>0$. Then as
$(f-c)_+\le |f-c|$ and $(f-c)_-\le |f-c|$, we have that 
\[
\lim_{r\to 0^+}\jint_{B_\eps(\xi,r)}\!(f-c)_{\pm}\, \dd\mu_\beta=0.
\]
It follows that for each positive real number $\eta$, we must have 
\[
\lim_{r\to 0^+}\frac{\mu_\beta(B(\xi,r)\cap\{(f-c)_{\pm}\ge \eta\})}{\mu_\beta(B(\xi,r))}=0.
\]
Therefore,
\[
\lim_{r\to 0^+}\frac{\mu_\beta(B(\xi,r)\cap\{c-\eta\le f\le c+\eta\})}{\mu_\beta(B(\xi,r))}=1,
\]
and so, choosing $\eta=c/2$, we get
\[
\lim_{r\to 0^+}\frac{\mu_\beta(B(\xi,r)\cap\{f\ge c/2\})}{\mu_\beta(B(\xi,r))}=1,
\]
violating the conclusion of Lemma~\ref{lemma:key}. Thus, our assumption is wrong, and hence $c=0$.
\end{proof}

\begin{theorem}\label{prop:two-pts-large-p}
Let $(X,d,\mu)$ be a Gromov hyperbolic space that satisfies our standing assumptions and is 
of locally uniformly $p$-controlled geometry for some $1\leq p<\infty$.
Suppose also that there exists a positive number $a$ such that for each $x\in X$, we have $\mu(B(x,1))\ge a$.
Let $(X, d_\eps, \mu_\beta)$ be a uniformization of $X$ with base point $z_0$, where $0 < \eps \leq \eps_0(\delta)$ and $\beta>\beta_0$.
Suppose that $\partial_\eps X$ has at least two points and that $p\ge \beta/\eps$. Then there exists a function 
$f\in D^{1,p}(X,d,\mu)$ such that
for each $c\in\R$ we have $f-c\not\in L^p(X,\mu)$.
\end{theorem}

\begin{proof}
We fix two distinct points $\xi,\zeta\in\partial_\eps X$, set $\tau=d_\eps(\xi,\zeta)/10$, and consider the function
$f:X\to\R$ given by $f(x):=d_\eps(x,B_\eps(\xi,\tau))$. Then $f=0$ on $B_\eps(\xi,\tau)$ and $f\ge 8\tau$ on $B_\eps(\zeta,\tau)$.
We next show that
\begin{enumerate}
\item[{\bf{(a)}}] $f\in D^{1,p}(X,d,\mu)$, and 
\item[{\bf{(b)}}] $f-c\not\in N^{1,p}(X,d,\mu)$ for each $c\in\R$.
\end{enumerate}
As the constant function $h_\eps\equiv 1$ is an upper gradient of $f$ in $(X,d_\eps)$,
and $\mu_\beta(X)<\infty$,  it follows by Lemma \ref{lem:1} that the function $h$ given by 
$h(x)=e^{-\eps\, d(z_0,x)}$ is an upper gradient of
$f$ in $(X,d,\mu)$ with $h\in L^p(X,\mu)$.
Thus, claim~{\bf(a)} follows. Note that $f\ge 8\tau>0$ in $B_\eps(\zeta,\tau)$, and $f=0$ in $B_\eps(\xi,\tau)$.
Hence, for each $c\in\R$
we have that either $f - c \ge 4\tau \ne 0$ in $B_\eps(\zeta,\tau)$, or else $c - f \ge 4\tau \ne 0$ in 
$B_\eps(\xi,\tau)$, and so by 
Lemma~\ref{lemma:key}, we know that $f-c\not\in L^p(X,\mu)$.
\end{proof}

\begin{example}\label{ex:quasihyperbolic_ball}
Consider the space $(X, d, \mu) := (\B^n, k, k^n d\mathcal{H}^n)$, where $\B^n \subset \R^n$ is the unit ball, 
$k$ is the quasihyperbolic metric, and $\mathcal{H}^n$ is the $n$-dimensional Hausdorff measure on $\mathbb{B}$ 
induced by the Euclidean metric. Then the space $X$ is Gromov hyperbolic and is of uniformly locally bounded 
geometry. The space $X$ is also bilipschitz equivalent to the standard hyperbolic space $\H^n$; we will show 
in Lemma \ref{lem:Hn_counterexamples} that $D^{1,p}(\H^n) \ne W^{1,p}(\H^n) + \R$ when $p > n-1$, which then 
implies that $D^{1,p}(X) \ne W^{1,p}(X) + \R$ when $p > n-1$.
\end{example}

\subsection{$p$-parabolic Gromov hyperbolic spaces} \label{sub:GHyper_p-para}
In contrast to Example \ref{ex:quasihyperbolic_ball}, we then provide an example of a space $(X, k, \mu)$ 
equipped with the quasihyperbolic metric such that $D^{1,p}(X, k, \mu) \ne N^{1,p}(X, k, \mu) + \R$ for all $p \in [1, \infty)$. 

\begin{example}\label{ex:puncture}
The punctured closed ball $X := \overline{\mathbb{B}}\setminus\{0\}$ 
in $\R^n$, equipped with the quasihyperbolic metric $k$ and measure $\mu$, is $p$-parabolic. We use Lemma~\ref{lem:parab-volume}
to verify this. Indeed, for every $t>0$, we find a point $x_t \in X$ such that 
$k(x_t,\mathbb{S}^{n-1})=\log(1/|x_t|)=t$. Therefore,
if we fix a point $x_0 \in \S^{n-1}$ and denote $V(t)=\mu(B(x_0,t))$, we have for sufficiently large $t$ that
\[
	V(t) \approx \int_{|x_t|}^1 \tfrac{1}{s^n}\, s^{n-1}\, ds=\log(1/|x_t|)=t,
\]
with the comparison constant dependent only on the dimension $n$. Hence, 
\[
\int_1^\infty\! \left(\frac{t}{V(t)}\right)^{1/(p-1)}\, dt=\infty,
\]
and thus, by Lemma~\ref{lem:parab-volume}, we know that $X$ is $p$-parabolic for each $p>1$.

Let $u:X\to\R$ be given by $u(x)=[\log(e/|x|)]^q$ for some fixed non-zero real number $q$. Then the 
Euclidean upper gradient $g$ of $u$ is
comparable to $[\log(e/|x|)]^{q-1}\, \tfrac{1}{|x|}$, and so the quasihyperbolic upper gradient $g_k$ is 
comparable to $[\log(e/|x|)]^{q-1}$. Thus, with 
$\mathcal{H}^n$ the $n$-dimensional Hausdorff measure on $X$ induced by the Euclidean metric, we have 
\[
\int_X\!g_k^p\, d\mu \approx \int_X\![\log(e/|x|)]^{pq-p}\, \frac{1}{|x|^n}\, 
\dd\mathcal{H}^n(x)
\approx \int_0^1\![\log(e/t)]^{pq-p}\, \frac{1}{t}\, dt=\int^\infty_1 \! \tau^{pq-p}\, d\tau.
\]

The above integral is finite if and only if $pq-p<-1$, that is, $q<1-\tfrac1p$. On the other hand,
\[
\int_Xu^p\, d\mu \approx \int_0^1[\log(e/t)]^{pq}\, \frac{1}{t}\, dt=\int^\infty_1 \tau^{pq}\, 
d\tau,
\]
and this integral is infinite if and only if $pq>-1$, that is, $q>-1/p$. Hence choosing $q$ such that $0\ne q\in (-1/p,\, 1-1/p)$ gives a
function in $D^{1,p}(X,d,\mu)$ that is not in $N^{1,p}(X,d,\mu)+\R$.
\end{example}

Example~\ref{ex:puncture} above motivates the next lemma and the subsequent theorem.

\begin{lemma}\label{lem:singeton-bdy-parabolic}
	Suppose that $(X,d,\mu)$ satisfies our standing assumptions, is an $M$-roughly $z_0$-starlike Gromov $\delta$-hyperbolic 
	space, and has uniformly locally $p$-controlled geometry 
	for some $1<p<\infty$.
	Let $(X, d_\eps)$ be a uniformization of $X$ with base point $z_0$, 
	where $0 < \eps \le \eps_0(\delta)$.
	If $\sup_{x \in X} \mu(B(x, 1)) < \infty$, and if
	$\partial_\eps X$ has only one point,
	then $X$ is $p$-parabolic.
\end{lemma}

\begin{proof}

We fix an arbitrary $\beta > \beta_0$. 
We use $\xi$ to denote the single point in $\partial_\eps X$, and for each
positive integer $j$, we set $\Om_j:=\{x\in X\, :\, d_\eps(x,\xi)>2^{-j}\}$. Our objective is to show that 
the sets $\Omega_j$ satisfy \eqref{eq:p_para_cond}, after which the claim follows by Lemma \ref{lem:parab-volume-sum}.

We begin by observing that if $K_1 = K_1(M, \eps)$ is the constant of~\eqref{eq:rho_d_comparison} from
Lemma~\ref{lem:unif_comparisons}, then by utilizing the fact that $\Omega_{j+1}\setminus\Omega_j \subset B_\eps(\xi, 2^{-j+1})$, we have
\begin{multline*}
	\mu(\Om_{j+1}\setminus\Om_j)=\int_{\Om_{j+1}\setminus\Om_j} \frac{1}{\rho_\beta} \, d\mu_\beta
	\le K_1^{\frac{\beta}{\eps}} \int_{\Om_{j+1}\setminus\Om_j}\frac{1}{d_\eps(x,\xi)^{\beta/\eps}}\, d\mu_\beta\\
	\le K_1^{\frac{\beta}{\eps}} 2^{\frac{(j+1)\beta}{\eps}} \mu_\beta( \Omega_{j+1}\setminus\Om_j )
	\le K_1^{\frac{\beta}{\eps}} 2^{\frac{(j+1)\beta}{\eps}} \mu_\beta( B_\eps(\xi, 2^{-j+1}) ).
\end{multline*}
Next, for all sufficiently large $j$, we may select a point $z_j \in X$ such that 
$d_\eps(z_j, \xi) = 2^{-j}$. It follows by~\eqref{eq:boundary_to_inside_comparison} from Lemma~\ref{lem:unif_comparisons} 
with $r = 2^{-j+1}$ and $a_0 = 1/4$ that
\[
	2^{\frac{(j+1)\beta}{\eps}} \mu_\beta(B_\eps(\xi, 2^{-j+1})) \le K_3 \mu \left( B\left( z_j, \frac{2^{-j-1}}{\rho_\eps(z_j)} \right) \right),
\]
where $K_3$ depends on all of the parameters $\delta$, $M$, $\eps$, $\beta$, $R_0$ and $C_d$. We then again 
apply~\eqref{eq:rho_d_comparison} along with $d_\eps(z_j, \xi) = 2^{-j}$ to obtain that $2^{-j-1}/\rho_\eps(z_j) \le K_1/2$. By combining the above deductions with the locally $p$-controlled geometry of $X$ and the assumption that $\sup_{x \in X} \mu(B(x, 1)) < \infty$, we conclude that
\begin{equation}\label{eq:sup_bounded}
	\sup_{j \in \N} \mu(\Om_{j+1}\setminus\Om_j) < \infty.
\end{equation}

Next, let  $x_j \in\overline{\Om_j}$ and $y_j \in X\setminus\Om_{j+1}$
be such that $d(x_j, y_j)=\Delta_j = \dist(\Omega_j, X \setminus \Omega_{j+1})$, and 
let $\gamma$ be a geodesic in $(X,d)$ with end points $x_j,y_j$. Note that apart from the two end points,
$\gamma$ is contained in $\Om_{j+1}\setminus \Om_j$.
Then, since we have $\dist_\eps(\overline{\Om_j},X\setminus\Om_{j+1}) \ge 2^{-j-1}$ for all $j$ sufficiently large that $\Omega_j \ne \emptyset$, it follows with another use of Lemma~\ref{lem:unif_comparisons} part~\eqref{eq:rho_d_comparison} that
\[
2^{-j} \le 2d_\eps(x_j,y_j)\le 2 \int_\gamma\! \rho_\eps(\gamma(t))\, \dd t \le 2K_1 
\int_\gamma d_\eps(\gamma(t),\xi)\, dt \le 2^{-j+1} K_1\, \Delta_j.
\]
for all such $j$. That is, we have $\Delta_j \ge 1/(2K_1) > 0$ for all sufficiently large $j$. By combining this with \eqref{eq:sup_bounded}, we conclude that
\[
\sum_{j=1}^\infty\left(\frac{\Delta_j^p}{\mu(\Om_{j+1}\setminus\Om_j)}\right)^{1/(p-1)} =\infty,
\]
and thus, by Lemma~\ref{lem:parab-volume-sum}, the desired conclusion follows.
\end{proof}

\begin{theorem}\label{thm:para}
If $(X,d,\mu)$ is a metric measure space that satisfies our standing assumptions
and is $p$-parabolic for some $1\leq p<\infty$, 
then $D^{1,p}(X,d,\mu)\ne N^{1,p}(X,d,\mu)+\R$.
\end{theorem}

\begin{proof}
Recall that it is our standing assumption throughout this paper that $\mu(X)=\infty$.

Fix $x_0\in X$ and $r_1=1$. As $X$ is $p$-parabolic, we can find a function $u_1\in N^{1,p}(X,d,\mu)$ 
such that $u_1=1$ on $B(x_0,r_1)$ and $u_1$ has compact support in $X$ with $\int_Xg_{u_1}^p\, 
d\mu<2^{-p}$. Let $R_1>r_1$ such that $u_1=0$ in $X\setminus B(x_0,R_1)$.
We next choose $r_2>R_1$ such that $\mu(B(x_0,r_2)\setminus B(x_0,R_1))\ge 1$; this is possible because 
$\mu(X)=\infty$. Then we can find $u_2\in N^{1,p}(X,d,\mu)$ such that $u_2=1$ on $B(x_0,r_2)$ and $u_2$ 
is compactly supported in $X$ with $\int_Xg_{u_2}^p\, d\mu<2^{-2p}$.  Let $R_2>r_2$ such that
$u_2=0$ in $X\setminus B(x_0,R_2)$. Continuing inductively this way, having chosen $r_k<R_k$ and a function $u_k\in N^{1,p}(X,d,\mu)$ 
with $u_k=1$ on $B(x_0,r_k)$, $u_k=0$ on $X\setminus B(x_0,R_k)$, and $\int_Xg_{u_k}^p\, d\mu<2^{-kp}$, we choose 
$r_{k+1}>R_k$ so that $\mu(B(x_0,r_{k+1})\setminus B(x_0,R_k))\ge 1$, and find $u_{k+1}\in N^{1,p}(X,d,\mu)$ and $R_{k+1}>r_{k+1}$ such that
$u_{k+1}=1$ on $B(x_0,r_{k+1})$, $u_{k+1}=0$ in $X\setminus B(x_0,R_{k+1})$, and $\int_Xg_{u_{k+1}}^p\, d\mu<2^{-(k+1)p}$.

Now setting for each positive integer $k$ the function $v_k=1-u_k$, and then setting $f=\sum_{k=1}^\infty v_k$,
we get a function $f\in D^{1,p}(X,d,\mu)$. However, $f\not\in N^{1,p}(X)+\R$. To see this, note that for each positive integer $n$,
the set $\{f\ge n\}$ contains the union of pairwise disjoint annuli $\mu(B(x_0,r_{k+1})\setminus B(x_0,R_k))$, $k\in\N$, each of measure at 
least $1$, and so $\{f\ge n\}$ has infinite measure. So for each $c\in\R$ we also have that $\{f-c\ge n\}$ also has 
infinite measure for each
positive integer $n$; thus $f\not\in L^p(X,\mu)+\R$.
\end{proof}

Thus, combining the previous two results yields the following immediate corollary.

\begin{cor}\label{cor:one_bdry_point}
	Suppose that $(X,d,\mu)$ satisfies our standing assumptions, is a roughly starlike 
	Gromov $\delta$-hyperbolic space, and has locally $p$-controlled geometry for some $1<p<\infty$.
	Let $(X, d_\eps)$ be a uniformization of $X$, where $0 < \eps < \eps_0(\delta)$. If $\sup_{x \in X} 
	\mu(B(x, 1)) < \infty$ and if $\partial_\eps X$ has only 
	one point, then $D^{1,p}(X,d,\mu)\ne N^{1,p}(X,d,\mu)+\R$.
\end{cor}

\section{Results based on ends at infinity}

Next, we investigate how the question of whether 
$D^{1,p}(X) = 
N^{1,p}(X) + \R$ is affected by the number of $p$-parabolic and $p$-hyperbolic ends of the space 
$(X, d, \mu)$. 

Before proceeding, a comparison with the results in Section~\ref{Sec:4} is in 
order. First, we note that as discussed in subsection~\ref{sub:hypunif}, it is possible 
for the uniformized boundary $\partial_\eps X$ of a Gromov hyperbolic space $X$  to contain multiple 
points even if $X$ has only one end. Thus, the results of subsection \ref{sub:GHyper_twopts-bdry} 
differ from the results we show for spaces with multiple ends. 
Moreover, in subsection~\ref{sub:GHyper_p-para} we considered the $p$-parabolicity of the 
\emph{entire} 
space $X$, instead of the $p$-parabolicity of the ends of $X$. It is possible for a space $X$ to be 
$p$-hyperbolic yet only have $p$-parabolic ends; see e.g.\ \cite[Example~8.5]{BBS-B} 
and~\cite[Example~7.2]{BBS-A}.

However, some connections between the parabolicity or hyperbolicity of the space itself and 
its ends do apply, as evidenced by the following lemma.

\begin{lemma}\label{lem:1paraend-paraspace}
	Suppose that $(X, d, \mu)$ satisfies our standing assumptions and 
	has uniformly locally $p$-controlled geometry. If $X$ 
	has finitely many ends at infinity, all of which are $p$-parabolic, then $X$ is $p$-parabolic. 
	Conversely, 
	if $X$ has at least one  $p$-hyperbolic end at infinity, then $X$ is $p$-hyperbolic.
\end{lemma}

Before starting the proof of Lemma \ref{lem:1paraend-paraspace} proper, we prove two sub-lemmas. For the following two results, we do not assume our usual standing assumptions, since their statements are purely metric.

\begin{lemma}\label{lem:fin_conn_components}
	Let $(X, d)$ be a proper, connected, unbounded metric space that is locally uniformly 
	quasiconvex with associated constants $R, C$. 
	Then for every $x_0 \in X$ and $r > 0$, only finitely many connected components of $X \setminus 
	B(x_0, r)$ intersect the set $X \setminus B(x_0, r + 2CR)$.
\end{lemma}
 
\begin{proof}
	By the hypothesis, $X$ is locally path connected, and thus locally connected. In 
	particular, connected components of open sets in $X$ are open. We fix $x_0 \in X$ and let $r > 0$. Note 
	that if $E_1$ and $E_2$ are two different components of $X \setminus B(x_0, r)$ that intersect $X 
	\setminus B(x_0, r + 2CR)$, then
	\[
		\dist(E_1\setminus B(x_0,r+2CR), E_2\setminus B(x_0,r + 2CR))\ge R,
	\]
	as otherwise $E_1\setminus B(x_0,r+2CR)$ and  $E_2 \setminus B(x_0,r+2CR)$ could be connected with a 
	rectifiable curve in $X\setminus B(x_0,r)$, violating the fact that $E_1$ and $E_2$ are distinct components.
	
	Moreover, if $E$ is a component of $X\setminus B(x_0,r)$ that meets $B(x_0,r+2CR)$, then $E$ must meet 
	$S := \{y\in X\, :\, d(x_0,y)=r+2CR\}$. Indeed, suppose towards contradiction that this is not the case. 
	Then since $E$ is connected and meets $X\setminus B(x_0,r+2CR)$, $E$ does not meet $\overline{B}(x_0, r + 2CR)$. 
	It follows that $E$ is a connected component of both $X \setminus B(x_0, r + 2CR)$ and 
	$X \setminus \overline{B}(x_0, r + 2CR)$. Since $X$ is locally connected, it follows that $E$ 
	is both closed and open. This is a contradiction, since $X$ is connected and $E$ is a nonempty 
	subset of $X$ that does not contain $x_0$. Thus, we conclude that $E$ intersects $S$. 
	
	Hence, the intersections $E \cap S \neq \emptyset$, where $E$ ranges over connected components 
	of $X \setminus B(x_0, r)$ that meet $X \setminus B(x_0, 2CR)$, form a $R$-separated family of sets. 
	Since $X$ is proper, $S$ is compact, and thus that the number of such components $E$ is indeed finite.
\end{proof}

\begin{lemma}\label{lem:one_unbdd_comp}
	Let $(X, d)$ satisfy the assumptions of Lemma \ref{lem:fin_conn_components}, and let $x_0 \in X$. Suppose that $(X, d)$ has exactly $k$ ends at infinity, where $k$ is a positive integer. Then for every $r > 0$, the set $X \setminus B(x_0, r)$ has at most $k$ unbounded components. Moreover, for sufficiently large $r$, the set $X \setminus B(x_0, r)$ has exactly $k$ unbounded components.
\end{lemma}

\begin{proof}
	We begin by noting that if $x_0 \in X$, $0 < r_1 < r_2$, and $E$ is an unbounded component of $X\setminus B(x_0,r_1)$, then every connected component of $X\setminus B(x_0,r_2)$ is either contained in $E$, or does not meet $E$. Moreover, there exists an unbounded component $E'$ of $X\setminus B(x_0,r_2)$ that is contained in $E$, as otherwise there would necessarily exist infinitely many bounded 
	components $E(i) \subset E$, $i\in\N$, of $X\setminus B(x_0,r_2)$ that meet the set $X \setminus B(x_0, r_2 + 2CR)$, which in turn would contradict Lemma \ref{lem:fin_conn_components}. As a consequence of this, we observe that the number of unbounded components of $X \setminus B(x_0, r)$ is non-decreasing with respect to $r$.
	
	For the first claim, suppose towards contradiction that $X\setminus B(x_0,r)$ has $(k+1)$ distinct unbounded components, $E_0^{1}, \dots, E_0^{k+1}$. Then there exists at least one unbounded component $E \subset E_0^1$ of $X\setminus B(x_0,r+1)$; we fix one such component and denote it $E_{1}^1$. We continue this process inductively to find connected components $E_{j}^1$ of $X\setminus B(x_0,r+j)$ with $E_{j+1}^1 \subset E_{j}^1$, and then perform the same argument on every other $E_{0}^{i}$ to find a sequence $E_{j}^{i}$. Now, $\{E_{j}^{i}\}_j$ are $(k+1)$ different ends of $X$, contradicting the assumption that $X$ has only $k$ ends. The first claim hence follows.
	
	For the second claim, since the number of unbounded components of $X \setminus B(x_0, r)$ is bounded, integer-valued, and 
	non-decreasing with respect to $r$, there exists $l \in \{1, \dots, k\}$ and $r_0 = r_0(x_0) > 0$ such that $X \setminus B(x_0, r)$ 
	has exactly $l$ components for all $r \ge r_0$. Moreover, the sequences of nested unbounded components of the sets 
	$X \setminus B(x_0, r_0 + j)$ form $l$ distinct ends $\{E_j^1\}_j, \dots, \{E_j^l\}_j$. Now, let $\{F_j\}_j$ be an arbitrary end of 
	$X$, with corresponding compact sets $K_j$. Note that every $K_j$ is contained in some ball $B(x_0, r_0 + i)$ due to 
	compactness of $K_j$, and that every $B(x_0, r_0 + i)$ meets only finitely many $F_j$ since $X$ is proper. From these 
	observations, it follows that $\{F_j\}_j$ must equal one of the ends $\{E_j^i\}_j$
	in the sense of Definition~\ref{def:ends}. Since $X$ has exactly $k$ ends by assumption, we thus must have $l = k$.
\end{proof}

We are now ready to complete the proof of Lemma~\ref{lem:1paraend-paraspace}.

\begin{proof}[Proof of Lemma \ref{lem:1paraend-paraspace}]
	Note that by Lemma~\ref{lem:bdd_geom_is_ULQC}, $X$ is uniformly locally quasiconvex, with 
	constants $R, C$ dependent only on the data associated with the uniformly locally
	$p$-controlled geometry of 
	$X$. 

	Suppose first that $X$ has $k$ ends, all of which are $p$-parabolic. We fix $x_0 \in X$ and $r > 0$. 
	By Lemma~\ref{lem:one_unbdd_comp}, there exists an $r_0 \ge r$ such that the set $X\setminus 
	B(x_0,r_0+n)$ has exactly $k$ unbounded components for every non-negative integer $n$. Thus, the $k$ 
	distinct ends of $X$ can be expressed in the form $\{E_n^1\}_n, \dots, \{E_n^k\}$, where every 
	$E_n^i$ is an unbounded component of $X\setminus B(x_0,r_0+n)$. Additionally, by Lemma 
	\ref{lem:fin_conn_components}, for every $n \in \N$, all the bounded components of $X\setminus 
	B(x_0,r_0+n)$ lie within $B(x_0,T_n)$ for some $T_n\ge r_0+n+ 2C R$. Now, if $\Gamma_n^i$ is the 
	collection of all curves in $X$ with one endpoint in $B(x_0,r)$ and the other in $E_n^i$, we know by 
	assumption that $\lim_{n\to\infty}\Mod_p(\Gamma_n^i)=0$ for each $i$. Moreover, for each $n \in \N$, 
	we know that $\Gamma(B(x_0,r), X\setminus B(x_0,T_n))\subset \bigcup_{i=1}^k \Gamma_n^i$, and hence 
	we also have that $\lim_{n\to\infty}\Mod_p(\Gamma(B(x_0,r), X\setminus B(x_0,T_n)))=0$. Therefore, 
	by Lemma~\ref{lem:mod-cap}, we know that $X$ is $p$-parabolic.
	
	Now, suppose that $X$ has a $p$-hyperbolic end at infinity, which we denote  $\{F_j\}_j$.
	Then we know that $\lim_j\Mod_p(\Gamma_j)>0$, where $\Gamma_j$ is the collection of all curves in 
	$X$ with one endpoint in  the set $K_1$ associated with this end  and the other 
	endpoint in $F_j$.	 We fix $x_0 \in X$ and $r > 0$ with $x_0\in K_1\subset B(x_0,r)$, and 
	denote $r_j=\dist(x_0,F_j)$. We then  see that $\Gamma_j\subset\Gamma(B(x_0,r),X\setminus 
	B(x_0,r_j))$, and hence $\lim_{j\to\infty} \Mod_p(\Gamma(B(x_0,r),X\setminus B(x_0,
	r_j)))>0$. Therefore, by another use of Lemma~\ref{lem:mod-cap}, we have that $X$ is 
	$p$-hyperbolic.
\end{proof}

Next, we show that spaces with multiple ends have $D^{1,p}(X)\ne N^{1,p}(X)+\R$ 
under relatively mild assumptions.

\begin{theorem}\label{thm:2ends}
Let $(X, d, \mu)$ be a metric measure space that satisfies our standing assumptions. 
Suppose that $X$ is a length space, $X$ has at least two ends at $\infty$, 
and there exists an $r > 0$ such that $a := \inf_{x \in X} \mu(B(x, r)) > 0$. Then 
for each $1\le p<\infty$, we have $D^{1,p}(X)\ne N^{1,p}(X)+\R$.
\end{theorem}

\begin{proof}
Let $\{F_j\}_j$ and $\{E_j\}_j$ be two distinct ends at infinity. We may assume without loss of
generality that $\overline{E}_1\cap\overline{F}_1$ is empty. We first claim that for each $j \in \N$, 
we have $\mu(E_j) = \mu(F_j) = \infty$. Indeed, let $j \in \N$, and let $K_j$ be the corresponding 
compact set of $E_j$. Since $E_j$ is unbounded, there exists an $x \in E_j$ such that $\dist(x, K_j) > r$. 
Now, since $X$ is a length space, $B(x, r)$ is connected, and consequently $B(x, r) \subset E$. Moreover, 
since $X$ is proper, there exists a $j' \in \N$ such that $E_{j'} \cap B(x, r) = \emptyset$. 
It follows that $\mu(E_j) \ge a + \mu(E_{j'})$. By applying this argument inductively to $E_{j'}$, 
we conclude that $\mu(E_j) = \infty$, and $\mu(F_j) = \infty$ follows analogously.

Next, let $S := X\setminus F_1$, noting that $E_1\subset S$, and let $\tau=\dist(S, \overline{F}_2) > 0$. Moreover, observe that if $K_1$ is the corresponding compact set of $F_1$ in the end $\{F_j\}_j$, then since $X$ is a length space, it follows that $\dist(S, x) = \dist(K_1, x)$ for all $x \notin S$. We define a function $u:X\to\R$ by 
\[
u(x)=\min\{1,\, \dist(x, S)/\tau\}.
\]
The function $u$ is $\tau^{-1}$-Lipschitz continuous on $X$ with an upper gradient 
$g=\tau^{-1}\, \chi_{V}$, where $V := \{x \in X \mid \dist(x, K_1) \le \tau\}$. Since $X$ is proper and 
$K_1$ is compact, $V$ is compact. Thus, since balls in $X$ have finite measure, we conclude that 
$\mu(V) < \infty$, and therefore $g\in L^p(X)$. 
Hence, $u\in D^{1,p}(X)$. On the other hand, $u=1$ on $F_2$ and $u=0$ on $E_1$; 
since both of these sets have infinite measure, we have $u-c \notin L^p(X)$ for all $c \in \R$.
\end{proof}

Note that by combining the previous result with our results on $p$-parabolic spaces, we obtain the 
following corollary. 

\begin{cor}\label{cor:parabolic_end}
	Let $(X, d, \mu)$ be a metric measure space that satisfies our assumptions and is of 
	uniformly locally $p$-controlled geometry, where $1 \le p < \infty$. Suppose that $X$ is a 
	length space and $\inf_{x \in X} \mu(B(x, 1)) > 0$. If $X$ has a $p$-parabolic end, then $D^{1,p}(X)\ne N^{1,p}(X)+\R$.
\end{cor}
\begin{proof}
	If $X$ has more than one end at infinity, then the claim follows from Theorem \ref{thm:2ends}. 
	Thus, it remains to consider the case when $X$ has exactly one end at infinity, which is $p$-parabolic by 
	assumption. In this case, it follows from Lemma \ref{lem:1paraend-paraspace} that $X$ is $p$-parabolic, and the 
	claim hence follows by Theorem \ref{thm:para}.
\end{proof}

Given the above results in conjunction with the results of Section~\ref{Sec:4}, the main remaining 
question for Gromov hyperbolic spaces $(X, d, \mu)$ is what happens when $X$ has only 
one end at $\infty$, and that end is $p$-hyperbolic with $p \le \beta_0/\eps_0$. Note that if 
$1<p<\beta/\eps$ and $u\in D^{1,p}(X)$ 
has an upper gradient $g\in L^p(X)$, then the function $g_\eps$ given
by $g_\eps(x)=e^{-\eps d(x,z_0)}g(x)$ is an upper gradient of $u$ in $(X,d_\eps)$, and as $p<\beta/\eps$, we have that 
$g_\eps\in L^p(X,\mu_\beta)$; that is, $u\in N^{1,p}(X,d_\eps,\mu_\beta)$. As $(X,d_\eps,\mu_\beta)$ supports a global
$p$-Poincar\'e inequality (see~\cite[Theorem~1.1]{BBS1}), we also have 
from~\cite[Proposition~7.1]{AS} that $u$ has an extension to $\partial_\eps X$
that is in $N^{1,p}(\overline{X}_\eps,d_\eps,\mu_\beta)$; this extension is given via Lebesgue point traces.

\section{Identifying functions in Dirichlet spaces that belong to $N^{1,p}(X)$ via behavior along 
paths}\label{Sec:N1p-tame}

Our objective in this section is to prove the second part of Theorem \ref{thm:main2}. In order to 
achieve this objective, we first show that under our standing assumptions, spaces of uniformly locally 
$p$-controlled geometry satisfy the $MEC_p$ property from Definition \ref{def:mecp}. Note that for 
spaces of globally $p$-controlled geometry, this is shown in \cite[Theorem 9.3.4]{HKSTbook}.

We begin by showing two lemmas that are used in the proof of this fact. The first one is a 
convenient interaction of lower semicontinuity and path integrals that occurs multiple times in our 
arguments.

\begin{lemma}\label{lem:lsc_for_path_conv}
	Let $(X, d, \mu)$ be a metric measure space that satisfies our standing assumptions. Let $\gamma_i 
	\colon [0, \infty) \to X$, $i \in \N$, be a sequence of paths with $\ell(\gamma_i) < \infty$, and 
	suppose that $\gamma_i$ are parametrized first by arclength and then as constant. Suppose that 
	$\gamma \colon [0, \infty) \to X$ is such that $\gamma_i \to \gamma$ locally uniformly. Then for 
	every lower semicontinuous function $\rho \colon X \to [0, \infty]$, we have
	\[
		\int_\gamma \rho \, \dd s \le \liminf_{i \to \infty} \int_{\gamma_i} \rho \, \dd s.
	\]
\end{lemma}

\begin{proof}
	There exists an $M \in [0, \infty]$ and a subsequence 
	$\gamma_{i_j}$ of $\gamma_i$ such that $\lim_{j \to \infty} \ell(\gamma_{i_j}) = M$ and
	\[
		\lim_{j \to \infty} \int_{\gamma_{i_j}} \rho \, \dd s 
		= \liminf_{i \to \infty} \int_{\gamma_i} \rho \, \dd s.
	\]
	Due to the 
	way that $\gamma_i$ have been parametrized, we have that $\gamma$ is 1-Lipschitz, and that $\gamma$ 
	is constant on $[M, \infty)$ if $M$ is finite. Let $m \in [0, M)$. We may then use the $1$-Lipschitz property of 
	$\gamma$, Fatou's lemma, and the lower semicontinuity of $\rho$ to obtain that
	\begin{multline*}
		\int_0^m (\rho \circ \gamma)(t) \abs{\gamma'(t)} \, \dd t
		\le \int_0^m (\rho \circ \gamma)(t) \, \dd t\\
		\le \int_0^m \liminf_{j \to \infty} (\rho \circ \gamma_{i_j})(t) \, \dd t
		\le \liminf_{j \to \infty} \int_0^m (\rho \circ \gamma_{i_j})(t) \, \dd t.
	\end{multline*}
	Since $\lim_{j \to \infty} \ell(\gamma_{i_j}) = M > m$, and since $\gamma_i$ have been parametrized 
	by arclength, it follows that for sufficiently large $j$, we have $\smallabs{\gamma_{i_j}'(t)} = 1$ 
	for 
	a.e.\ $t \in [0, m]$. Hence,
	\[
		\liminf_{j \to \infty} \int_0^m (\rho \circ \gamma_{i_j})(t) \, \dd t
		\le \lim_{j \to \infty} \int_{\gamma_{i_j}} \rho \, \dd s
		 = \liminf_{i \to \infty} \int_{\gamma_{i}} \rho \, \dd s.
	\]
	Since $\gamma$ is constant on $[M, \infty)$, it hence follows that
	\[
		\int_{\gamma} \rho \, \dd s
		= \lim_{m \to M^{-}} \int_0^m (\rho \circ \gamma)(t) \abs{\gamma'(t)} \, \dd t
		\le \liminf_{i \to \infty} \int_{\gamma_{i}} \rho \, \dd s,
	\]
	completing the proof.
\end{proof}
	
\begin{lemma}\label{lem:upper-is-upper}
Let $(X,d,\mu)$ be a metric measure space  that satisfies our standing assumptions, 
and let $\rho$ be a non-negative
Borel measurable function on $X$ such that $\rho\in L^p(X)$. Fix $k \in [0, \infty)$ and a 
nonempty closed set $A\subset X$, and define
$u:X\to\R$ and $v:X\to\R\cup\{\infty\}$ by
\begin{align*}
u(x)&=\inf\bigg\lbrace k,\int_\gamma\!\rho\, ds\, :\, \gamma\text{ is rectifiable and connects }A\text{ 
to }x\bigg\rbrace,\\
v(x)&=\inf\bigg\lbrace \int_\gamma\!\rho\, ds\, :\, \gamma\text{ is rectifiable and connects }A\text{ to 
}x\bigg\rbrace.
\end{align*}
Then $u$ and $v$ are measurable, and
$\rho$ is an upper gradient of $u$. Moreover, if there exists a constant $C\ge 0$ 
with $\rho\le C$, and every point in $X$ can be connected to $A$ by a rectifiable 
curve, then $\rho$ is an upper gradient of $v$ as well.
\end{lemma}

\begin{proof}
Since $A$ is a nonempty closed set, we know that both $u$ and $v$ are measurable functions on $X$ 
by~\cite[Theorem~9.3.1]{HKSTbook}. 
We first prove the following statement, which we use to prove both of our remaining claimed results: if $x, y 
\in X$ are such that $v(x) < \infty$, and if $\gamma$ is a rectifiable curve with end points $x$ and 
$y$, then
\begin{equation}\label{eq:common_subclaim}
	\abs{v(x) - v(y)} \le \int_\gamma\!\rho\, \dd s.
\end{equation}

For this, fix $x$, $y$, and $\gamma$, and suppose that $v(x) < \infty$. We may assume that 
$\int_\gamma \rho\, \dd s < \infty$, as the claim is trivial otherwise. Since $v(x) < \infty$, the set 
of rectifiable curves connecting $A$ to $x$ is nonempty. Thus, for every $\eps>0$, there exists a 
rectifiable curve $\beta_\eps$ in $X$ connecting $A$ to $x$ such that $\int_{\beta_\eps}\rho\, ds \le 
v(x) + \eps$. Now, the concatenation $\gamma_\eps$ of $\gamma$ and $\beta_\eps$ is a rectifiable curve
in $X$ connecting $A$ to $y$, and thus,
\[
	v(y)\le \int_{\gamma_\eps}\!\rho\, \dd s
	=\int_\gamma\!\rho\, \dd s +\int_{\beta_\eps}\!\rho\, \dd s
	\le \int_\gamma\!\rho\, \dd s+v(x)+\eps.
\]
In particular, since $v(x) < \infty$ and $\int_\gamma \rho\, \dd s < \infty$, we have $v(y) < \infty$. 
Hence, by the exact same argument with $x$ and $y$ reversed, we have 
\[
	v(x) \le \int_\gamma\!\rho\, \dd s+v(y)+\eps.
\]
Thus, letting $\eps \to 0$ in the previous two estimates, the proof of \eqref{eq:common_subclaim} is 
complete.

Next, we prove the two claims of the lemma. For the claim on $v$, suppose that $\rho \le C$ and 
for every $x \in X$, there exists a rectifiable curve $\gamma_x$ connecting $x$ to $A$. Then we have 
$v(x) \le C \ell(\gamma_x) < \infty$ for every $x \in X$, and \eqref{eq:common_subclaim} thus applies, 
proving the claim for $v$.

For the claim on $u$, let $x, y \in X$, and let $\gamma$ be a rectifiable curve connecting $x$ to $y$, 
with the aim of showing that
\[
|u(x)-u(y)|\le \int_\gamma\!\rho\, \dd s.
\]
If $u(x) = u(y)$, then the above is trivially true. Hence, by symmetry, we may assume that $u(x) < 
u(y)$. In particular, since $u(y) \le k$, we have $u(x) < k$, and hence $v(x) = u(x) < \infty$. Thus, 
by \eqref{eq:common_subclaim}, we have 
\[
	\abs{u(x) - u(y)} = u(y) - v(x) \le v(y) - v(x) \le \abs{v(x) - v(y)} 
	\le \int_\gamma\!\rho\, \dd s.
\]
The claimed result for $u$ hence follows.
\end{proof}

We are now ready to prove the $MEC_p$-property of spaces of uniformly locally $p$-controlled 
geometry. 
	
\begin{lemma}\label{lem:MECp}
Let $1 \le p < \infty$, and suppose that $(X, d, \mu)$ is a space of uniformly locally 
$p$-controlled geometry that satisfies our standing assumptions. 
Then, $X$ has the $MEC_p$ property. 
\end{lemma}

We split the proof into two distinct parts. 

\begin{lemma}\label{lem:MECp-fullmeasure}
	Let $1 \le p < \infty$, and suppose that $(X, d, \mu)$ is a space of uniformly locally 
	$p$-controlled geometry that satisfies our standing assumptions. Let $0\leq\rho\in L^p(X)$ be a 
	Borel function, and let $E$ be a $\rho$-equivalence class in $X$. Then $E$ is $\mu$-measurable, and if moreover 
	$\mu(E) > 0$, then $E$ is a set of full measure.
\end{lemma}

\begin{proof}
	Since $X$ is complete and separable, and since $\mu$-measures of balls are finite,  it 
	follows from~\cite[Theorem 1.8]{JJRRS} that $E$ is $\mu$-measurable; it is in fact an analytic set, 
	but we do not require this stronger property here.  Next, suppose that $\mu(E)>0$, and 
	let $F := X \setminus E$, with the objective of showing that $\mu(F)=0$.
	
	Let $R_0$ be the threshold such that the doubling property and Poincar\'e 
	inequality of $(X, d, \mu)$ hold for balls $B(x, r)$ with $r < R_0$. Suppose then that 
	there exist $x\in X$ and $0<r<R_0$ such that $\mu(B(x,r)\cap E)>0$ and $\mu(B(x,r)\cap F)>0$, 
	with the aim of reaching a contradiction. 
	By the Borel regularity of $\mu$, we may select a closed $K \subset B(x,r)\cap E$ such 
	that $\mu(B(x,r)\cap E) < 2\mu(K)$; see e.g.\ \cite[Proposition 3.3.37]{HKSTbook}. 
	
	Fixing $k \in \N$, we define $u \colon X \to [0, k]$ by
	\[
		u(y) =\inf \left\{ k,\,\int_{\gamma}\!\rho\,\dd s : \gamma \text{ connects } y \text{ to } 
		K \right\}
	\]
	for every $y\in X$. Then, by Lemma~\ref{lem:upper-is-upper}, $u$ is measurable and $\rho$ is an
		upper gradient of $u$. As $u$ is
		bounded by $k$, it follows that $u$ is
	in $L^1_\loc(X)$. Moreover, $u \equiv 0$ on $K$, and $u \equiv k$ on $F$ since every path connecting $F$ to $K$ also connects $F$ to $E$. 
	
	Since $u \equiv k$ on $B(x, r) \cap F$ and $u \ge 0$, we have
	\[
		u_{B(x,r)} \ge  \frac{1}{\mu(B(x,r))} \int_{B(x,r) \cap F} u\, \dd\mu = k\, \frac{\mu(B(x,r)\cap F)}{\mu(B(x,r))}.
	\]
	Thus, since $u \equiv 0$ on $K$ and $\mu(B(x,r)\cap E) < 2\mu(K)$, it follows that
	\begin{align*}
	\fint_{B(x,r)} |u-u_{B(x,r)}|\, \dd\mu
	\ge \frac{1}{\mu(B(x,r))} \int_{K} |u-u_{B(x,r)}|\, \dd\mu
	& =\frac{\mu(K)}{\mu(B(x,r))}\, u_{B(x,r)}\\
	&\ge k\, \frac{\mu(B(x,r)\cap F)\, \mu(K)}{\mu(B(x,r))^2}\\
	&> k\, \frac{\mu(B(x,r)\cap F)\, \mu(B(x,r) \cap E)}{2\mu(B(x,r))^2}.
	\end{align*}
	An application of the Poincar\'e inequality, together with the fact that $r< R_0$, yields
	\[
	k\ \frac{\mu(B(x,r)\cap E)\, \mu(B(x,r)\cap F)}{2\mu(B(x,r))^2}\le C\, r\, \left(\fint_{B(x,\lambda 
		r)}\rho^p\, \dd\mu\right)^{1/p},
	\]
	with $C$ independent of $x$, $r$, and $k$. 
	This is not possible under the assumption that
	$\mu(B(x,r)\cap E)>0$ and $\mu(B(x,r)\cap F)>0$. 
	Thus we conclude that for all $x \in X$ and $0 < r < R_0$, we have either 
	$\mu(B(x,r)\cap E)=0$ or $\mu(B(x,r)\cap F)=\mu(B(x,r)\setminus E)=0$, but not both since 
	$\mu(B(x, r)) > 0$.
	
	Next, we set $U$ to be the collection of all $x\in X$ such that $\mu(B(x,r)\cap F)=0$ for 
	some $r>0$. Then $U$ is an open subset of $X$ with $\mu$-almost every $x\in E$ belonging to 
	$U$. In particular, $U$ is nonempty. We also set $W$ to be the collection of all $x\in X$ 
	such that $\mu(B(x,r)\cap E)=0$ for some $r>0$, and note that $W$ is also open. By the 
	discussion in the previous paragraph of this proof, we  also know  that $U\cap W$ is empty
	and $U\cup W=X$.
	Since $X$ is connected,  it now follows that $W$ is empty, and hence $\mu(F)=0$; that 
	is, $E$ is the main equivalence class of $\rho$ in $X$.
\end{proof}

The remaining part of the statement of Lemma \ref{lem:MECp} is the following result.
	
\begin{lemma}\label{lem:MECp-existence}
	Let $1 \le p < \infty$, let $(X, d, \mu)$ be a space of uniformly locally $p$-controlled 
	geometry that satisfies our standing assumptions, and let $0\leq\rho\in L^p(X)$ be a Borel function. 
	Then there exists a $\rho$-equivalence class $E$ with $\mu(E) > 0$.
\end{lemma}

\begin{proof}
Let $C_d$ and $R_0$ be the parameters of uniform local doubling of $X$. By Lemma~\ref{lem:bdd_geom_is_ULQC}, 
$X$ is uniformly locally quasiconvex with parameters $R$ and $C_q$, where we 
may assume that $R = R_0$ by shrinking either one of them if necessary.
Notably, $X$ is locally rectifiably connected, and since $X$ is also connected, it follows that $X$ is 
rectifiably connected.

We adapt the argument 
from~\cite[Proof of Theorem~9.3.4]{HKSTbook}. 
Thanks to the Vitali-Carath\'eodory theorem, we can assume without loss of generality that $\rho$ is lower semicontinuous.
Since $\mu$ is locally uniformly doubling, the space $(X, d, \mu)$ satisfies a restricted 
Hardy-Littlewood maximal inequality; that is, there exists a constant $C_0 = C_0(C_d)\ge 
1$ such that whenever $h\in L^1(X)$,
$\tau>0$, and $0 < r \le R_0$, we have
\begin{equation}\label{eq:HL-max}
\mu\left(\lbrace y\in X\, :\, M_{r} h(y)>\tau\rbrace\right)\le \frac{C_0}{\tau}\, \int_X|h|\, 
\dd\mu.
\end{equation}
Here, given $r > 0$ and a Borel function $h \colon X \to \overline{\R}$, the \emph{restricted 
maximal function} $M_r h$ of $h$ is defined by
\[
	M_rh(x) = \sup_{0 < r' < r}\, \fint_{B(x, r')} |h| \, \dd\mu.
\]
For details, see e.g.\ the proof of \cite[Theorem 3.5.6]{HKSTbook}.

Fix $x_0\in X$, let $R:=R_{0}/10$, and let $\beta := R/(C_q(4\lambda +1))$, where $\lambda$ is 
the scaling constant
in the local Poincar\'e inequality supported by $X$.
For every positive integer $m$,
we set $G_m:=\{y\in B(x_0,\beta)\, :\, M_{4\lambda \beta}(\rho^p)(y)\le m^p\}$. As we have that 
$M_{4\lambda \beta}(\rho^p)=M_{4\lambda \beta}(\chi_{B(x_0,R)}\, \rho^p)$ on $B(x_0,\beta)$, it follows from~\eqref{eq:HL-max} that
$\mu(B(x_0,\beta)\setminus G_m)\le C_0\, m^{-p}\, \|\rho\|_{L^p(B(x_0,R))}^p$. 
It follows that with $G := \bigcup_{m=1}^\infty G_m$, we necessarily have $\mu(G)>0$ since 
$\mu(B(x_0,\beta)\setminus G)=0$.
Thus, it suffices to show that $G\subset E$ for some $\rho$-equivalence class $E$.

We fix a point $y_0\in G$, choose $m_0\in\N$ such that $y_0\in G_{m_0}$ and consider the functions 
\[
u_k(x):=\inf_\gamma\int_\gamma(1+\min\{\rho, k\})\, \dd s,
\]
where the infimum is over all curves $\gamma$ with end points $y_0$ and $x$, and $k\in\N$. 
Since $X$ is rectifiably connected, the funtions $u_k$ are finite-valued, and we may 
additionally invoke Lemma~\ref{lem:upper-is-upper} to conclude 
that $1+\min\{\rho,k\}$ is an upper gradient of $u_k$. Moreover, if $x, y \in B(x_0, \beta)$, then 
$x$ can be connected to $y$ with a rectifiable curve $\gamma_{xy}$ with $\ell(\gamma_{xy}) \le C_q 
d(x,y)$, and consequently $|u_k(x)-u_k(y)| \le (1+k)\, \ell(\gamma_{xy}) \le C_q\, (1+k)\, d(x,y)$. 
Thus, $u$ is Lipschitz on $B(x_0, \beta)$, and consequently every point in $B(x_0, \beta)$ is a Lebesgue 
point of $u_k$. 

Applying a double-telescoping argument as in the proof of~\cite[Theorem 8.1.7]{HKSTbook}
yields the estimate
\begin{align*}
|u_k(x)&-u_k(y)|\\
  &\le C\, d(x,y)\, \left[ M_{4\lambda \beta}(1+\min\{\rho,k\})^p(x)+M_{4\lambda \beta}(1+\min\{\rho,k\})^p(y)\right]^{1/p}\\
  &\le C\, d(x,y)\, \left[ M_{4\lambda \beta}(1+\rho)^p(x)+M_{4\lambda \beta}(1+\rho)^p(y)\right]^{1/p}
\end{align*}
for every $x,y\in B(x_0,\beta)$, where $C$ depends only on constants associated to the uniformly 
locally $p$-controlled geometry of $X$. 
Note that 
\[
	M_{4\lambda \beta}(1+\rho)^p\le 2^p\, [1+M_{4\lambda \beta}(\rho^p)].
\]
Thus, whenever $x,y\in G_m$, we have that for all $k\in\N$,
\[
	|u_k(x)-u_k(y)|\le C\, d(x,y)\, \left[2\, 2^p[1+m^p]\right]^{1/p}  \le 2^{2+1/p} C\, m \, d(x,y).
\]  
By setting $y = y_0$ and noting that $y_0 \in G_m$ and $u_k(y_0) = 0$ for $m \ge m_0$, it follows 
that for each $m \ge m_0$ and $x\in G_m$, $\{u_k(x)\}_k$ is a bounded non-decreasing sequence 
of numbers. We therefore obtain a finite-valued pointwise limit map $v := \lim_{k \to \infty} u_k$ on $G 
= \bigcup_{m \ge m_0} G_m$.

For each $x\in G$, and for each $k\in\N$, we know that there exists a rectifiable curve 
$\gamma_k$ in $X$, connecting
$y_0$ to $x$, such that
\[
\infty>v(x)\ge u_k(x)\ge \int_{\gamma_k}(1+\min\{\rho,k\})\, \dd s-2^{-k}.
\]
We parametrize the curves $\gamma_k$ by arc-length, in which case they are $1$-Lipschitz. By 
the above inequality, we have for each $k$ the estimate
\[
\ell(\gamma_k)\le v(x)+2^{-k}\le v(x)+1<\infty.
\]
Thus, the sequence of curves $\{\gamma_k\}_k$ is of uniformly bounded length, and we may use the 
Arzel\`a-Ascoli theorem to find 
a subsequence $\{\gamma_{k_j}\}_j$ that converges uniformly to a curve $\gamma$ connecting $y_0$ to 
$x$. Hence, by Lemma~\ref{lem:lsc_for_path_conv} and the lower semicontinuity of $\rho$, we have 
\begin{align*}
v(x)\ge \limsup_{j \to \infty}  \int_{\gamma_{k_j}}\!(1+\min\{\rho, k_j\})\, \dd s
&\ge \limsup_{j \to \infty} \int_{\gamma_{k_j}}\!(1+\min\{\rho, n\})\, \dd s\\
&\ge \int_\gamma \! (1+\min\{\rho, n\})\, \dd s
\end{align*}
for each positive integer $n$. 
Now, by the monotone convergence theorem, we have
\[
v(x)\ge \int_\gamma(1+\rho)\, \dd s.
\]
In particular, $\gamma$ is a rectifiable curve  connecting $y_0$ to $x$ such that $\int_\gamma 
\rho\, \dd s$ is finite, and hence  $x$ is in the same $\rho$-equivalence class as $y_0$. As 
this holds for each $x\in G$, we have 
$G\subset E$, where $E$ is the $\rho$-equivalence class of $y_0$.
\end{proof}

With the previous two lemmas, the proof of Lemma \ref{lem:MECp} is hence complete.

Before starting the proof of the second part of Theorem \ref{thm:main2}, we require one more major 
lemma. For the statement of this lemma, recall that a domain $\Omega$ in a metric space $(X,d)$ is a 
nonempty, open, connected subset of $X$. As the statement involves multiple claims, we split the proof 
into multiple parts.
	
\begin{lemma}\label{lem:lsc_lemma}
	Let $1 \le p < \infty$, and let $(X,d,\mu)$ be a space of uniformly locally $p$-bounded 
	geometry that satisfies our standing assumptions. 
	Let $\Omega \subset X$ be a bounded 
	domain, and let $\rho:X\to[0,\infty]$. Suppose that $\rho$ is lower semicontinuous, and that 
	there exists a constant $c > 0$ such that $\rho \geq c$ on $\overline{\Omega}$. Define
	\[
		u(x) = \inf \left\{ \int_{\gamma}\!\rho\,\dd s : 
		\gamma \text{ connects } x \text{ to } X \setminus \Omega \right\}.
	\]
	Then, $u$ is lower semicontinuous, and for every $x \in X$  
	there exists a rectifiable path $\gamma_x$ connecting $x$ to $X \setminus \Omega$ such that
	\[
		u(x) = \int_{\gamma_{x}}\!\rho\,\dd s.
	\]
	Moreover, if $u \in L^p(X)$, 
	then $\rho$ is a $p$-weak 
	upper gradient of $u$.
\end{lemma}

We point out here that as $\mu(X)=\infty$, we necessarily have $\mu(X\setminus\Om) > 0$ when $\Om$ is bounded.

\begin{proof}[Proof of the existence of $\gamma_x$]
	Let $x \in X$. Note that by Lemma~\ref{lem:bdd_geom_is_ULQC}, $X$ is locally rectifiably 
	path-connected. Since $X$ is also connected by our standing assumptions, it follows 
	that $X$ is rectifiably path-connected. Hence, 
	there is at least one rectifiable curve connecting $x$ to $X\setminus\Om$.
	
	The claim is trivial if $u(x) = \infty$, since we may then select $\gamma_x$ to be any 
	rectifiable path that connects $x$ to $X \setminus \Omega$, in which case $\int_{\gamma_x} \rho \, 
	\dd s = \infty$ by the definition of $u(x)$. The claim is also clearly true if  $x \in X 
	\setminus \Omega$, as 
	$u(x) = 0$ for such $x$, and we can use a constant path. Thus, suppose that $x \in \Omega$ 
	with $u(x) < \infty$, 
	and let $\{\gamma_i\}_i$ be an infimizing sequence of paths connecting $x$ to $X \setminus \Omega$, 
	where we have
	\begin{equation}\label{eq:int-control}
		u(x) \leq \int_{\gamma_i}\!\rho\,\dd s \leq u(x) + \frac{1}{i}
	\end{equation}
	for each $i$. Then, since $\rho \geq c$ on $\overline{\Omega}$, we have
	\begin{equation}\label{eq:length_estimate}
		\len(\gamma_i) \leq c^{-1} \int_{\gamma_i}\! \rho\,\dd s \leq c^{-1} (u(x) + 1)<\infty. 
	\end{equation}
	We apply a parametrization to each $\gamma_i$ so that it is arclength parametrized on $[0, \len 
	(\gamma_i)]$ 
	and constant on $[\len (\gamma_i), c^{-1} (u(x) + 1)]$. Thus, $\{\gamma_i\}_i$ forms a totally 
	bounded family of $1$-Lipschitz functions. By the Arzel\`a-Ascoli theorem, we have, after passing to 
	a subsequence (not re-labeled), 
	that $\{\gamma_i\}_i$ converges uniformly to a path $\gamma_{x}$ from $x$ to $X \setminus \Omega$.
	Hence, by applying Lemma \ref{lem:lsc_for_path_conv} and the lower semicontinuity of $\rho$, we 
	have
	\[
		\int_{\gamma_{x}}\!\rho\,\dd s
		\le \limsup_{i \to \infty} \int_{\gamma_i}\!\rho\,\dd s \le u(x).
	\] 
	Thus, the claim follows.
\end{proof}

\begin{proof}[Proof of the lower semicontinuity of $u$.]
	Let $\alpha \in [0, \infty)$, and suppose that $x_i \in X$ are points such that $u(x_i) 
	\leq \alpha$ 
	and $x_i \to x \in X$. Then by \eqref{eq:length_estimate}, $\len (\gamma_{x_i})$ are uniformly 
	bounded by $c^{-1}(\alpha + 1)$, and we again, after arclength-parametrizing and moving to a 
	subsequence, have that $\gamma_{x_i} \to \gamma$ uniformly. Hence, by Lemma~\ref{lem:lsc_for_path_conv}, we have 
	\[
		u(x) \leq \int_\gamma \! \rho \,\dd s 
		\leq \limsup_{i \to \infty} \int_{\gamma_{x_i}} \! \rho \,\dd s 
		= \limsup_{i \to \infty} u(x_i)
		\leq \alpha.
	\]
	Thus, sublevel sets of $u$ are closed, implying that $u$ is lower semicontinuous.
\end{proof}

\begin{proof}[Proof that $\rho$ is a $p$-weak upper gradient of $u$.]
	We first let $x, y \in X$ 
	be points such that $u(x) < \infty$ and $u(y)< \infty$, and let $\gamma$ be a path connecting $y$ to 
	$x$. Then, by the definitions of $u$ and $\gamma_x$,
	\[
		u(y) - u(x) \leq \left(\int_{\gamma \cup \gamma_x}\!\rho\,\dd s\right) - u(x) =  
		\left(\int_\gamma\!\rho\,\dd s\right) + u(x) - u(x) = \int_\gamma\!\rho\,\dd s.
	\]
	By performing the same inequality in reverse using $\gamma_y$, we get
	\[
		\abs{u(x) - u(y)} \leq \int_\gamma\!\rho\,\dd s.
	\]
	Thus, to conclude that $\rho$ is a $p$-weak upper gradient of $u$, let $\Gamma$ be the 
	family of all non-constant compact curves in $X$ whose trajectory has non-empty intersection with 
	the infinity set $\Lambda_\infty:=\{x\in X:u(x)=\infty\}$ of $u$. Our objective now is to show 
	that $\Gamma$ has vanishing $p$-modulus. 
	
	As $X$ has the $MEC_p$ property by Lemma \ref{lem:MECp}, 
	there exists a measurable set $F\subset X$ such that $\mu(F)=0$ and for all $x\in 
	X\setminus F$ and $y \in X$, there exists a rectifiable curve 
	$\gamma_{x,y}$ connecting $x$ and $y$ such that $\int_{\gamma_{x,y}}\!\rho\,ds<\infty$ if and only 
	if $y \in X \setminus F$. 
	Note that when $x\in X\setminus F$, we must have $u(x) < \infty$, as we can find a point $w\in 
	X\setminus(\Om\cup F)$
	by the fact that $\mu(X\setminus\Om)>0$. Hence, we have $\Lambda_\infty\subset F$. 
	
	Next, denote by $\Gamma_F$ the collection of all non-constant compact curves in $X$ whose 
	trajectory has non-empty 
	intersection with $F$, and by $\Gamma^+_F$ the subcollection of those curves whose trajectory 
	intersects 
	$E$ on a set of positive one-dimensional Hausdorff measure. Note that $\Lambda_\infty\subset F$ 
	implies 
	that $\Gamma\subset\Gamma_F$. Since $\mu(F)=0$, the function $\infty\chi_{F}$ is admissible for the 
	$p$-modulus of $\Gamma^+_F$, and thus $\Mod_p(\Gamma^+_F)=0$. For 
	$\gamma\in\Gamma_F\setminus\Gamma^+_F$, 
	there must exist a point $x$ along the trajectory of $\gamma$ that lies in $F$ and another point $y$ 
	along 
	the trajectory of $\gamma$ that lies outside of $F$. It follows that $\int_{\gamma}\rho\,ds=\infty$ 
	for 
	such a curve. As $\rho\in L^p(X)$, this can only happen for a collection of curves of $p$-modulus 
	zero, 
	and hence $\Mod_p(\Gamma_F\setminus\Gamma^+_F)=0$. Therefore, $\Mod_p(\Gamma) \le 
	\Mod_p(\Gamma_F) 
	\le \Mod_p(\Gamma^{+}_F) + \Mod_p(\Gamma_F \setminus \Gamma^{+}_F) = 0$, completing the proof of 
	Lemma \ref{lem:lsc_lemma}. 
\end{proof}

We are now ready to prove a proposition which contains the essential content of the second part of 
Theorem \ref{thm:main2}. For the statement, we say that a curve $\gamma:[0,\infty)\to X$ tends to 
$\infty$ if for each $R>0$ there exists $t_R>0$ such that $\gamma(t)\in X\setminus B( \gamma(0),R)$ for 
each $t>t_R$. Moreover, in what follows, we will use the fact that since $(X, d, \mu)$ satisfies 
our standing assumptions, there exists a positive continuous function $h_0$ on $X$ such that 
$h_0\in 
L^p(X)$.  
Such a function $h_0$ can be constructed for instance via a discrete convolution of the function $h$ defined 
as follows: fixing $x_0\in X$,  we set $E_n :=  B(x_0,n)\setminus B(x_0,n-1)$ for every 
positive integer $n$, with the understanding that $B(x_0,0)$ is the empty set, and then define 
\[
h := \sum_{n=1}^\infty \frac{2^{-n/p}}{\mu(E_n)^{1/p}}\, \chi_{E_n}.
\]
As $X$ is connected and unbounded, the sets $E_n$ have nonempty interior, and we hence have 
$\mu(E_n)>0$ for every $n$. Thus, the above series is well-defined and lies in $L^p(X)$.

\begin{prop}\label{prop:pathwiseNewt}
Let $1 \le p < \infty$, and let $(X,d,\mu)$ be a metric measure space of uniformly locally 
$p$-controlled geometry that satisfies our standing assumptions.
Suppose that $X$ supports a global $(p,p)$-Sobolev inequality
for compactly supported functions in $N^{1,p}(X)$ as in Definition~\ref{def:global-pp}.
Let $f\in D^{1,p}(X)$, and suppose that
there exists a constant $c \in \R$ so that for $p$-almost every curve $\gamma \colon [0, \infty) \to X$ 
tending to $\infty$, we have $\lim\limits_{t \to \infty} f(\gamma(t)) = c$. Then $f - c \in N^{1,p}(X)$.
\end{prop} 

\begin{proof}
We select an upper gradient $\rho \colon X \to [0, \infty]$ of $f$ which has the following 
properties:
\begin{enumerate}
\item\label{enum:ug_finite_p-int} we have 
\[
\int_{X}\!\rho^p\,\dd \mu < \infty;
\]
\item\label{enum:ug_infinite_on_exc} we have 
\[
\int_{\gamma}\!\rho\,\dd s = \infty
\] 
for every path $\gamma$ tending to $\infty$ for which $\lim\limits_{t \to \infty} f(\gamma(t)) \neq c$;
\item\label{enum:ug_lsc} the function $\rho$ is lower semicontinuous.
\end{enumerate}
Indeed, \eqref{enum:ug_finite_p-int} is possible by just the definition of $D^{1,p}(X)$, and we can add 
\eqref{enum:ug_infinite_on_exc} by~\cite[Chapter I Theorem 2]{Fuglede} since the family 
of paths $\gamma$ tending to $\infty$
with $\lim_{t \to \infty} f(\gamma(t)) \neq c$ has zero $p$-modulus. For \eqref{enum:ug_lsc}, 
the Vitali-Carath\'eodory theorem, see for example \cite[p.~108]{HKSTbook}, allows us to approximate the 
previous $\rho$ from above by a lower semicontinuous function in the $L^p$-norm.
	
We fix a point $x_0 \in X$ and a positive continuous function $h_0\in L^p(X)$.
We let $\rho_0 = \rho + h_0$, 
and note that $\rho_0$ is also an upper gradient of $f$ that satisfies 
\eqref{enum:ug_finite_p-int}-\eqref{enum:ug_lsc}. For each $r>0$, we define $u_r \colon X \to 
[0, \infty]$ as follows:
\begin{equation}\label{eq:ur_def}
u_r(x):= \inf \left\{ \int_{\gamma}\!\rho_0\,\dd s : \gamma \text{ connects } x \text{ to } X 
\setminus B(x_0, r) \right\}.
\end{equation}
In particular, since $X$ is proper, every $u_r$ is compactly supported with $\supp u_r \subset \overline{B(x_0, r)}$. 
Moreover, since $\rho$ is non-negative and $h_0$ is both positive and continuous, for every $r > 0$ 
there exists some 
$\tilde{c}>0$ such that $\rho_0 \geq \tilde{c}$ on $\overline{B(x_0, r)}$. As $\rho_0$ is lower semicontinuous 
by \eqref{enum:ug_lsc}, Lemma \ref{lem:lsc_lemma} yields that $u_r$ is lower semicontinuous, $\rho_0$ is an 
upper gradient of $u_r$, and there exist paths $\gamma_{x, r}$ connecting $x$ to $X \setminus 
B(x_0, r)$ with
\begin{equation*}\label{eq:ur_representative_path}
u_r(x) = \int_{\gamma_{x, r}}\!\rho_0\,\dd s.
\end{equation*}
	
Notably, $u_r$ is compactly supported and in $D^{1,p}(X)$, so $u_r \in N^{1,p}(X)$ by the local 
$p$-Poincar\'e inequality. The global Sobolev inequality then yields 
\[
\int_{X}\! \abs{u_r}^p\,\dd\mu \leq C\, \int_{X}\! \rho_0^p\,\dd\mu < \infty.
\]
We also observe that if $x \in X$ is fixed, then $u_r(x)$ is increasing with respect to $r$, since the family 
of paths that we infimize over in the definition of $u_r(x)$ from~\eqref{eq:ur_def} shrinks as $r$ increases. Thus, 
we have a well defined limit function
\[
u \colon X \to [0, \infty], \qquad u(x) = \lim_{r \to \infty} u_r(x),
\]
and we have by the monotone convergence theorem that
\begin{equation}\label{eq:u_is_Lp}
\int_{X}\! \abs{u}^p \,\dd\mu = \lim_{r \to \infty} \int_{X}\! \abs{u_r}^p \,\dd\mu \leq C\, \int_{X}\! \rho_0^p \,\dd\mu < \infty.
\end{equation}
We claim that $\abs{f(x) - c} \leq u(x)$ for a.e.\ $x \in X$; if this holds, then $f - c \in L^p(X)$ by 
\eqref{eq:u_is_Lp}, and the proof is complete. For this, fix a point $x \in X$ such that $u(x) < \infty$; 
since $u$ is in $L^p(X)$, this holds for a.e.\ $x \in X$. 
We then consider the paths $\gamma_r:=\gamma_{x, r}:[0,\infty)\to X$ identified above, parametrized first by arclength and 
then as constant. We wish to select a sequence $r_i$ of radii tending to $\infty$ so that $\gamma_{r_i}$ 
converge locally uniformly to some path $\gamma_x$ on the domain $[0, \infty)$. Note that
\[
(r - d(x, x_0)) \leq \len (\gamma_{r}),
\]
and therefore the lengths of the paths increase to infinity as $r$ increases to infinity. Moreover, 
as $\rho_0\geq\tilde{c}$ on $\overline{B(x_0,r)}$, we have
\[
\len(\gamma_{r})\leq\frac{1}{\tilde{c}}\int_{\gamma_{r}}\!\rho_0\,ds = \frac{1}{\tilde{c}}u_r(x)\leq \frac{1}{\tilde{c}}u(x)<\infty.
\]
Thus, we may, by the Arzel\`a-Ascoli theorem, first select a sequence $\gamma_{ r(i, 1)}$ such that $r(i, 1) \geq 1 + d(x, x_0)$ 
for all $i$ and $\gamma_{r(i, 1)}$ converges uniformly to a path on $[0, 1]$. Then we take a subsequence 
$\gamma_{r(i, 2)}$ of this sequence such that $r(i, 2) \geq 2 + d(x, x_0)$ for all $i$ and we have uniform 
convergence on $[0, 2]$. We continue 
inductively; once we have the subsequence $\gamma_{r(i,k)}$ is chosen such that $r(i,k)\ge k+d(x,x_0)$ for all $i$ that 
converges uniformly in $[0,k]$, we choose a further subsequence $\gamma_{r(i,k+1)}$ with
$r(i,k+1)\ge k+1+d(x,x_0)$ for each $i$ such that this further subsequence converges uniformly in $[0,k+1]$.
Finally, we take the diagonal subsequence 
$\gamma^i := \gamma_{r(i, i)}$, which converges locally uniformly to some path $\gamma_x:[0,\infty)\to 
X$ with $\gamma_x(0)=x$.

Now, since $\rho_0$ is lower semicontinuous, we have by Lemma~\ref{lem:lsc_for_path_conv} the 
estimate
\begin{equation*}
	\int_{\gamma_x}\! \rho_0 \,\dd s 
	\leq \limsup_{i \to \infty} \int_{\gamma^i}\! \rho_0 \,\dd s
	= \limsup_{i \to \infty} u_{r(i,i)}(x)
	= u(x). 
\end{equation*}
Finally, since the integral of $\rho_0$ over $\gamma_x$ is finite by our assumption that $u(x) < 
\infty$, 
we have by 
Condition~\eqref{enum:ug_infinite_on_exc} that $f \to c$ on $\gamma_x$. Hence, because $\rho_0$ is an upper 
gradient of $f$, we have
\[
\abs{f(x) - c} \leq \int_{\gamma}\! \rho_0 \,\dd s \leq u(x).
\]
Thus, the desired $\abs{f - c} \leq u$ has been shown almost everywhere, and the proof is complete.
\end{proof}

We then derive our main result of this section from the previous proposition. 

\begin{theorem}\label{thm:lp-f-at-infty}
Let $1 \le p < \infty$, and let $(X,d,\mu)$ be a metric measure space of uniformly locally 
$p$-controlled geometry that satisfies our standing assumptions.
Suppose that $X$ supports a global $(p,p)$-Sobolev inequality
for compactly supported functions in $N^{1,p}(X)$ as in Definition~\ref{def:global-pp}.
Let $f\in D^{1,p}(X)$. Then $f\in N^{1,p}(X)$ if and only
if for $p$-almost every curve $\gamma:[0,\infty)\to X$ tending to $\infty$, we have $\lim\limits_{t\to\infty}f(\gamma(t))=0$.
\end{theorem}

\begin{remark}\label{rem:no-Sobolev}
The proof of the above theorem will show that even if $X$ does not support a global $(p,p)$-Sobolev inequality, we must have
that $\lim_{t\to\infty}f(\gamma(t))=0$ for $p$-almost every curve tending to $\infty$ whenever $f\in N^{1,p}(X)$.
\end{remark}

\begin{proof}
Thanks to Proposition~\ref{prop:pathwiseNewt}, to prove this theorem it now suffices to verify that 
for every $f \in N^{1,p}(X)$, we have $\lim_{t \to \infty} f(\gamma(t)) = 0$ along $p$-almost every 
curve $\gamma$ tending to $\infty$.
By replacing $f$ with $|f|$ if need be, we can assume without loss of generality that $f\ge 0$.
We fix such a function $f$, choose an upper gradient $g\in L^p(X)$ of $f$, and set $\Gamma$ to 
be the collection of all curves $\gamma$ in $X$ for which $\int_\gamma(f+g)\, \dd s=\infty$. Then 
we have that $\Mod_p(\Gamma)=0$.

Let $\gamma:[0,\infty)\to X$ be a curve tending to $\infty$ such that $\gamma\not\in\Gamma$. Then in particular,
$\int\limits_\gamma\, g\, ds<\infty$, and hence, for each $\eps>0$ there  exists a 
positive real number $T_\eps>0$ such that
\[
\int_{\gamma\vert_{[T_\eps,\infty)}}g\, \dd s<\eps.
\]
As $g$ is an upper gradient of $f$, it follows that for all $t,s\in [T_\eps,\infty)$ we have the 
estimate $|f(\gamma(t))-f(\gamma(s))|<\eps$. Hence, we conclude that there exists  
a $c_\gamma\ge 0$ such that $\lim_{t\to\infty} f(\gamma(t))=c_\gamma$.

To complete the proof, it remains to show that $c_\gamma=0$. Suppose towards 
contradiction that $c_\gamma>0$. Then there exists  a positive number $T$ such that
$f(\gamma(t))\ge c_\gamma/2>0$ whenever $t\ge T$. It follows that 
$\int_{\gamma\vert_{[T,\infty)}}f\, \dd s=\infty$, violating the condition that $\int_\gamma\, f\, \dd 
s$ is finite. Thus, $c_\gamma=0$, completing the proof.
\end{proof}

Based on the above proof, we  know that whenever $(X,d,\mu)$ satisfies our standing 
assumptions, $f\in D^{1,p}(X)$ and $g$ is a $p$-weak upper gradient of $f$ with $\int_Xg^p\, 
\dd\mu$ finite, we have that for $p$-modulus almost every curve $\gamma$
tending to $\infty$, the limit $c_\gamma := \lim_{t\to\infty}f(\gamma(t))$ exists. As pointed out 
in the introduction, we know that
when $X=\R^n$ and $\gamma$ is also a radial line, we can choose $c_\gamma$ to be independent of $\gamma$, see~\cite{Usp},
and if $X=\R^n\times[0,\infty)$, then for $p$-modulus almost every vertical line $\gamma$ we can choose $c_\gamma$ to be 
independent of $\gamma$, see~\cite{Feff}. This is in contrast to functions in $D^{1,p}(X)\setminus(N^{1,p}(X)+\R)$ for $X$
either $\R^n$ or $\R^n\times[0,\infty)$ for which for each $c\in\R$ there is a positive $p$-modulus worth of curves $\gamma$ tending to
$\infty$ for which $c_\gamma\ne c$. We refer the reader to~\cite{EKN, KN, KKN, LN} for more information.

\section{The case of the standard hyperbolic space}\label{Sec:6}

Given the discussion in the preceding sections, the only situation where we do not know the
relationship between $N^{1,p}(X)+\R$ and $D^{1,p}(X)$ is the case where $X$ has only one end, and that end is $p$-hyperbolic.
The focus of this section is to consider one example of this, namely, the standard Riemannian hyperbolic $n$-space $\H^n$.
In particular, as an example of the fact that $N^{1,p}(X, d, \mu) + \R$ and $D^{1,p}(X, \mu, \R)$ can sometimes coincide even in 
non-compact spaces, we prove the following. 

\begin{theorem}\label{thm:Hn_classification}
Let $1\le p <\infty$, and let $\H^n$ denote the standard hyperbolic space of dimension 
$n>1$. Then we have $N^{1,p}(\H^n) + \R = D^{1,p}(\H^n)$ if and only if $p \le n-1$.
\end{theorem}

Our definition of $\H^n$ is based on the higher-dimensional Poincar\'e ball model, as this is a 
relatively standard model for it. However, we will also consider a polar coordinate representation of 
$\H^n$ based on a warp product metric on $\S^{n-1}\times[0,\infty)$, where Euclidean rays of the form $\{\theta\} 
\times [0, \infty)$ with $\theta \in \S^{n-1}$ are mapped into $\H^n$ isometrically, but lateral 
distances between such rays increase exponentially as one approaches infinity. This latter 
representation appears to be more convenient for proving Theorem \ref{thm:Hn_classification}. 

However, using the upper gradient formalism with the model based on $\S^{n-1}\times[0,\infty)$ is not well 
suited to carrying out the proof of Theorem \ref{thm:Hn_classification}. This is because in the 
argument, it is crucial for us to be able to split the derivative of a map $u \in D^{1,p}(\H^n)$ into 
components which are parallel to $\S^{n-1}$ and $[0, \infty)$, respectively, in order to compare the 
behavior of $u$ on two different rays of the form $\{\theta\} \times [0, \infty)$, $\theta \in 
\S^{n-1}$. In considering an upper gradient of $u$, this directional information on the derivative is 
lost. For this reason, we switch from the language of analysis on metric spaces to a more Riemannian 
style of analysis in this section.  

For our precise definition of $\H^n$, we set $\H^n := (\B^n, d_{\H^n})$, where $d_{\H^n}$ is 
the Riemannian distance arising from the smooth Riemannian metric $g_{\H^n}$ defined by
\begin{equation}\label{eq:comparison_of_riem_metrics}
	\ip{v}{w}_{g_{\H^n}} = \frac{4\ip{v}{w}}{(1 - \abs{x}^2)^2}, \qquad v, w \in T_x \B^n,
\end{equation}
where $\langle v,w\rangle$ denotes the Euclidean inner product of the vectors $v$ and $w$.
As we define $\H^n$ in this way, we may use $\B^n$ and $\H^n$ interchangeably in objects and statements that only depend on the smooth structure of the space. If a definition however depends on the metric, and no explicit metric is given, then the choice of $\B^n$ or $\H^n$ signals whether the definition uses the Euclidean or hyperbolic metric.
	
The Riemannian metric $g_{\H^n}$ also yields a norm $\abs{\cdot}_{g_{\H^n}}$ on tangent vectors $v \in T_x \B^n$, 
defined by $\abs{v}_{g_{\H^n}} = \sqrt{\langle v,v \rangle_{g_{\H^n}}}$. Thus, a norm also arises for 
covectors $\alpha \in T^*_x \B^n$ 
in the dual space by 
$\abs{\alpha}_{g_{\H^n}} = \sup \bigl\{\alpha(v) : v \in T_x \B^n, \abs{v}_{g_{\H^n}} = 1 \bigr\}$. 
The comparison with the Euclidean norm for covectors is that
\begin{equation}\label{eq:norm_for_covectors}
	\abs{\alpha}_{g_{\H^n}} = \frac{1 - \abs{x}^2}{2} \abs{\alpha}, \quad \alpha \in T^*_x \B^n.
\end{equation}
Note that since $\B^n \subset \R^n$, a covector field $\omega$ on $\B^n$, also known as a differential $1$-form, has a 
representation $\omega_x =\sum_{j=1}^n\omega_j(x)\, dx_j$, and its pointwise Euclidean norm is given by 
\[
	 |\omega_x| = \sqrt{ [\omega_1(x)]^2 + \dots + [\omega_n(x)]^2}.
\]

Given a differentiable function $f \in C^1(\H^n, \R)$, there are two main ways to reason about its derivative. 
The first is to define a gradient $\nabla_{\H} f$. To achieve this, we note that if $e_1, \dots, e_n$ are the 
standard basis vectors of $\R^n$, then using the standard identification of the tangent spaces of $\R^n$ 
with $\R^n$, the vectors $v_i = 2^{-1}(1-\abs{x}^2) e_i$ form a $g_{\H^n}$-orthonormal basis of 
$T_x \B^n$
by \eqref{eq:comparison_of_riem_metrics}. Thus, the hyperbolic gradient is given by
\begin{equation}\label{eq:hyp_gradient_def}
	\nabla_{\H} f(x) = \sum_{i=1}^n (\partial_{v_i} f(x)) v_i
	= \frac{(1-\abs{x}^2)^2}{4} \sum_{i=1}^n (\partial_{e_i} f(x)) e_i
	= \frac{(1-\abs{x}^2)^2}{4} \nabla f(x).
\end{equation}

The other approach is the exterior derivative $df$, which is a covector field, the value of which at a point $x$ is given by
\begin{equation}\label{eq:df_representation}
	df_x = \partial_{e_1} f (x) dx_1 + \dots + \partial_{e_n} f(x) dx_n.
\end{equation}
The exterior derivative in particular is independent of metric; unlike how $\nabla_\H f$ and $\nabla f$ are different, 
$df$ remains the same regardless if one considers it on $\B^n$ equipped with the Euclidean metric, or $\B^n$ 
equipped with the hyperbolic metric. The three derivatives are related to each other by the fact that if $v$ is 
an $n$-dimensional vector, we have
\[
	df_x(v) = \ip{\nabla f(x)}{v} = \ip{\nabla_\H f(x)}{v}_{g_{\H^n}}.
\]
Due to \eqref{eq:comparison_of_riem_metrics}, \eqref{eq:norm_for_covectors} and \eqref{eq:hyp_gradient_def}, 
we have $\abs{\nabla_{\H} f(x)}_{g_{\H^n}} = \abs{df_x}_{g_{\H^n}}$ for all $x \in \H^n$. Tying this theory to the 
metric version, this quantity also acts as the minimal 
$p$-weak
upper gradient $\rho_f$ of $f$:
\begin{equation}\label{eq:hyperbolic_upper_gradient}
	\rho_f(x) = \abs{df_x}_{g_{\H^n}} = \abs{\nabla_{\H} f(x)}_{g_{\H^n}} = \frac{1 - \abs{x}^2}{2} \abs{\nabla f(x)}.
\end{equation}
In most of our calculations, we lean towards using $df$ over $\nabla_{\H} f$, as there is one less dependence on the metric to keep track of. 

The Riemannian metric $g_{\H^n}$ also induces a volume form $\vol_{\H^n}$ on $\B^n$. 
The comparison with the Euclidean volume form $\vol_n$ on $\B^n$ is that
\begin{equation}\label{eq:volume_forms}
	\vol_{\H^n} = \frac{2^n}{(1 - \abs{x}^2)^n} \vol_n. 
\end{equation}
Thus, we obtain $L^p$-spaces of differential $1$-forms on $\H^n$, where a 
$1$-form $\omega \in L^p(\wedge^1 T^* \H^n)$ if
\[
\int_{\B^n}\! \abs{\omega}_{g_{\H^n}}^{p} \vol_{\H^n}
= \int_{\B^n}\! \frac{(1-|x|^2)^{p-n}}{2^{p-n}} \abs{\omega}^p \vol_n
< \infty.
\]
Moreover, note that $L^1_\loc(\H^n) = L^1_\loc(\B^n)$ and $L^1_\loc(\wedge^1 T^*\H^n) = L^1_\loc(\wedge^1 T^*\B^n)$, since the hyperbolic and Euclidean metrics and measures
are locally comparable. Therefore, if $f \in L^1_\loc(\H^n)$, we say that $f$ is weakly differentiable if it is weakly differentiable in the Euclidean sense, and we use the Euclidean weak gradient $\nabla f$ to define $\nabla_\H f$ as 
in~\eqref{eq:hyp_gradient_def}, and $df$ as in \eqref{eq:df_representation}.
We can then express the Riemannian Dirichlet space on $\H^n$ as
\begin{equation}\label{eq:Riemann-kobam}
	D^{1,p}(\H^n)
	:= \left\{ f \in L^1_\loc(\H^n) : 
	f \text{ is weakly differentiable and } df \in L^p(\wedge^1 T^* \H^n)\right\}.
\end{equation}
That is, for $p \neq \infty$, a function $f \in D^{1,p}(\H^n)$ is a locally integrable function satisfying
\begin{equation}\label{eq:hyperbolic_Lp_for_df}
	\int_{\B^n}\! \abs{df}_{g_{\H^n}}^p \vol_{\H^n}
	= \int_{\B^n}\! \frac{(1-|x|^2)^{p-n}}{2^{p-n}} \abs{\nabla f(x)}^p \vol_n
	< \infty.
\end{equation}

As discussed in Remark \ref{rem:Keith}, the metric Dirichlet space $D^{1,p}(\H^n, d_{\H^n}, \vol_{\H^n})$, which is based on upper gradients, coincides 
with the Riemannian Dirichlet space $D^{1,p}(\H^n)$ defined in~\eqref{eq:Riemann-kobam}. 
Moreover, the minimal $p$-weak upper gradient $\rho_f$ of a function $f \in D^{1,p}(\H^n)$ is related to its weak derivatives by \eqref{eq:hyperbolic_upper_gradient}. 
We also define the standard Sobolev space 
\[
W^{1,p}(\H^n) := D^{1,p}(\H^n) \cap L^p(\H^n), 
\]
and observe that it similarly coincides with 
$N^{1,p}(\H^n, d_{\H^n}, \vol_{\H^n})$. 

Next, we show that according to our metric definition, $\H^n$ is indeed $p$-hyperbolic for 
each $1 \le p<\infty$. 

\begin{lemma}\label{lem:Hn-p-hyperb}
Let $p \in [1, \infty)$, and let $E\subset\S^{n-1}$ be a Borel set with $\mathcal{H}^{n-1}(E)>0$, where 
$\mathcal{H}^{n-1}$ is measured using the Euclidean
metric on the unit sphere $\S^{n-1}$. Let $0<\rho<1$ and $\Gamma(\rho E)$ be the collection of all curves in 
$\H^n$ starting from 
$\rho E\subset \rho\S^{n-1}\subset\B^n$. Then the $p$-modulus $\Mod_p(\Gamma(\rho E))$ of 
$\Gamma(\rho E)$, measured with respect to the hyperbolic
	metric on $\B^n$, is positive. Moreover, $\H^n$ is $p$-hyperbolic.	
\end{lemma}

\begin{proof}
	Let $\Gamma_L(\rho E)$ be the collection of all curves $\gamma_\theta:[\rho,1)\to\B^n$ given by $\gamma_\theta(t)=t\theta$, $\theta\in E$.
	Then $\Gamma_L(\rho E)\subset \Gamma(\rho E)$. Let $h$ be admissible, in the hyperbolic metric, for computing the $p$-modulus of 
	$\Gamma_L(\rho E)$. Then for each $\theta\in E$ we have that
	\begin{multline*}
		1 \le \int_\rho^1\! h(\gamma_\theta(t))\, \abs{\gamma_\theta'(t)}_{g_{\H^n}} \,\dd t
		= \int_\rho^1\! \frac{2h(t\theta)}{1 - t^2}\, \dd t
		= \int_\rho^1\! \frac{2^{n/p} h(t\theta)}{(1 - t^2)^{n/p}} 
			\frac{(1 - t^2)^{n/p - 1}}{2^{n/p - 1}} \, \dd t\\
		\le \left(\int_\rho^1\! \frac{h(t\theta)^p}{(1-t^2)^n}\, \dd t\right)^{\frac{1}{p}}\, \norm{t \mapsto \left(\frac{1-t^2}{2}\right)^{\frac{n-p}{p}} }_{L^{\frac{p}{p-1}}([\rho, 1))}.
	\end{multline*}
	Moreover, we have
	\[
		\norm{t \mapsto \left(\frac{1-t^2}{2}\right)^{\frac{n-p}{p}} }_{L^{\frac{p}{p-1}}([1-\rho, 1))}
		\leq C_{n, p} \, (1 - \rho)^{\frac{n-1}{p}},
	\]
	where $C_{n,p} = 1$ if $p = 1$ and $C_{n,p} = [(p-1)/(n-1)]^{(p-1)/p}$ otherwise.
	Thus, it follows that
	\[
		(1-\rho)^{1-n} \le C_{n,p}^p \int_\rho^1\!
		\frac{h(t\theta)^p}{(1-t^2)^n}\, \dd t.
	\]
	Integrating the above with respect to $\theta\in E$, we get
	\[
	\frac{\mathcal{H}^{n-1}(E)}{(1-\rho)^{n-1}} \le \frac{C_{n,p}^p}{2^n \rho^{n-1}}  
	\int_{\B^n}\! h^p\, 
	\dd\mathcal{H}^n_{\H^n}.
	\]
	Taking the infimum over all such $h$ gives
	\[
	\frac{2^n \rho^{n-1}\mathcal{H}^{n-1}(E)}{C_{n,p}^n(1-\rho)^{n-1}}
	\le \Mod_p(\Gamma_L(\rho E)) \le \Mod_p(\Gamma(\rho E)).
	\]
	This proves the first claim of the lemma. The second claim now follows by applying the above to the 
	case $E=\S^{n-1}$.
\end{proof}

\subsection{The polar coordinate model of the hyperbolic space}

While the Poincar\'e ball model of $\H^n$ is widely known and standard, it is not the most 
convenient model for us. We instead heavily employ a polar coordinate representation of 
$\H^n$, which is given by a diffeomorphism $\phi \colon (0, \infty) \times \S^{n-1} \to \B^n \setminus \{0\}$ 
that maps the lines $(0,\infty) \times \{\theta\}$ isometrically to hyperbolic geodesics. Explicitly, this map $\phi$ is given by
\begin{equation}\label{eq:phi_def}
	\phi(r, \theta) = \tanh\left(\frac{r}{2}\right) \theta
\end{equation}
for $r \in (0, \infty)$ and $\theta \in \S^{n-1}$.

We record the key features of $\phi$ in the following lemma. Note that we orient $\R \times \S^{n-1}$ 
so that $\phi$ is orientation-preserving. We also
recall from general differential geometry that, given a differentiable map $f : M \to N$ between smooth manifolds and a 
differential $k$-form $\omega$ on $N$, the pull-back $f^*\omega$ is defined by 
$(f^*\omega)_x(v_1, \dots, v_k) := \omega_{f(x)}(Df(x)v_1, \dots, Df(x)v_k)$ for all $x \in M$ and $v_1, \dots, v_k \in T_x M$
and encodes the classical chain rule.

\begin{lemma}\label{lem:polar_props}
	Let $r > 0$, $\theta \in \S^{n-1}$, and $v, w \in T_{(r, \theta)} (\R \times \S^{n-1})$, where $v$ is 
	tangential to $\R$ and $w$ is tangential to $\S^{n-1}$. Then the following results hold.
	\begin{enumerate}[label=(\roman*)]
		\item \label{enum:phi_locdiffeo} The map $\phi$ is a diffeomorphism.
		\item \label{enum:phi_radial} We have
		\[
		\abs{D\phi(r, \theta) v}_{g_{\H^n}} = \abs{v}.
		\]
		\item \label{enum:phi_tangential} We have
		\[
		\abs{D\phi(r, \theta) w}_{g_{\H^n}} = \sinh (r) \abs{w}.
		\]
		\item \label{enum:phi_perpendicular} We have
		\[
		\ip{D\phi(r, \theta) v}{D\phi(r, \theta) w}_{g_{\H^n}} = 0.
		\]
		\item \label{enum:phi_jacobian} The map $\phi$ has the hyperbolic Jacobian
		\[
			\phi^* \vol_{\H^n} = \sinh^{n-1}(r) \vol_{\R \times \S^{n-1}},
		\]
		where $\vol_{\R\times\S^{n-1}}$ is the Euclidean volume form on $\R\times\S^{n-1}$.
	\end{enumerate}
\end{lemma}

Note that when $r>0$ is large, $\sinh(r)\approx e^r$, and for small $r>0$ we have that $\sinh(r)\approx r$.

\begin{proof}
	Claim \ref{enum:phi_locdiffeo} is immediately clear, since $\phi$ is smooth and bijective. 
	Since the vector $v$ is 
	tangential to $\R$, we have $v = t e_1$ for some value of $t$, where $e_1$ 
	is the standard basis vector of $T\R$. Thus, we have
	\begin{equation}\label{eq:phi_radial}
		D\phi(r, \theta) v
		= \left(\frac{d}{dr} \tanh\left(\frac{r}{2}\right)\right) t\theta
		= \left(\frac{1}{2 \cosh^2(r/2)}\right) t\theta.
	\end{equation}
	By \eqref{eq:comparison_of_riem_metrics} and standard hyperbolic identities, it follows that
	\[
	\abs{D\phi(r, \theta) v}_{g_{\H^n}} = \frac{2}{1 - \tanh^2(r/2)} \cdot \frac{1}{2 \cosh^2(r/2)}\abs{v} = \abs{v},
	\]
	completing the proof of \ref{enum:phi_radial}. On the other hand, for the vector $w$ tangential to $\S^{n-1}$, we have
	\begin{equation}\label{eq:phi_tangential}
		D\phi(r, \theta) w 
		= \tanh\left(\frac{r}{2}\right) w.
	\end{equation}
	Thus, by \eqref{eq:comparison_of_riem_metrics} and hyperbolic identities, we get
	\[
	\abs{D\phi(r, \theta) w}_{g_{\H^n}} = \frac{2}{1 - \tanh^2(r/2)} \cdot \tanh\left(\frac{r}{2}\right) \abs{w}  = \sinh(r) \abs{w},
	\]
	proving \ref{enum:phi_tangential}.
	
	We also get \ref{enum:phi_perpendicular}, since $g_{\H^n}$ is conformal to the Euclidean metric on $\B^n$, and since $\theta$ is perpendicular to $w$. Now, $\phi$ is orientation-preserving, $D\phi$ is norm-preserving on vectors tangent to $\R$ by \ref{enum:phi_radial}, $D\phi$ scales all distances by $\sinh(r)$ on vectors tangent to $\S^{n-1}$ by \ref{enum:phi_tangential}, and $D\phi$ moreover retains the perpendicularity of these two tangent spaces by \ref{enum:phi_perpendicular}. Thus, \ref{enum:phi_jacobian} follows.
\end{proof}

With these properties, we can conveniently express Sobolev norms of hyperbolic functions. 
For this, if $f$ is function on $\R \times \S^{n-1}$, we decompose its derivative to $\R$ and 
$\S^{n-1}$ -components, denoted $df = d_{\R} f + d_{\S^{n-1}} f$; that is, if 
$v_1, v_2 \in T_{(r, \theta)} \R \times \S^{n-1}$ are such that $v_1$ is tangential to $\R$ and $v_2$ is 
tangential to $\S^{n-1}$, then $d_{\R} f_{(r, \theta)}(v_1 + v_2) = df_{(r, \theta)}(v_1)$ and  $d_{\S^{n-1}} f_{(r, \theta)}(v_1 + v_2) = df_{(r, \theta)}(v_2)$.

\begin{lemma}\label{lem:hyperbolic_chara}
	Let $u \in D^{1,p}(\H^n)$, where $p \in [1, \infty)$, and denote $f = u \circ \phi$. Then $f$ is a weakly differentiable function in $(0, \infty) \times \S^{n-1}$, and for every measurable set $E \subset \H^n$, we have
	\begin{gather*}
		\norm{u}_{L^p(E)}^p = \int_{\phi^{-1} E} \sinh^{n-1}(r) \abs{f}^p \vol_{\R \times \S^{n-1}},\\
		\norm{du}_{L^p(E)}^p \approx 
		 \int_{\phi^{-1} E} \sinh^{n-1}(r)\, \left(\left(\abs{d_{\R}u}^p + \sinh^{-p}(r) 
		 \abs{d_{\S^{n-1}}u}^p \right)\right) \vol_{\R \times \S^{n-1}},
	\end{gather*}
	where the comparison constants depend only on $p$.
\end{lemma}

\begin{proof}
	The fact that $f$ is weakly differentiable is due to a chain rule of Sobolev maps and local diffeomorphisms; 
	see e.g.\ \cite[Theorem 3.41]{Adams-Fournier_SobolevBook} for the Euclidean version, which generalizes to 
	manifolds such as $\R \times \S^{n-1}$ by use of charts and a smooth partition of unity. By the chain rule, 
	we also have $df = \phi^* du = du \circ D\phi$ almost everywhere on $\R \times \S^{n-1}$. The first identity then follows by Lemma \ref{lem:polar_props} \ref{enum:phi_jacobian} and a change of variables:
	\[
	\int_{E} \abs{u}^p \vol_{\H^n} 
	= \int_{\phi^{-1} E} \phi^* (\abs{u}^p \vol_{\H^n})
	= \int_{\phi^{-1} E} \abs{u(\phi(r, \theta))}^p\sinh^{n-1}(r) \vol_{\R \times \S^{n-1}}(r, \theta).
	\]
	
	For the second claim, by the exact same argument, we also have
	\[
	\int_{E} \abs{du}_{g_{\H^n}}^p \vol_{\H^n} 
	= \int_{\phi^{-1} E} \phi^* (\abs{du}_{g_{\H^n}}^p \vol_{\H^n})	= \int_{\phi^{-1} E} \smallabs{du_{\phi(r, \theta)}}_{g_{\H^n}}^p \sinh^{n-1}(r) \vol_{\R \times \S^{n-1}}(r, \theta).
	\]
	In a manner similar to how we decomposed $df$, we decompose $du = d_R u + d_T u$ on the image of 
	$\phi$, where if $w_1, w_2 \in T_x \B^n$ are two vectors such that $w_1$ is radial and $w_2$ is tangential, 
	then $d_R u_x (w_1 + w_2) = du_x(w_1)$ and $d_T u_x(w_1 + w_2) = du_x(w_2)$. Since the radial and tangential 
	spaces on $\B^n\setminus \{0\}$ are orthogonal to each other, and since $\H^n$ is conformal to $\B^n$, we can 
	conclude with standard linear algebra that
	\[
		\abs{du_x}_{g_{\H^n}}^p = \left( \abs{d_R u_x}_{g_{\H^n}}^2 + \abs{d_T u_x}_{g_{\H^n}}^2  
		\right)^\frac{p}{2} \approx \abs{d_R u_x}_{g_{\H^n}}^p + \abs{d_T u_x}_{g_{\H^n}}^p
	\]
	for all $x \in \B^n \setminus \{0\}$, with comparison constants depending only on $p$. Moreover, 
	since $D\phi$ maps vectors of $T\R$ to radial vectors in $T\B^n$, and maps vectors of $T\S^{n-1}$ to 
	vectors in $T\B^n$ that are 
	tangential in the polar coordinate sense, we obtain using $df = \phi^* du$ that $d_\R f = \phi^* d_R 
	u$ and $d_{\S^{n-1}} f = \phi^* d_T u$ a.e.\ on $\R \times \S^{n-1}$. Using 
	\ref{enum:phi_radial}-\ref{enum:phi_perpendicular} of Lemma \ref{lem:polar_props}, we can hence 
	reason that
	\begin{align*}
		\abs{d_{\R} f_{(r, \theta)}} 
		&= \abs{(\phi^* d_R u)_{(r, \theta)}}
		= \abs{d_R u_{\phi(r, \theta)} \circ D\phi(r, \theta)}
		= \abs{d_R u_{\phi(r, \theta)}}_{g_{\H^n}} \qquad \text{and}\\
		\abs{d_{\S^{n-1}} f_{(r, \theta)}} 
		&=  \abs{(\phi^* d_T u)_{(r, \theta)}}
		= \abs{d_T u_{\phi(r, \theta)} \circ D\phi(r, \theta)}
		= \sinh(r) \abs{d_T u_{\phi(r, \theta)}}_{g_{\H^n}}
	\end{align*}
	for a.e.\ $(r, \theta) \in \R \times \S^{n-1}$. The second claim therefore follows.
\end{proof}

Thanks to Lemma \ref{lem:hyperbolic_chara}, we can already present the simple counterexamples that yield the ``only if'' -direction of Theorem \ref{thm:Hn_classification}.

\begin{lemma}\label{lem:Hn_counterexamples}
	If $p \in (n-1, \infty]$, then $D^{1,p}(\H^n)\ne W^{1,p}(\H^n)+\R$. 
\end{lemma}
\begin{proof}
	For $p = \infty$, we use $u(x) = d_{\H^n}(x, 0)$. Then $u$ is $1$-Lipschitz from $\H^n$ to $\R$, and thus $u \in D^{1,\infty}(\H^n)$. On the other hand, 
	$u$ is unbounded, and thus also $u + c$ is unbounded for every $c \in \R$. Hence, $u \notin W^{1, \infty}(\H^n) + \R$.
	
	For $p \in (n-1, \infty)$, we select two relatively 
	open sets $U_1, U_2 \subset \S^{n-1}$ that are of positive distance from each other, and let $\psi \in C^{\infty}(\S^{n-1}, [0, 1])$ be a function that 
	is $0$ on $U_1$ and $1$ on $U_2$. We select a smooth cutoff function $\eta \in C^{\infty}([0, \infty),[0, 1])$ with $\eta \equiv 1$ on $[1, \infty)$ and 
	$\eta \equiv 0$ on a neighborhood of $0$. We then define
	\[
	f \colon \R \times \S^{n-1} \to \R, \quad f(r, \theta) = \eta(r) \psi(\theta).
	\]
	
	Next, we let $u = f \circ \phi^{-1}$ on $\H^n \setminus \{0\}$. Since $\eta(r)$ vanishes for small $r$, $u$ extends to a smooth function 
	$u \in C^\infty(\H^n, \R)$. We then apply Lemma \ref{lem:hyperbolic_chara}, and observe that
	\begin{multline*}
		\int_{[1, \infty) \times \S^{n-1}}  \left( \sinh^{n-1}(r) \abs{d_{\R}f}^p + \sinh^{n-1-p}(r) \abs{d_{\S^{n-1}}f}^p \right)\\
		= \int_{\S^{n-1}} \int_1^\infty \sinh^{n-1-p}(r) \abs{d\psi}^p(\theta) \, dr \vol_{\S^{n-1}} \theta\\
		= \left(\int_1^\infty \sinh^{-(p-(n-1))}(r) \, dr \right) \norm{d\psi}_{L^p(T^*\S^{n-1})}^p < \infty.
	\end{multline*} 
	A similar integral over $(0, 1) \times \S^{n-1}$ is also finite, since $df$ is locally bounded and vanishes near 0. Thus, $u \in D^{1,p}(\H^n)$. 
	However, we cannot have $u + c \in W^{1,p}(\H^n)$ for any $c$, since there is a set of infinite measure in $\H^n$ where $u \equiv 1$, and there 
	also is a set of infinite measure in $\H^n$ where $u \equiv 0$. Thus, $u \notin W^{1,p}(\H^n) + \R$.
\end{proof}

\subsection{Existence of traces}

We now proceed to show that if $u \in D^{1,p}(\H^n)$ with $p \in [1, \infty)$, then $u$ has a type of trace on $\S^{n-1}$. For this, let $\phi$ be 
the hyperbolic polar coordinate map from \eqref{eq:phi_def}. For convenience, we use $\phi_r$ to denote the map $\S^{n-1} \to \H^n$ defined 
by $\phi_r(\theta) = \phi(r, \theta)$, and $\phi^\theta$ to denote the map $(0, \infty) \to \H^n$ defined by $\phi^\theta(r) = \phi(r, \theta)$.

We also adopt for every $E \subset \S^{n-1}$, every $r \in (0, \infty)$, and every subset $S \subset (0, \infty)$ the notation 
$E(\H^n, r) = \phi_r (E)$ and $E(\H^n, S) = \phi(S \times E)$. For every $r > 0$, the space $\S^{n-1}(\H^n, r)$ is an $(n-1)$-dimensional 
submanifold of $\H^n$. Therefore, $g_{\H^n}$ induces a volume form $\vol_{\S^{n-1}(\H^n, r)}$ on it. The map 
$\phi_r \colon \S^{n-1} \to \S^{n-1}(\H^n, r)$ is conformal, and by the properties shown in Lemma \ref{lem:polar_props}, 
its Jacobian is given by
\begin{equation}\label{eq:spherical_volume_pullback}
	\phi_r^* \vol_{\S^{n-1}(\H^n, r)} 
	= \sinh^{n-1} (r) \vol_{\S^{n-1}}.
\end{equation} 

If $u \in L^1_\loc(\H^n)$, $E \subset \S^{n-1}$, and $r > 0$, then we denote
\[
u_{E(\H^n, r)} = \fint_{E(\H^n, r)} u(y) \vol_{\S^{n-1}(\H^n, r)}(y).
\]
Note that a standard Fubini-type argument yields that this integral is finite for a.e.\ $r \in (0, \infty)$, and changing $u$ in a set of measure zero only changes this integral in a null-set of $r \in (0, \infty)$. 

The result which allows us to define traces for $u \in D^{1,p}(\H^n)$ is the following estimate.

\begin{prop}\label{prop:Cauchy_for_averages}
	Let $u \in D^{1,p}(\H^n)$ with $p \in [1, \infty)$, and let $E \subset \S^{n-1}$ be a measurable set with $\vol_{\S^{n-1}}(E) > 0$. Then for almost all $r, s \in (0, \infty)$ with $r < s$, we have
	\begin{multline*}
		\abs{u_{E(\H^n, s)} - u_{E(\H^n, r)}}\\ 
		\leq \frac{1}{[\vol_{\S^{n-1}}(E)]^\frac{1}{p}}
		\norm{\sinh^{-\frac{n-1}{p}}}_{L^{\frac{p}{p-1}}([r, s])}
		\left( \int_{E(\H^n, [r,s])} \abs{du}_{g_{\H^n}}^p \vol_{\H^n} \right)^\frac{1}{p}.
	\end{multline*}
\end{prop}
\begin{proof}
	By changing $u$ on a set of measure zero, which only affects our statement in a null-set of $r$ and $s$, we may assume that $u$ is locally absolutely continuous on the line $r \mapsto \phi^\theta(r)$ for a.e.\ $\theta \in \S^{n-1}$. Indeed, if
	$\Omega$ is an open subset of $\R^{n-1}$, then every Sobolev function in $W^{1,p}_\loc(\R \times \Omega)$ is locally absolutely continuous on $\R \times \{x\}$ for a.e.\ $x \in \Omega$; see e.g.\ \cite[4.9.2 Theorem 2]{Evans-Gariepy-book}. The claim then follows for $u$ by utilizing diffeomorphic charts of the form $\id_\R \times \psi$ where $\psi$ is a chart on $\S^{n-1}$.
	
	We let $f = u \circ \phi$ as in Lemma \ref{lem:hyperbolic_chara}, and additionally denote $f(r, \theta) = f_r(\theta) = f^\theta(r)$. We observe by \eqref{eq:spherical_volume_pullback} and the change of variables -formula that
	\begin{equation}\label{eq:averages_match}
		u_{E(\H^n, r)} = \int_{E} \dfrac
		{\phi_r^* (u \vol_{\S^{n-1}(\H^n, r)})}
		{\sinh^{n-1}(r)  \vol_{\S^{n-1}}(E)}
		= \fint_{E} f_r \vol_{\S^{n-1}}.
	\end{equation}
	for a.e.\ $r \in (0, \infty)$. Thus, we get
	\[
	\abs{u_{E(\H^n, s)} - u_{E(\H^n, r)}}
	= \abs{
		\fint_{E} (f_s - f_r) \vol_{\S^{n-1}}
	}
	\]
	for almost all $r$ and $s$. 
	
	Recall that by our choice of $u$, the maps $f^\theta$ are locally absolutely continuous for a.e.\ $\theta \in \S^{n-1}$. Thus, for such $\theta$, we have
	\[
	\abs{f_s(\theta) - f_r(\theta)}
	= \abs{\int_r^s \frac{df^\theta(\rho)}{d\rho} \, d\rho}
	\le \int_r^s \abs{d_\R f_{(\rho, \theta)}} \, d\rho .
	\]
	We chain this estimate with an application of H\"older's inequality, which yields that
	\begin{multline*}
		\int_r^s \smallabs{d_\R f_{(\rho, \theta)}} \, d\rho
		= \int_r^s \smallabs{d_\R f_{(\rho, \theta)}} \sinh^{\frac{n-1}{p}}(\rho) \sinh^{-\frac{n-1}{p}}(\rho) \,
		d\rho\\
		\leq \norm{\sinh^{-\frac{n-1}{p}}}_{L^{\frac{p-1}{p}}([r, s])} \left( \int_r^s \smallabs{d_\R f_{(\rho, \theta)}}^p \sinh^{n-1}(\rho) \, d\rho \right)^\frac{1}{p}.
	\end{multline*}
	Thus, by integrating the resulting estimate over $\theta$ and applying H\"older's inequality a second time, we get,
	\begin{multline*}
		\abs{u_{E(\H^n, s)} - u_{E(\H^n, r)}}
		= \abs{\fint_{E} (f_s - f_r) \vol_{\S^{n-1}}}\\
		\leq 
		\fint_{E} 
		\left[ \norm{\sinh^{-\frac{n-1}{p}}}_{L^{\frac{p-1}{p}}([r, s])} \left( \int_r^s \smallabs{d_\R f_{(\rho, \theta)}}^p 	
		\sinh^{n-1}(\rho) \, d\rho \right)^\frac{1}{p} 
		\right] \vol_{\S^{n-1}}\\
		\leq \norm{\sinh^{-\frac{n-1}{p}}}_{L^{\frac{p-1}{p}}([r, s])}
		\left( \fint_{E} \int_r^s \smallabs{d_\R f_{(\rho, \theta)}}^p \sinh^{n-1}(\rho) \, d\rho \vol_{\S^{n-1}} \right)^\frac{1}{p}.
	\end{multline*}
	The claim then follows by the characterization of $\norm{du}_{L^p(E(\H^n, r, s))}$ given in Lemma \ref{lem:hyperbolic_chara}.
\end{proof}

With the aid of Proposition \ref{prop:Cauchy_for_averages}, we are able to define a form of trace measure for $u \in D^{1,p}(\H^n)$. For the definition, if $g \colon (a, \infty) \to \R$ is a map, we say that
\[
\esslim_{r \to \infty} g(r) = c \quad \text{if} \quad 
\lim_{\substack{r \to \infty \\ r \notin E}} g(r) = c,
\]
for some null-set $E \subset (a, \infty)$.

\begin{defn}\label{def:trace_measure}
	Let $u \in D^{1,p}(\H^n)$ with $p \in [1, \infty)$. We define the \emph{trace measure} $\mu_u$ of $u$ on $\S^{n-1}$ as follows: 
	\begin{equation}\label{eq:trace_measure_def}
		\mu_u(E) = \esslim_{r \to \infty} \int_{E} u \circ \phi_r \vol_{\S^{n-1}}.
	\end{equation}
	Note that if $\vol_{\S^{n-1}}(E) > 0$, then the integral on the right hand side of \eqref{eq:trace_measure_def} equals $\vol_{\S^{n-1}}(E) u_{E(\H^n, r)}$ by \eqref{eq:averages_match}. Thus, by Proposition \ref{prop:Cauchy_for_averages}, the integrals form a Cauchy net outside a null-set of radii $r$, and the limit is thus well-defined. If on the other hand $\vol_{\S^{n-1}}(E) = 0$, then the integrals are well-defined and equal zero for a.e.\ $r > 0$, implying that $\mu_{u}(E) = 0$. Thus, $\mu_u(E)$ is well defined for all measurable $E \subset \S^{n-1}$.
\end{defn}

\begin{lemma}\label{lem:trace_props}
	Let $u \in D^{1,p}(\H^n)$ with $p \in [1, \infty)$. Then $\mu_u$ is a finite signed measure on $\S^{n-1}$. Moreover, $\mu_u$ is absolutely continuous with respect to $\vol_{\S^{n-1}}$.
\end{lemma}
\begin{proof}
	We first prove that $\mu_u$ is a measure. It is immediately clear from the definition that $\mu_u(\emptyset) = 0$, so the step that requires attention is countable additivity. For this, let $E \subset \S^{n-1}$ be a disjoint union of measurable sets $E_i \subset \S^{n-1}$, where $i \in \Z_{> 0}$. For a.e.\ $r \in (0, \infty)$, we have $u \circ \phi_r \in L^1(\S^{n-1})$, and thus
	\[
	\int_{E} u \circ \phi_r \vol_{\S^{n-1}}
	= \sum_{i=1}^\infty \int_{E_i} u \circ \phi_r \vol_{\S^{n-1}}.
	\]
	Hence, the question is about whether we can perform the following exchange of limit and sum, where the left hand side is $\mu_{u}(E)$ and the right hand side is $\sum_{i=1}^\infty \mu_{u}(E_i)$:
	\begin{equation}\label{eq:lim_sum_exchange}
		\esslim_{r \to \infty} \sum_{i=1}^\infty \int_{E_i} u \circ \phi_r \vol_{\S^{n-1}}
		= \sum_{i=1}^\infty \esslim_{r \to \infty} \int_{E_i} u \circ \phi_r \vol_{\S^{n-1}}
	\end{equation}
	
	We justify this switch by next using dominated convergence with respect to the counting measure. We first select $r_0 \in (0, \infty)$ such that $u \circ \phi_{r_0} \in L^1(\S^{n-1})$, and such that for all of the sets $E_i$, the estimate of Proposition \ref{prop:Cauchy_for_averages} is valid on $E_i$ for this $r_0$ and a.e.\ $s > r_0$. Then, for all $E_i$, Proposition \ref{prop:Cauchy_for_averages} yields a dominant
	\begin{multline*}
		\abs{\int_{E_i} u \circ \phi_s \vol_{\S^{n-1}}} \le 
		\int_{E_i} \abs{u \circ \phi_{r_0}} \vol_{\S^{n-1}}\\
		+ \left[ \vol_{\S^{n-1}}(E_i) \right]^\frac{p-1}{p} 
		\norm{\sinh^{-\frac{n-1}{p}}}_{L^{\frac{p}{p-1}}([r_0, \infty))}
		\left( \int_{E_i(\H^n, [r_0, \infty))} \abs{du}_{g_{\H^n}}^p \vol_{\H^n} \right)^\frac{1}{p}
	\end{multline*}
	for a.e.\ $s > r_0$. Moreover, this dominant has a finite sum over $i$, since by using H\"older's inequality for sums, we have
	\begin{multline*}
		\sum_{i = 1}^\infty 
		\left[ \vol_{\S^{n-1}}(E_i) \right]^\frac{p-1}{p} 
		\left( \int_{E_i(\H^n, [r_0, \infty))} \abs{du}_{g_{\H^n}}^p 
		\vol_{\H^n} \right)^\frac{1}{p}\\
		\le \left( \vol_{\S^{n-1}}(E) \right)^\frac{p-1}{p} 
		\left( \int_{E(\H^n, [r_0, \infty))} \abs{du}_{g_{\H^n}}^p 
		\vol_{\H^n} \right)^\frac{1}{p} < \infty. 
	\end{multline*}
	Thus, dominated convergence applies, and \eqref{eq:lim_sum_exchange} hence follows, completing the proof that $\mu_u$ is a measure.
	
	The remaining properties are simpler in comparison. Indeed, finiteness follows from Proposition \ref{prop:Cauchy_for_averages}, since Cauchy sequences of real numbers always converge to a finite limit. Absolute continuity is then immediate from the definition, as $\mu_u(E) = 0$ whenever $\vol_{\S^{n-1}}(E) = 0$.
\end{proof}

Thus, we are led to the following type of trace functions.

\begin{defn}\label{def:trace_function}
	Let $u \in D^{1,p}(\H^n)$ with $p \in [1, \infty)$. By Lemma \ref{lem:trace_props}, $\mu_u$ is a signed measure on $\S^{n-1}$ that is absolutely continuous with respect to $\vol_{\S^{n-1}}$. Hence, $\mu_u$ has a Radon-Nikodym derivative with respect to $\vol_{\S^{n-1}}$; we call this Radon-Nikodym derivative the \emph{trace} of $u$, and denote it by $\tr(u)$. In particular, $\tr(u) \in L^1(\S^{n-1})$, and
	\[
	\int_E \tr(u) \vol_{\S^n} = \mu_u(E) = \esslim_{r \to \infty} \int_{E} u \circ \phi_r \vol_{\S^{n-1}}
	\]
	for every measurable $E \subset \S^{n-1}$.
\end{defn}

\subsection{Constant traces}

With the trace $\tr(u)$ defined for $u \in D^{1,p}(\H^n)$, our next step is to show that if $p \le n-1$, then $\tr(u)$ is constant. We reduce this fact to the following proposition.

\begin{prop}\label{prop:lateral_estimate}
	Let $u \in D^{1,p}(\H^n)$ with $p \in [1, n-1]$. Let $C_1, C_2 \subset \S^{n-1}$ be a pair of 
	isometric spherical caps. Then for a.e.\ $r > \sinh^{-1}(1)$, there exists a constant $C_p$,
	depending only on $p$, such that
	\begin{multline*}
		\abs{u_{C_1(\H^n, r)} - u_{C_2(\H^n, r)}}\\ 
		\leq \frac{C_p}{[\vol_{\S^{n-1}}(C_1)]^\frac{1}{p}}
		\norm{\sinh^{1 - \frac{n-1}{p}}}_{L^{\frac{p}{p-1}}([r, r+\pi/2])}
		\left( \int_{\S^{n-1}(\H^n, [r,r+\pi/2])} \abs{du}_{g_{\H^n}}^p \vol_{\H^n} \right)^\frac{1}{p}.
	\end{multline*}
\end{prop}

Before proving Proposition \ref{prop:lateral_estimate}, we point how the constancy of the trace follows from it.

\begin{cor}\label{cor:constant_trace}
	Let $u \in D^{1,p}(\H^n)$ with $p \in[1, n-1]$. If $C_1, C_2 \subset \S^{n-1}$ are a pair of isometric spherical caps, then we have
	\[
	\mu_u(C_1) = \mu_u(C_2).
	\]
	Consequently, $\tr(u)$ is constant.
\end{cor}
\begin{proof}
	The estimate of Proposition \ref{prop:lateral_estimate} yields that
	\[
	\esslim_{r \to \infty} \abs{u_{C_1(\H^n, r)} - u_{C_2(\H^n, r)}} = 0.
	\]
	Since also $\vol_{\S^{n-1}}(C_1) = \vol_{\S^{n-1}}(C_2)$ by isometry, we obtain the desired 
	\begin{align*}
		\mu_u(C_1) 
		= \esslim_{r \to \infty} \vol_{\S^{n-1}}(C_1) u_{C_1(\H^n, r)}
		= \esslim_{r \to \infty} \vol_{\S^{n-1}}(C_2) u_{C_2(\H^n, r)} 
		= \mu_u(C_2).
	\end{align*}
	For the claim that $\tr(u)$ is constant, the Lebesgue differentiation theorem yields for $\vol_{\S^{n-1}}$-a.e.\ $\theta \in \S^{n-1}$ that
	\[
	\tr(u) (\theta) = \lim_{r \to 0} \frac{\mu_u(\S^{n-1} \cap \B^n(\theta, r))}{\vol_{\S^{n-1}}(\S^{n-1} \cap \B^n(\theta, r))}.
	\]
	By the previous part of the Corollary, the limit on the right hand side is independent of $\theta$, thus proving that $\tr(u)$ is constant.
\end{proof}

Next, we prove Proposition \ref{prop:lateral_estimate}.

\begin{proof}[Proof of Proposition \ref{prop:lateral_estimate}]
	We again denote $f = u \circ \phi$. Let $\sigma$ be the spherical distance between the centers of $C_1$ and $C_2$, and note that $\sigma \leq \pi$. Let $R \colon [0, \sigma] \times \S^{n-1} \to \S^{n-1}$ be a rotation map, where $R(\alpha, \theta)$ rotates $\theta$ by the angle $\alpha \in [0, \sigma]$ around a given axis, where the axis and direction of rotation are chosen so that $R(\sigma, C_1) = C_2$. We also define a map $S \colon [0, \sigma] \times (0, \infty) \times \S^{n-1} \to (0, \infty) \times \S^{n-1}$ by $S(\alpha, r, \theta) = (r+\alpha, R(\alpha, \theta))$. For convenience, we also denote $R(\alpha, \theta) = R_\alpha(\theta) = R^\theta(\alpha)$ and $S(\alpha, r, \theta) = S_\alpha(r, \theta) = S^r(\alpha, \theta) = S^{r, \theta}(\alpha)$, similarly to the corresponding notation for $\phi$.
	
	Our main objective is to prove that there exists a constant $C_p' \in (0, \infty)$ depending only on 
	$p$ such that
	\begin{multline}\label{eq:twisted_corridor_estimate}
		\abs{u_{(R_{\sigma/2}C_1)(\H^n, r+\sigma/2)} - u_{C_1(\H^n, r)}}\\ 
		\le \frac{C_p'}{[\vol_{\S^{n-1}}(C_1)]^\frac{1}{p}} 
		\norm{\sinh^{1 - \frac{n-1}{p}}}_{L^{\frac{p}{p-1}}([r, r+\sigma/2])}
		\left( \int_{\S^{n-1}(\H^n, [r,r+\sigma/2])} \abs{du}_{g_{\H^n}}^p \vol_{\H^n} \right)^\frac{1}{p}.
	\end{multline}
	for a.e.\ $r \in (0, \infty)$. We start as in the proof of Proposition \ref{prop:Cauchy_for_averages}, and observe that
	\begin{multline}\label{eq:twisted_average_int_conversion}
		\abs{u_{(R_{\sigma/2}C_1)(\H^n, r+\sigma/2)} - u_{C_1(\H^n, r)}}
		= \abs{
			\fint_{C_1} \left(f \circ S^{r, \theta}(\sigma/2)  - f \circ S^{r, \theta}(0)\right) \vol_{\S^{n-1}}(\theta)
		}
	\end{multline}
	for a.e.\ $r \in (0, \infty)$. We then note that we can similarly as before select a representative of $u$ for which $f \circ S^{r, \theta}$ is locally absolutely continuous for a.e.\ $r \in (0, \infty)$ and a.e.\ $\theta \in \S^{n-1}$. Applying the fundamental theorem of calculus, the chain rule, and standard operator norm estimates, we obtain for a.e.\ $r$ and $\theta$ that
	\begin{multline*}
		\abs{f \circ S^{r, \theta}(\sigma/2) - f \circ S^{r, \theta}(0)}
		= \abs{\int_{0}^{\sigma/2} d(f \circ S^{r, \theta})} 
		\leq \int_{0}^{\sigma/2} \smallabs{df_{S^{r, \theta}(\alpha)}}\cdot  \smallabs{(S^{r, \theta})'(\alpha)} \, d\alpha.
	\end{multline*}
	
	We note that  $\smallabs{(S^{r, \theta})'(\alpha)}^2 = 1 + \smallabs{(R^\theta)'(\alpha)}^2$. Here, $0 \leq \smallabs{(R^\theta)'(\alpha)} \leq 1$, where the value is smaller the closer $\theta$ is to the axis of rotation. In particular, $1 \leq  \smallabs{(S^{r, \theta})'(\alpha)} \leq \sqrt{2}$. By combining this estimate with H\"older's inequality, we get
	\begin{multline*}
		\int_{0}^{\sigma/2} \smallabs{df_{S^{r, \theta}(\alpha)}} \cdot \smallabs{(S^{r, \theta})'(\alpha)} \, d\alpha\\
		\leq \sqrt{2} \int_{0}^{\sigma/2} \smallabs{df_{S^{r, \theta}(\alpha)}} \sinh^{\frac{n-1}{p} - 1}(r + \alpha) \sinh^{1 -\frac{n-1}{p}}(r + \alpha) \, d\alpha\\
		\leq \sqrt{2} \norm{\sinh^{1 - \frac{n-1}{p}}}_{L^{\frac{p}{p-1}}([r, r+\sigma/2])} \left( \int_{0}^{\sigma/2} \smallabs{df_{S^{r, \theta}(\alpha)}}^p \sinh^{n-1 - p}(r + \alpha) \, d\alpha\right)^\frac{1}{p}.
	\end{multline*}
	Thus, by chaining together the previous two estimates, integrating over $\theta$, and applying H\"older's inequality once more, we obtain that
	\begin{align}
		&\abs{ \fint_{C_1} \left(f \circ S^{r, \theta}(\sigma/2)  - f \circ S^{r, \theta}(0)\right) \vol_{\S^{n-1}}(\theta)}\label{eq:twisted_average_int_estimate}\\
		&\qquad\le \frac{\sqrt{2}}{\left[\vol_{\S^{n-1}}(C_1)\right]^\frac{1}{p}} \norm{\sinh^{1 - \frac{n-1}{p}}}_{L^{\frac{p}{p-1}}([r, r+\sigma/2])} \times \nonumber\\
		&\qquad\hspace{3cm}\left( \int_{[0, \sigma/2] \times C_1} \smallabs{df_{S^{r, \theta}(\alpha)}}^p \sinh^{n-1 - p}(r + \alpha) \, \vol_{\R \times \S^{n-1}}(\alpha, \theta) \right)^\frac{1}{p}.\nonumber
	\end{align}
	
	We then wish to perform a change of variables under $S^r$. For this, we note that the $\R$-component of $S^r(\alpha, \theta)$ is $r+\alpha$, which is independent of $\theta$ and satisfies $\partial_\alpha (r+\alpha) = 1$. Moreover, the $\S^{n-1}$-component of $S^r(\alpha, \theta)$ is $R(\alpha+r, \theta)$, which for fixed $\alpha$ and $r$ is an orientation-preserving isometry of $\S^{n-1}$ with respect to $\theta$. It follows that the Jacobian of $S^r$ is identically 1; that is, 
	\[
	(S^r)^* \vol_{\R \times \S^{n-1}} = \vol_{\R \times \S^{n-1}}.
	\]
	Thus, since $S^r([0, \sigma/2] \times C_1) \subset [r, r+\sigma/2] \times \S^{n-1}$ we conclude that
	\begin{multline*}
		\int_{[0, \sigma/2] \times C_1} \smallabs{df_{S^{r, \theta}(\alpha)}}^p \sinh^{n-1 - p}(r + \alpha) \, \vol_{\R \times \S^{n-1}}(\alpha, \theta)\\
		\le \int_{[r, r+\sigma/2] \times \S^{n-1}}
		\smallabs{df_{(\rho, \theta)}}^p\sinh^{n-1 - p}(\rho) \, \vol_{\R \times \S^{n-1}}(\rho, \theta).
	\end{multline*}
	Moreover, since we assumed that $r > \sinh^{-1}(1)$, we have for all $\rho \ge r$ the estimate $\sinh^{n-1-p}(\rho) \le \sinh^{n-1}(\rho)$, which in turn yields us
	\begin{align*}
		\smallabs{df_{(\rho, \theta)}}^p\sinh^{n-1 - p}(\rho)
		&= \smallabs{d_\R f_{(\rho, \theta)} + d_{\S^{n-1}} f_{(\rho, \theta)}}^p \sinh^{n-1 - p}(\rho)\\
		&\le 2^{p-1} \left( \smallabs{d_\R f_{(\rho, \theta)}}^p \sinh^{n-1}(\rho)
		+ \smallabs{d_{\S^{n-1}} f_{(\rho, \theta)}}^p \sinh^{n-1-p}(\rho) \right).
	\end{align*} 
	Thus, we may combine the previous estimates with Lemma~\ref{lem:hyperbolic_chara} to obtain that
	\begin{multline}\label{eq:twisted_avg_int_into_hyp_norm}
		\int_{[0, \sigma/2] \times C_1} \smallabs{df_{S^{r, \theta}(\alpha)}}^p \sinh^{n-1 - p}(r + \alpha) \, \vol_{\R \times \S^{n-1}}(\alpha, \theta)\\
		\lesssim \int_{\S^{n-1}(\H^n, [r, r+\sigma/2])}
		\smallabs{du}_{g_{\H^n}}^p \vol_{\H^n},
	\end{multline}
	with the comparison constant depending only on $p$.
	
	The proof of \eqref{eq:twisted_corridor_estimate} is hence completed by combining all of 
	\eqref{eq:twisted_average_int_conversion}, \eqref{eq:twisted_average_int_estimate}, and 
	\eqref{eq:twisted_avg_int_into_hyp_norm}. With \eqref{eq:twisted_corridor_estimate} shown, we then 
	perform the same estimate, but with an alternate rotation map $R'$ with the opposite direction of 
	rotation, in order to obtain an identical upper bound for $\smallabs{u_{(R_{\sigma/2}C_1)(\H^n, 
	r+\sigma/2)} - u_{C_2(\H^n, r)}}$. By combining these two upper bounds, and setting $C_p = 
	2C_p'$,  the proof of the Proposition is complete.
\end{proof}

\subsection{Limits along geodesics}

We then utilize our trace measure results from the previous section to conclude the following.

\begin{prop}\label{lem:conv_along_radii}
	Let $u \in D^{1,p}(\H^n)$ with $p \in [1, \infty)$, and denote $f = u \circ \phi$. Suppose that $\tr(u)$ is constant (up to null-set), and denote its constant value by $c$. Then there is a $L^1_\loc$-representative of $u$ such that, for $\vol_{\S^{n-1}}$-a.e.\ $\theta \in \S^{n-1}$, we have
	\[
	\lim_{r \to \infty} f(r, \theta) = c.
	\]
\end{prop}

We split the proof of Proposition~\ref{lem:conv_along_radii} into several parts. First, we show that the 
limits along a.e.\ radial path exist. Note that if $E \subset \S^{n-1}$ with $\mathcal{H}^{n-1}(E) 
> 0$, then rays of the form $\phi([1, \infty) \times \{\theta\})$ with $\theta \in E$ form a family of 
positive $p$-modulus; see Lemma~\ref{lem:Hn-p-hyperb}. Thus, our considerations here are in a sense 
similar to those at the end of Section \ref{Sec:N1p-tame}. 

\begin{lemma}\label{lem:lim_exists_along_radii}
	Let $u \in D^{1,p}(\H^n)$ with $p \in [1, \infty)$, and denote $f = u \circ \phi$. Then there is a $L^1_\loc$-representative of $u$ such that, for $\vol_{\S^{n-1}}$-a.e.\ $\theta \in \S^{n-1}$, the limit $\lim_{r \to \infty} f(r, \theta)$ exists.
\end{lemma}

\begin{proof}
	We select the representative of $u$ so that $f$ is locally absolutely continuous on $(0, \infty) \times \{\theta\}$ for a.e.\ $\theta \in \S^{n-1}$; see the beginning of the proof of Proposition \ref{prop:Cauchy_for_averages}. We then observe that for a.e.\ $\theta \in \S^{n-1}$, we have by absolute continuity that for all $r_1, r_2 \in (1, \infty)$ with $r_1 < r_2$,
	\begin{multline*}
		\abs{f(r_2, \theta) - f(r_1, \theta)}
		\leq \int_{r_1}^{r_2} \abs{d_\R f}(r, \theta) \, dr\\
		\leq \norm{\sinh^{-\frac{n-1}{p}}(r)}_{L^\frac{p}{p-1}([r_1, \infty))}
		\left( \int_{1}^\infty \sinh^{n-1}(r) \abs{d_\R f}^p(r, \theta) \, dr \right)^\frac{1}{p},
	\end{multline*} 
	where by Lemma \ref{lem:hyperbolic_chara}, the right hand side tends to $0$ at a.e.\ $\theta \in \S^{n-1}$ as $r_1 \to \infty$. Thus, $r \mapsto f(r, \theta)$ forms a Cauchy net at a.e.\ $\theta \in \S^{n-1}$, and hence, $\lim_{r \to \infty} f(r, \theta)$ exists for a.e.\ $\theta \in \S^{n-1}$.
\end{proof}

Next, we record a technical measurability result used in our argument. 
The result is closely related to various 
Scorza--Dragoni -type theorems, see e.g.\ \cite[Theorem 3.8]{Dacorogna-book}, and its proof is 
relatively standard, but we regardless outline the proof for the unfamiliar reader. 

\begin{lemma}\label{lem:lim_measurability_along_radii_general}
	Let $Z$ be a $\sigma$-finite measure space, and let $f \colon Z \times (0, \infty) \to \R$ be a measurable function such that $t \mapsto f(x, t)$ is continuous for a.e.\ $x \in Z$. Then the sets
	\begin{align*}
		\cI_f^{+}(t, r) &= \{ x \in Z : \textstyle \inf_{\rho \geq r} f(x, \rho) \ge t \},\\
		\cS_f^{-}(t, r) &= \{ x \in Z : \textstyle \sup_{\rho \geq r} f(x, \rho) \le t \},\\
		\cL\cI_f^{+}(t) &= \{ x \in Z : \textstyle \liminf_{\rho \to \infty} f(x, \rho) \ge t \},  \text{ and}\\
		\cL\cS_f^{-}(t) &= \{ x \in Z : \textstyle \limsup_{\rho \to \infty} f(x, \rho) \le t \}
	\end{align*}
	are all measurable subsets of $Z$ for all $t \in \R$ and $r > 0$.
\end{lemma}
\begin{proof}
	It suffices to prove the claims for $\cI_f^{+}(t, r)$ and $\cL\cI_f^{+}(t)$, as the supremum claims follow by applying the infimum claims to $-f$. 
	
	Note that the map $x \mapsto f(x, r)$ is measurable for a.e.\ $r \in (0, \infty)$; this fact is generally considered a part of Fubini's theorem, where it ensues that the iterated integrals used in the statement are well-defined. We denote
	\[
	Z_f^{+}(r, t) = \{ x \in Z : f(x, r) \ge t \}.
	\]
	Then for every $t \in \R$, $Z_f^{+}(r, t)$ are measurable for a.e.\ $r > 0$; we denote this full-measure set of radii by $\cR_t$. We also denote the full measure set of $x \in Z$ where the map $r \mapsto f(x, r)$ is continuous by $\cF$.
	
	Now, fix $t$. Since $\cR_t$ is of full measure, it is dense in $(0, \infty)$. Moreover, since $(0, \infty)$ is second-countable, $\cR_t$ is also second-countable, and thus separable. Hence, we may select a countable subset $\cQ_t \subset \cR_t$ that is dense in $(0, \infty)$. Now, if $x \in \cF$, then by continuity, we have $\inf_{\rho \geq r} f(x, \rho) \ge t$ if and only if $f(x, \rho) \ge t$ for all $\rho \in \cQ_t \cap [r, \infty)$. Thus, it follows that
	\[
	\cI_f^{+}(t, r) \cap \cF = \bigcap_{\rho \in \cQ_t \cap [r, \infty)} Z_f^{+}(\rho, t).
	\]
	Since the right hand side is a countable intersection of measurable sets, we conclude that $\cI_f^{+}(t, r) \cap \cF$ is measurable. And since $\cF$ is of full measure, it follows that $\cI_f^{+}(t, r)$ is measurable for all $t \in \R$. 
	
	For $\cL\cI_f^{+}(t)$, one shows that
	\begin{equation}\label{eq:limit_set_def}
		\cL\cI_f^{+}(t) = \bigcap_{k = 1}^\infty \bigcap_{j = 1}^\infty 
		\bigcup_{i = j}^\infty \cI_f^{+}(t - k^{-1}, i).
	\end{equation}
	We leave the proof of \eqref{eq:limit_set_def} to the reader, as it is a straightforward process of reasoning about membership in sets, where the only observation one needs is that $\inf_{\rho \geq r_1} f(x, \rho) \le \inf_{\rho \geq r_2} f(x, \rho)$ whenever $r_1 \le r_2$. Thus, $\cL\cI_f^{+}(t)$ is measurable, since it can be formed out of measurable sets by countable unions and intersections.
\end{proof}

We are now ready to prove Proposition~\ref{lem:conv_along_radii}

\begin{proof}[Proof of Proposition~\ref{lem:conv_along_radii}]
	We use the representative of $u$ from Lemma \ref{lem:lim_exists_along_radii} which makes $f$ locally 
	absolutely continuous on almost all lines of the form $(0, \infty) \times \{\theta\}$. Note that, in 
	particular, this means that Lemma \ref{lem:lim_measurability_along_radii_general} applies to $f$ 
	with $(Z, \mu) = (\S^{n-1}, \vol_{\S^{n-1}})$. 
	Let $L_f$ again denote the set of $\theta \in \S^{n-1}$ where $\lim_{r \to \infty} f(r, \theta)$ exists, in which case $L_f$ is a full-measure set. We then let
	\[
	A_f(c) = \{ \theta \in L_f : \lim_{r \to \infty} f(r, \theta) = c\},
	\]
	and suppose towards contradiction that $A_f(c)$ is not of full measure. Using the notation of 
	Lemma~\ref{lem:lim_measurability_along_radii_general}, we have
	\[
	L_f \setminus A_f(c) = L_f \cap \bigcup_{k=1}^\infty \left( \cL\cI_f^{+}(c + k^{-1}) \cup \cL\cS_f^{-}(c - k^{-1})\right).
	\]
	Since the sets $\cL\cI_f^{+}(t)$ and $\cL\cS_f^{-}(t)$ are measurable by Lemma \ref{lem:lim_measurability_along_radii_general}, there exists a $k \in \Z_{> 0}$ such that either $\cL\cI_f^{+}(c + k^{-1})$ or $\cL\cS_f^{-}(c - k^{-1})$ has positive measure. We restrict our discussion to the case where $\cL\cI_f^{+}(c + k^{-1})$ has positive measure, as the proof is symmetric for the other case.
	
	Let $c' = c + k^{-1} - (k+1)^{-1}$, where we note that $c' > c$. By \eqref{eq:limit_set_def}, we observe that
	\[
	\cL\cI_f^{+}(c + k^{-1}) \subset \bigcup_{i=j}^\infty \cI_f^{+}(c', j),
	\]
	where the sets $\cI_f^{+}(c', j)$ are also measurable by Lemma \ref{lem:lim_measurability_along_radii_general}. Thus, we can find a $j_0 > 0$ such that $\cI_f^{+}(c', j_0)$ has positive measure. 
	Therefore, we have that 
	$c \vol_{\S^{n-1}}(\cI_f^{+}(c', j_0))< c' \vol_{\S^{n-1}}(\cI_f^{+}(c', j_0))$.
	Now, however, we can use the assumption that $\tr(u) \equiv c$ and the definition of the set $\cI_f^{+}(c', j_0)$ to conclude that
	\begin{multline*}
		c \vol_{\S^{n-1}}(\cI_f^{+}(c', j_0))
		< c' \vol_{\S^{n-1}}(\cI_f^{+}(c', j_0))
		= \esslim_{r \to \infty} \int_{\cI_f^{+}(c', j_0)} c' \, \vol_{\S^{n-1}}\\
		\le \esslim_{r \to \infty} \int_{\cI_f^{+}(c', j_0)} f \circ \phi_r \, \vol_{\S^{n-1}}
		= \mu_u(\cI_f^{+}(c', j_0))
		= c \vol_{\S^{n-1}}(\cI_f^{+}(c', j_0)),
	\end{multline*}
	which is impossible. Thus, $A_f(c)$ has full measure, and the claim is proven. 
\end{proof}

\subsection{One-dimensional Sobolev inequality}

The last piece required to complete the proof of Theorem \ref{thm:Hn_classification} is the following one-dimensional Sobolev inequality which, beneath the technical considerations, is just a simple integration by parts followed by H\"older's inequality.

\begin{lemma}\label{lem:one_dim_exp_Sobolev}
	Let $p \in [1, \infty)$, let $\kappa > 0$, and let $f \colon (0, \infty) \to \R$ be a locally absolutely continuous function with $\lim_{t \to 0} f(t) = \lim_{t \to \infty} f(t) = 0$. Then
	\[
	\int_0^\infty \abs{f(t)}^p e^{\kappa t} \, dt 
	\leq \left( \frac{p}{\kappa} \right)^p 
	\int_0^\infty \abs{f'(t)}^p e^{\kappa t} \, dt 
	\]
\end{lemma}
\begin{proof}
	We fix $\eps > 0$ and define a function $f_\eps \colon (0, \infty) \to [0, \infty)$ by the following truncation procedure:
	\[
	f_\eps(t) = \begin{cases}
		f(t) + \eps & f(t) \le -\eps,\\
		0 & -\eps < f(t) < \eps,\\
		f(t) - \eps & f(t) \ge \eps.
	\end{cases}
	\]
	By our assumption on the limits of $f$, $f_\eps$ is compactly supported. Moreover, since we have 
	$\abs{f_\eps(t_1) - f_\eps(t_2)} \le \abs{f(t_1) - f(t_2)}$ for all $t_1, t_2 \in (0, \infty)$, and since $f$ is locally absolutely continuous, we get immediately from the $\eps$-$\delta$ -definition of absolute continuity that $f_\eps$ is locally absolutely continuous. We also observe that $\abs{f_\eps'(t)} \le \abs{f'(t)}$ for a.e.\ $t \in (0, \infty)$. 
	
	The function $s \mapsto \abs{s}^p$ on the reals is locally Lipschitz. Thus, by a chain rule of locally Lipschitz and locally absolutely continuous functions (see e.g.\ \cite[Theorem 3.59]{Leoni_SobolevBook}), $\abs{f_\eps}^p$ is locally absolutely continuous with
	\[
	\frac{d}{dt} \abs{f_\eps(t)}^p = p \abs{f_\eps(t)}^{p-2} f_\eps(t) f_\eps'(t) 
	\]
	for a.e.\ $t \in (0, \infty)$, where we interpret the right hand side as 0 whenever $f_\eps'(t) = 0$.
	
	Since $f_\eps$ is compactly supported, by using integration by parts for absolutely continuous functions (see e.g.\ \cite[Corollary 3.23]{Leoni_SobolevBook}), we thus have
	\[
	\int_0^\infty \abs{f_\eps(t)}^p e^{\kappa t} \, dt
	= -\int_0^\infty p \abs{f_\eps(t)}^{p-2} f_\eps(t) f_\eps'(t) \frac{e^{\kappa t}}{\kappa} \, dt.
	\]
	We then apply H\"older's inequality on the right hand side, obtaining that
	\[
	\int_0^\infty \abs{f_\eps(t)}^p e^{\kappa t} \, dt
	\leq \frac{p}{\kappa} 
	\left( \int_0^\infty \abs{f_\eps(t)}^p e^{\kappa t} \, dt \right)^\frac{p-1}{p}
	\left( \int_0^\infty \abs{f_\eps'(t)}^p e^{\kappa t} \, dt \right)^\frac{1}{p}.
	\]
	Since $f_\eps$ is continuous and compactly supported, the first integral on the right hand side is finite, and we may thus absorb it to the left hand side, obtaining
	\[
	\int_0^\infty \abs{f_\eps(t)}^p e^{\kappa t} \, dt
	\le \left( \frac{p}{\kappa} \right)^p \int_0^\infty \abs{f_\eps'(t)}^p e^{\kappa t} \, dt
	\le \left( \frac{p}{\kappa} \right)^p \int_0^\infty \abs{f'(t)}^p e^{\kappa t} \, dt.
	\]
	Finally, we let $\eps \to 0$ and use monotone convergence on the left hand side of the above estimate, and the claim follows.
\end{proof}

\subsection{Completion of the proof of Theorem~\ref{thm:Hn_classification}}

\begin{proof}[Proof of Theorem \ref{thm:Hn_classification}]
	If $p > n-1$, then $W^{1,p}(\H^n) + \R \ne D^{1,p}(\H^n)$ by Lemma \ref{lem:Hn_counterexamples}. Thus, for the remaining part of the theorem, we suppose that $u \in D^{1,p}(\H^n)$ with $p \in [1, n-1]$, with the objective of showing that $u - c \in D^{1,p}(\H^n)$ for some $c \in \R$. Since $u \in D^{1,p}(\H^n)$ with $p \in [1, \infty)$, the trace $\tr(u)$ exists. Moreover, since $p \le n-1$, we have by Corollary \ref{cor:constant_trace} that $\tr(u) \equiv c$ for some constant value $c \in \R$. We claim that $u - c \in L^p(\H^n)$. 
	
	For this, we first note that we may assume that $u - c$ vanishes in a unit ball $\B_{\H^n}(0, 1)$ around our chosen origin of $\H^n$. We achieve this by removing a compactly supported piece of $u$ by using a cutoff function in $C^\infty_0(\H^n)$ and noting that the removed piece has finite $L^p$-norm by a local Sobolev inequality. Following this, we denote $f = (u - c) \circ \phi$, and note that by 
	Proposition~\ref{lem:conv_along_radii}, we may assume that $\lim_{r \to \infty} f(r, \theta) = 0$ for a.e.\ $\theta \in \S^{n-1}$. 
	
	Now, we are in position to apply Lemma \ref{lem:one_dim_exp_Sobolev}, which tells us that for a.e.\ $\theta \in \S^{n-1}$, we have
	\[
	\int_0^\infty \abs{f(r, \theta)}^p e^{(n-1)r} \, dr
	\le \left( \frac{p}{n-1} \right)^p 
	\int_0^\infty \abs{\frac{d}{dr} f(r, \theta)}^p e^{(n-1)r} \, dr.
	\]
	Moreover, we have for all $r \ge 1$ that $\left(2^{-1} - e^{-2}\right) e^{r} \le \sinh(r) \le 2^{-1} e^{r}$. Thus, since $f(r, \theta)$ vanishes for $r < 1$, we can convert the previous estimate into
	\[
	\int_0^\infty \abs{f(r, \theta)}^p \sinh^{n-1}(r) \, dr
	\le \left( \frac{e^2}{e^2 - 2} \right)^{n-1} \left( \frac{p}{n-1} \right)^p 
	\int_0^\infty \abs{d_\R f_{(r, \theta)}}^p \sinh^{n-1}(r) \, dr.
	\]
	But now, by integrating this estimate over $\S^{n-1}$ and applying Lemma \ref{lem:hyperbolic_chara}, we conclude that
	\begin{multline*}
		\norm{u-c}_{L^p(\H^n)}^p 
		= \int_{(0, \infty) \times \S^{n-1}} \sinh^{n-1}(r) \abs{f(r, \theta)}^p \, dr \vol_{\S^{n-1}}\\
		\lesssim \int_{(0, \infty) \times \S^{n-1}} \sinh^{n-1}(r) \abs{d_\R f_{(r, \theta)}}^p \, dr 
		\vol_{\S^{n-1}}
		\lesssim \norm{du}_{L^p(T^* \H^n)}^p < \infty,
	\end{multline*}
	completing the proof of the claim.
\end{proof}

\end{document}